\documentclass[11pt,reqno]{amsart}
\usepackage{amsmath, amsfonts, amsthm, amssymb, color, cite, soul}

\newtheorem{Theorem}{Theorem}[section]
\newtheorem{Lemma}{Lemma}[section]
\newtheorem{Proposition}{Proposition}[section]
\newtheorem{Corollary}{Corollary}[section]

\theoremstyle{definition}
\newtheorem{Definition}{Definition}[section]

\theoremstyle{remark}
\newtheorem{Remark}{Remark}[section]

\numberwithin{equation}{section}
\allowdisplaybreaks

\def\va{\varphi}

\renewcommand{\i}{{\mathcal I}}
\renewcommand{\j}{{\mathcal J}}
\newcommand{\R}{{\mathbb R}}

\newcommand{\tr}{{\rm tr}}

\newcommand{\na}{\nabla}
\newcommand{\dl}{\delta}
\def\f{\frac}
\renewcommand{\O}{\Omega}

\def\D{\Delta }

\def\hf1{^\f{1}{1-\xi^2}}

\def\avint{\mathop{\,\rlap{--}\!\!\int}\nolimits}

\def\be{\begin{equation}}
\def\en{\end{equation}}
\def\bs{\begin{split}}
\def\es{\end{split}}
\def\ba{\begin{align}}
\def\ea{\end{align}}

\renewcommand{\a}{\alpha}

\renewcommand{\b}{\beta}
\newcommand{\del}{\partial}

\newcommand{\la}{\lambda}

\newcommand{\ve}{\varepsilon}
\newcommand{\md}{\mathrm{d}}
\newcommand{\tq}{\tilde{Q}}

\author[G.-Q. Chen]{Gui-Qiang G. Chen}
\address{Mathematical Institute, University of Oxford, Oxford OX2 6GG, UK.}
\email{chengq@maths.ox.ac.uk}

\author[A. Majumdar]{Apala Majumdar}
\address{Department of Mathematical Sciences, University of Bath, Bath,  BA2 7AY,
UK.}
\email{a.majumdar@bath.ac.uk}

\author[D. Wang]{Dehua Wang}
\address{Department of Mathematics, University of Pittsburgh,
                           Pittsburgh, PA 15260, USA.}
\email{dwang@math.pitt.edu}

\author[R. Zhang]{Rongfang Zhang}
\address{Department of Mathematics, University of Pittsburgh,
                           Pittsburgh, PA 15260, USA.}
\email{roz14@pitt.edu}

\title[Active liquid crystal system]
{Global Weak Solutions for the Compressible Active Liquid Crystal System}


\date{\today}

\begin{document}
\begin{abstract}
We study the hydrodynamics of compressible flows of active liquid crystals in the Beris-Edwards hydrodynamics framework,
using the Landau-de Gennes $Q$-tensor order parameter to describe liquid crystalline ordering.
We prove the existence of global weak solutions for this active   system in three space dimensions by the three-level approximations 
and weak convergence argument. New techniques and estimates are developed to overcome the difficulties caused by the active terms.
\end{abstract}

\maketitle
\section{Introduction}

Nematic liquid crystals are classical examples of complex liquids with long-range orientational order or anisotropic liquids
with distinguished directions of average molecular alignment \cite{G-1995, virga}.
Nematic order, often described in terms of the collective alignment of constituent elongated particles ({\it e.g.}, molecules),
is ubiquitous; we see collective motion or coordinated motion at all scales ranging from micro-organisms to traffic and flocks of animals.
The phrase {\it active hydrodynamics} is often used to describe the collective dynamics of particles that are constantly maintained
out of equilibrium by internal energy sources \cite{M-R-2006, G-M-C-H-2012}.
Active systems are quite generic in nature, including many biophysical systems such as
microtubule bundles \cite{S-C-D-H-D-2012}, dense suspensions of microswimmers \cite{W-D-H-D-G-L-Y-2012},
bacteria \cite{D-T-R-B-2010}, among others.
Furthermore, the collective oriented motion is often induced by the elongated shapes of the constituent particles,
and hence a large class of active systems are referred to as {\it active liquid crystals}, especially at high concentrations.
We refer to \cite{B-T-Y-2014, R-Y-2013, G-M-C-H-2011, G-M-C-H-2012, P-K-1992, D-E-1986,K-F-K-L-2008,C-D-2007,P-K-O-2004,M-J-R-L-P-R-A-2013} and the references cited therein
for more applications and discussions.
Active nematics are fundamentally different from the typical passive nematics in the sense that there is no notion of equilibrium; the constituent
particles continuously drive the system out of equilibrium leading to striking and novel effects such as the occurrence of
giant density fluctuations \cite{R-S-T-2003, M-R-2006, N-R-M-2007}, the spontaneous laminar flow \cite{V-J-P-2005, M-O-C-Y-2007, G-M-L-2008},
unconventional rheological properties \cite{S-A-2009, G-L-M-2010, F-M-C-2011},
low Reynolds number turbulence \cite{W-D-H-D-G-L-Y-2012, G-M-C-H-2012},
and exotic spatial and temporal patterns 
\cite{M-R-2006, C-G-M-2006, S-S-2008, G-P-B-C-2010, M-B-M-2010}.

Whilst 
active liquid crystals are popular in the theoretical physics community,
a rigorous mathematical description of {\it active nematics} is relatively new.
There are phenomenological models for active liquid crystals in \cite{R-2010},
and a common approach is to add phenomenological {\it active} terms to the hydrodynamic theories for nematic liquid crystals.
The mathematics of nematic liquid crystals has witnessed a renaissance in recent years, and there are different levels
of the mathematical description of passive nematic order: molecular variables describing the orientation and position of each molecule,
a Oseen-Frank vector field representing the unique direction of preferred molecular alignment,
and a Landau-de Gennes $\mathbf{Q}$-tensor order parameter that can describe primary and secondary directions of nematic alignment
along with variations in the degree of nematic order \cite{L-W-2014}.
More precisely, the Landau-de Gennes $\mathbf{Q}$-tensor order parameter is a $d$-dimensional symmetric and traceless matrix
for the $d$-dimensional case;
the isotropic phase is defined by 
$\mathbf{Q}=0$.
In \cite{C-M-W-Z}, we analyzed active hydrodynamics in an incompressible Beris-Edwards framework,
which is defined by two evolution equations -- an evolution equation for the velocity/flow field,
and an evolution equation for the $\mathbf{Q}$-tensor anisotropic stresses from the coupling between flow/order
and {\it active} stresses.
We established the existence of global weak solutions in two and three space dimensions
for the incompressible {\it active} Beris-Edwards system.
The mathematical machinery in \cite{C-M-W-Z} relies on the technical tools in \cite{
P-Z-2011} for the incompressible Beris-Edwards system and 
we developed new techniques to overcome additional analytical difficulties (compared to \cite{
P-Z-2011}) owing to the active stresses.

In this paper, we build on the work in \cite{
C-M-W-Z} to analyze the following system for compressible flows of active nematic
liquid crystals \cite{G-B-M-M-2013, G-M-C-H-2011} in a bounded domain $\mathcal{O}\subset \R^3$:
\be \label{Qtensor-1}
\begin{cases}
\del_{t}c+(u\cdot \na) c=D_{0}\D c,\\
\del_{t}\rho+\na \cdot (\rho u)=0,\\
\del_{t}(\rho u)+\na \cdot (\rho u\otimes u)+\na P(\rho)-\mu  \D u-(\nu +\mu)\na \mathrm{div} \,u
=\na \cdot \tau+\na \cdot \sigma,\\
\partial_{t}Q+(u\cdot \nabla )Q+Q\Omega-\Omega Q
=\Gamma H[Q,c],
\end{cases}
\en
where $c$ is the concentration of active  particles, $\rho$ is the density of the fluid,
$u\in \R^{3}$ is the flow velocity, the nematic tensor order parameter $Q$ is a traceless and symmetric $3\times 3$ matrix,
$P=\kappa \rho^{\gamma}$ denotes the pressure with adiabatic constant $\gamma >1$,
$D_0>0$ is the diffusion constant, $\mu>0$ and $\nu >0$ are  the viscosity coefficients,
$\Gamma^{-1}>0$ is the rotational viscosity,
and $\Omega=\frac{1}{2}\big(\nabla u-\nabla u^\top\big)$
is the antisymmetric part of the strain tensor.
Moreover, the  tensor:
\begin{equation*}
H[Q, c]:=K\D Q-\frac{k}{2}(c-c_{*})Q+b\big(Q^2 -\frac{\tr(Q^2)}{3} \mathrm{I}_{3}\big)-c_{*}Q\,\tr(Q^2)
\end{equation*}
describes the relaxational dynamics of the nematic phase, which can be obtained from
the Landau-de Gennes free energy, {\it i.e.}, $H_{\a \b}=-\frac{\dl \mathcal{F}}{\dl Q_{\a \b}}$ with
\begin{equation*}
\mathcal{F}=\int \Big( \frac{k}{4}(c-c_{*})\tr(Q^2) -\frac{b}{3}\tr(Q^3)
   +\frac{c_{*}}{4} |\tr(Q^{2})|^{2}+\frac{K}{2}|\na Q|^{2}\Big)dA,
\end{equation*}
where $K$ is the elastic constant for the one-constant elastic energy density,
$c_*$ is the critical concentration for the isotropic-nematic transition,
and $k>0$ and $b\in\R$ are  material-dependent constants.
Without loss of generality, we take $K=k=1$ in this paper.
The stress tensor $\sigma=(\sigma^{ij})$ has two contributions:
\begin{equation*}
\sigma^{ij}=\sigma^{ij}_{r}+\sigma^{ij}_{a},
\end{equation*}
with
\begin{align*}
\sigma^{ij}_{r}&=
Q^{ik}H^{kj}[Q, c]-H^{ik}[Q, c]Q^{kj},\\
\sigma^{ij}_{a}&=\sigma_{*} c^{2}Q^{ij},
\end{align*}
where $\sigma^{ij}_{r}$ is the stress due to the nematic elasticity,
and $\sigma^{ij}_{a}$ is the active contribution which describes contractile ($\sigma_{*}>0$)
or extensile ($\sigma_{*}<0$) stresses exerted by the active particles along the director field.
The symmetric additional stress tensor is denoted by
\begin{equation*}
\tau^{ij}=\mathrm{F}(Q)\delta_{ij}-\del_{j}Q^{kl}\del_{i}Q^{kl}=\mathrm{F}(Q)\delta_{ij}-(\na Q\odot \na Q)^{ij}
\end{equation*}
with
\begin{align*}
\mathrm{F}(Q)=\frac{1}{2}|\na Q|^2 +\frac{1}{2}\tr (Q^2)+\frac{c_{*}}{4}\tr^{2}(Q^2).
\end{align*}
Here and elsewhere, we use the Einstein summation convention, {\it i.e.}, we sum over the repeated indices.

We rewrite system (\ref{Qtensor-1}) as
\begin{align}
&\del_{t}c+(u\cdot \na) c=D_{0}\D c
,\label{c}\\
&\del_{t}\rho+\na \cdot (\rho u)=0, \label{rho}\\
&\del_{t}(\rho u)+\na \cdot (\rho u\otimes u)+\na \rho^{\gamma}
  =\mu  \D u+(\nu +\mu)\na \mathrm{div}\, u +\na \cdot \big(\mathrm{F}(Q)\mathrm{I_{3}}-\na Q\odot \na Q\big)\nonumber\\
&\quad \quad \quad \quad \quad \quad \quad \quad \quad \quad \quad \quad \quad\,\,
+\nabla \cdot (Q\D Q-\D Q Q)+\sigma_{*} \na \cdot (c^{2}Q), \label{velocity}\\
&\partial_{t}Q+(u\cdot \nabla )Q+Q\Omega-\Omega Q
=\Gamma H[Q, c],\label{Q}
\end{align}
with
$$
H[Q,c]=\D Q-\frac{c-c_{*}}{2}Q+b\Big(Q^2 -\frac{\tr(Q^2)}{3} I_{3}\Big)-c_{*}Q\tr(Q^2),$$
and $\Gamma >0, D_0>0, \mu >0, \nu>0, c_{*}>0, b, \sigma_{*} \in \mathbb{R}$, $\gamma >\frac{3}{2}$, $(x, t)\in \mathcal{O}\times \R^{+}$;
subject to the following initial conditions:
\be\label{I-C}
(c, \rho, \rho u, Q)|_{t=0}=(c_{0}, \rho_{0}, m_{0}, Q_{0})(x) \qquad \mbox{for $x\in \mathcal{O}\subset\R^3$},
\en
with
\begin{align*}
&c_{0}\in H^{1}(\mathcal{O}), \qquad\,\, 0<\underline{c}\leq c_{0} \leq \overline{c}<\infty,\\
&Q_{0}\in H^1 (\mathcal{O}), \qquad Q_{0}\in S_{0}^{3} \quad \mbox{\it a.e.}\mbox{ in } \mathcal{O},
\end{align*}
and the following boundary conditions on $\partial\mathcal{O}$ with unit outward normal $\vec{n}$:
\be\label{B-condition}
\na c \cdot \vec{n}|_{\partial \mathcal{O}}=0, \qquad u|_{\partial \mathcal{O}}=0,
\qquad \na Q \cdot \vec{n}|_{\partial \mathcal{O}}=0,
\en
satisfying the following compatibility conditions:
\be\label{compat-condition}
\rho_{0}\in L^{\gamma}(\mathcal{O}), \quad \rho_{0}\geq 0;
\qquad m_{0}\in L^1 (\mathcal{O}), \quad m_{0}=0 \,\,\, \mbox{if}\,\, \rho_{0}=0;
\qquad \frac{|m_{0}|^2}{\rho_{0}}\in L^1 (\mathcal{O}).
\en

The hydrodynamic equations
in \eqref{Qtensor-1}, or \eqref{c}--\eqref{Q},
are from \cite{
G-M-C-H-2012} with some differences, primarily for technical reasons.
Namely, in the concentration equation, the diffusion constants are assumed to
be the same in all directions,
and the active current is assumed to be zero,
which is equivalent to setting $\alpha_1=0$ in equations (15a)--(15c)
in \cite{
G-M-C-H-2012}.
Furthermore,  the flow-aligning parameter $\lambda$ in \cite{
G-M-C-H-2012} is assumed to be zero;
this is not a severe restriction, but just implies that we are in the flow-tumbling regime.
Finally, we have also neglected one of the terms in the passive ``nematic" stress which does not feature
for the two-dimensional systems but can play a role for the three-dimensional systems.
Despite these simplifications compared to the successful model presented in \cite{
G-M-C-H-2012}, our work is a first step in the rigorous analysis of initial-boundary value problems
for compressible active nematics in two and three space dimensions,
and the mathematical approach developed here is different from the previous approaches
in \cite{C-M-W-Z, W-Z} and \cite{
P-Z-2011}.

In the simplified system above, the fluid flow is dictated by the compressible Navier-Stokes equations;
the particle concentration in the fluid and the evolution of the order parameter $Q$ are governed
by the parabolic-type equations,
with extra nonlinear coupling terms as forcing terms.
The term, $\mathrm{F}(Q)$, is added to close the energy in our compressible system.
Since our system reduces to the compressible Navier-Stokes system in the absence of the concentration $c$ and the $Q$-tensor,
the best result we could expect can not be better than those in \cite{F-2001,F-N-P-2001,F-2004},
in which the existence of finite-energy weak solutions of the compressible Navier-Stokes system (allowing initial vacuum)
was proved for $\gamma>\frac{3}{2}$.
In this paper, our aim is to prove the existence of global weak solutions of this compressible coupled
system (\ref{c})$-$(\ref{compat-condition}) in three space dimensions.
In our system,
owing to the varying concentration $c=c(x, t)$,
we multiply the $Q$-tensor equation by $-\big(\D Q-Q-c_{*} Q\tr (Q^{2})\big)$,
rather than $-H[Q, c]$, to avoid dealing with the interaction terms of the concentration
and the $Q$-tensor and obtain the dissipation and {\it a priori} estimates for the system.
Moreover, the cubic term of the $Q$-tensor does not appear in the energy with this strategy,
so that a positive total energy for this system can be obtained,
unlike in \cite{C-M-W-Z} and \cite{W-Z} 
a specific positive energy for the system is re-defined by using the property of the $Q$-tensor,
{\it i.e.}, (2.5) in \cite{W-Z}.
Furthermore, the highly nonlinear terms in this system cause new mathematical difficulties
compared to \cite{C-M-W-Z}.
However, since the maximum principle holds for the concentration equation \eqref{c} for  $c$ ({\it i.e.}, $c$ is bounded
if the initial condition \eqref{I-C-c} is satisfied; see Lemma \ref{solution-Q}),
the highly nonlinear terms can be dealt with via using some cancellation rules as in Lemma \ref{estimate-matrix}
and \eqref{cancel} in Proposition \ref{energy-inequality}.
We remark that the symmetry and tracelessness of the $Q$-tensor play a key role in the cancellations which are crucial
for the proof of the existence of weak solutions.
For example, in order to obtain the essential compactness results,
the force term in the compressible Navier-Stokes equations should belong to $H^{-1}(\mathcal{O})$
due to Lions \cite{L-1996}.
However, the regularity of $Q$ obtained from the $Q$-tensor equation
is $L^{\infty}_{t}H^{1}_{x}\cap L^{2}_{t}H^{2}_{x}$, which is not enough to achieve this condition.
Owing to the cancellations, all the higher order nonlinear terms together vanish,
so we do not need to deal with them.

In this paper, we apply the Faedo-Galerkin's method \cite{Temam-book}
with three levels of approximations to prove the existence of the solutions of the initial-boundary value problem
\eqref{c}--\eqref{compat-condition}
in a bounded domain $\mathcal{O}\subset \R^{3}$.
The first level of approximation concerns the artificial pressure due to the possibility of vanishing density
and lower integrability of the density.
Here we lift the density above zero to avoid the vacuum and add the artificial pressure to increase the integrability
of the density.
The second level corresponds to the artificial viscosity, which changes the continuity equation from the hyperbolic
to parabolic type  which ensures higher regularity.
The last level is the approximation from the finite-dimensional
to infinite-dimensional space.
By the weak convergence argument,
we obtain the global existence of finite-energy weak solutions defined as follows:

\begin{Definition}
For any $T>0$, $(c, \rho, u, Q)$ is a finite-energy weak solution of problem \eqref{c}$-$\eqref{compat-condition}
if the following conditions are satisfied:
\begin{enumerate}
\item[(i)] $c>0$, $c\in L^{\infty}(0, T; L^{2}(\mathcal{O}))\cap L^{2}(0, T; H^{1}(\mathcal{O}))$; \,
$\rho \geq 0$,  $\rho \in L^{\infty}(0, T; L^{\gamma}(\mathcal{O}))$;
$u\in L^{2}(0, T; H^{1}_{0}(\mathcal{O}))$,  $Q\in L^{\infty}(0, T; H^{1}(\mathcal{O}))\cap L^{2}(0, T; H^{2}(\mathcal{O}))$, and $Q\in S^{3}_{0}$ {\it a.e.} in $\mathcal{O}_{T}= [0, T]\times \mathcal{O}$.

\smallskip
\item[(ii)] Equations (\ref{c})$-$(\ref{Q}) are valid in $\mathcal{D}^{'}(\mathcal{O}_{T})$.
Moreover, (\ref{rho}) is valid in $\mathcal{D}^{'}(0, T; \R^{3})$, if $(\rho, u)$ are extended to be zero on $\R^{3}\setminus \mathcal{O}$.

\smallskip
\item[(iii)] Energy $E(t)$ is locally integrable on $(0, T)$ and satisfies the energy inequality:
\begin{align*}
&\frac{\md}{\mathrm{d}t}E(t)+\frac{D_{0}}{2}\|\na c\|_{L^{2}}^{2} +\frac{\mu}{2} \|\nabla u\|_{L^2}^2 +(\nu +\mu)\|\mathrm{div}\, u\|_{L^2}^2+\frac{\Gamma}{2} \|\D Q\|_{L^{2}}^{2}+\frac{c_{*}^{2}\Gamma}{2} \|Q\|_{L^{6}}^{6}\\
&\,\,\leq C\big(\|u\|_{L^{2}}^{2}+\|\na Q\|_{L^{2}}^{2}+\| Q\|_{L^{2}}^{2}+\|Q\|_{L^{4}}^{4}\big)  \qquad\,\, \mbox{in $\mathcal{D}^{'}(0, T)$},
\end{align*}
where
\begin{align*}
E(t):=\int_{\mathcal{O}}\Big(\frac{1}{2}|c|^{2}+\frac{1}{2}\rho |u|^{2} +\frac{\rho^{\gamma}}{\gamma-1}
    +\frac{1}{2}|Q|^{2}+\frac{1}{2}|\na Q|^{2}+\frac{c_{*}}{4}|Q|^{4}\Big)\mathrm{d}x.
\end{align*}

\item[(iv)] Equation (\ref{rho}) is satisfied in the sense of renormalized solutions; that is, for any function $g\in C^{1}(\R)$ with the property:
\be\label{g}
g'(z)\equiv 0\qquad  \mbox{for all $z\geq M$ for a sufficiently large constant $M$},
\en
then
\be\label{renormalized}
\del_{t}g(\rho)+\mathrm{div}\,(g(\rho)u)+\big(g'(\rho)\rho -g(\rho)\big)\mathrm{div}\, u=0
\qquad\,\,  \mbox{in $\mathcal{D}^{'}(0, T)$}.
\en
\end{enumerate}
\end{Definition}

Our main result reads:

\begin{Theorem}\label{main-thm}
Let $\gamma>\frac{3}{2}$ and $\mathcal{O}\subset\R^3$ be a bounded domain of the class $C^{2+\tau}$, $\tau>0$. Assume that  the initial data function $(c_{0}, \rho_{0}, m_{0}, Q_{0})(x)$ satisfies
the compatibility conditions \eqref{compat-condition}.
Then, for any $T>0$, problem \eqref{c}$-$\eqref{B-condition} admits a finite-energy weak solution
$(c, \rho, u, Q)(t,x)$ on $\mathcal{O}_{T}$.
\end{Theorem}

We prove Theorem \ref{main-thm} by the aforementioned  three-level approximations,
including the Faedo-Galerkin approximation, artificial viscosity, artificial pressure,
as well as the weak convergence argument,
in the spirit of \cite{F-N-P-2001,F-2004}.
This approach was used to construct weak solutions to the compressible 
Beris-Edwards in \cite{W-X-Y-2015}.
As discussed above, new techniques are needed to overcome the difficulties arising from the concentration
equation and its coupling with both the fluid and Q-tensor equations.
Firstly, by using the Faedo-Galerkin approximation,
for any fixed $u_{n}$ in the finite-dimensional space $C(0, T; X_{n})$ (see \eqref{X-n}),
we obtain a unique solution $(\rho[u_{n}], c[u_{n}], Q[u_{n}])$ of
the initial-boundary value problem \eqref{c-FG}--\eqref{rho-FG} and \eqref{Q-FG}.
In Lemma \ref{solution-Q}, system \eqref{Q-system-FG} has complicated interaction terms which cause difficulties
in the proof of both the uniqueness of the solution $(c[u_{n}], Q[u_{n}])$
and the traceless property of $Q[u_{n}]$.
The tracelessness of the $Q$-tensor is an important property that guarantees the validity
of the cancellation rules in order to treat the highly nonlinear terms caused by
the appearance of the concentration and the $Q$-tensor.
In the proof of the tracelessness,
we make the $L^2$-estimate of $\tr\, Q$  from an energy inequality,
which implies the tracelessness of $Q$ when combined Gronwall's inequality and the tracelessness of the initial condition of $Q$.
As we mentioned before, the highly nonlinear interaction terms cause difficulties in this procedure.
Here we can see that the concentration equation has the maximum principle,
which provides the $L^{\infty}$--bound for the concentration.
Moreover, we show that the solutions to system \eqref{Q-system-FG}
are in $L^{\infty}(0, T; H^{1}(\mathcal{O}))\cap L^{2}(0, T; H^{2}(\mathcal{O}))$,
which provides sufficient regularity for the interaction terms
so that they stay bounded and we can prove the uniqueness of the solution and the traceless property of $Q[u_{n}]$.
Then, substituting $(\rho[u_{n}], c[u_{n}], Q[u_{n}])$  into the variational problem of the momentum equation,
we can construct a contraction map and obtain a local solution $(\rho_{n}, c_{n}, u_{n}, Q_{n})$
of the approximation system \eqref{c-FG}--\eqref{Q-FG} on the time interval $[0, T_{n}]$.
Moreover, the global existence of solutions follows from the uniform energy estimate
of the approximation system. 
In order to pass to the limit for solutions $(\rho_{n}, c_{n}, u_{n}, Q_{n})$ as $n\to \infty$
to obtain a solution $(\rho_{\ve,\dl}, c_{\ve, \dl}, u_{\ve, \dl}, Q_{\ve, \dl})$ for the approximation system
in the infinite space, we need enough integrability of the solutions.
First, the maximum principle satisfied by the concentration equation provides sufficient integrability for the concentration $c_{n}$.
This, together with the regularity of $(u_{n}, Q_{n})$ obtained in the uniform energy estimate,
gives enough compactness for the nonlinear interaction terms of 
$c_{n}$ and 
$Q_{n}$ in our system.
It actually also plays a crucial role
when the artificial viscosity and the artificial pressure tend to zero.
With  the above results and the artificial pressure and viscosity in the approximation system,
we have enough regularity and integrability of the density.
These integrability and compactness results allow us to pass to the limit as $n\to \infty$ and
obtain a solution for the approximation system in the infinite space.
Secondly, we let the artificial viscosity $\ve$ tend to zero to recover the original continuity equation.
Here we employ the convergence of the effective viscous flux sequence to deal with the lack of regularity
of the density sequence to retrieve the compactness results of the solutions.
Lastly, we pass to the limit of the vanishing artificial pressure sequence to obtain a finite-energy weak solution
of the original problem, including the vacuum case.
Again, we do not have enough integrability for the density.
Similarly to the vanishing artificial viscosity procedure, the convergence of the effective viscous flux sequence
gives us the much needed higher regularity for the density.
We remark that although the weak convergence argument is similar to that in \cite{F-N-P-2001},
owing to the extra terms and difficulties, we will provide most of the details for the sake
of completeness and convenience to the reader.

The rest of the paper is organized as follows:
In \S 2, we derive the dissipation principle and {\it a priori} estimates.
In \S 3, we employ the Faedo-Galerkin approximation to obtain a solution of the approximation problem \eqref{c-FG}--\eqref{Q-FG}.
In \S 4, we let the artificial viscosity $\ve \to 0$ to recover the solution to the hyperbolic continuity equation.
In \S 5, we pass the limit $\dl \to 0$ in the artificial pressure  to complete the proof of Theorem \ref{main-thm}.
In Appendix \ref{appendix-a}, we 
list some important lemmas and state important preliminaries that we use intensively in the paper.

\section{The Dissipation Principle and A Priori Estimates}\label{a-priori-estimates}

In this section,
we derive the dissipation principle and {\it a priori} estimates in Proposition \ref{energy-inequality} 
for system (\ref{c})$-$(\ref{Q}).

For the sake of convenience,  we first introduce some notations.
We denote the Sobolev space by $H^{k}(\mathcal{O})$ for integer $k\geq 1$, equipped with norm $\|\cdot\|_{H^{k}_{x}}$,
and $H^{-k}(\mathcal{O})$ is the dual space of $H^{k}_{0}(\mathcal{O})$.
Denote by $(\cdot, \cdot)$ the inner product in $L^2 (\mathcal{O})$:
If $a$ and $b$ are vectors, then
\begin{equation*}
(a, b)=\int_{\mathcal{O}}a(x)\cdot b(x)\,\md x;
\end{equation*}
and if $A, B$ are matrices, then
\begin{equation*}
(A, B)=\int_{\mathcal{O}}A:B\, \md x =\int_{\mathcal{O}}\tr (AB)\,\md x
\end{equation*}
with $A:B=\tr (AB)$.
Denote by $S_{0}^{3}\subset \mathbb{M}^{3\times 3}$ the space of $Q$-tensors in $\R^3$; that is,
\begin{equation*}
S_{0}^{3}:=\left\{Q\in \mathbb{M}^{3\times 3}: \;  Q^{ij}=Q^{ji}, \ \tr(Q)=0, \ i, j=1,2,3\right\}.
\end{equation*}
Define the norm of a matrix by using the Frobenius norm  denoted by
\begin{equation*}
|Q|^{2}:=\tr(Q^{2})=Q^{ij}Q^{ij}.
\end{equation*}
With respect to this norm, we define the Sobolev spaces for the $Q$-tensors:
\begin{equation*}
H^{1}(\mathcal{O}, S_{0}^{3}):=\left\{Q: \mathcal{O}\rightarrow S_{0}^{3}\, : \  \int_{\mathcal{O}}\big(|\na Q(x)|^{2}+|Q(x)|^{2}\big)\,dx <\infty \right\}.
\end{equation*}
Set $|\na Q|^2:=\del_{k}Q^{ij}\del_{k}Q^{ij}$ and $|\D Q|^2 :=\D Q^{ij}\D Q^{ij}$.

\begin{Proposition} \label{energy-inequality}
Let $(c, \rho, u, Q)$ be a smooth solution of  problem \eqref{rho}--\eqref{compat-condition}.
Then there exists $C>0$ depending only on $(D_{0}, b, c_{*},\sigma_{*}, \mu, \nu, \Gamma)$ and the initial data
such that, for a given $T$,
\be\label{dissipative}
\begin{split}
&\frac{\md}{\mathrm{d}t}E(t)+\frac{D_{0}}{2}\|\na c\|_{L^{2}}^{2} +\frac{\mu}{2} \|\nabla u\|_{L^2}^2 +(\nu +\mu)\|\mathrm{div}\, u\|_{L^2}^2+\frac{\Gamma}{2} \|\D Q\|_{L^{2}}^{2}+\frac{c_{*}^{2}\Gamma}{2} \|Q\|_{L^{6}}^{6}\\
&\leq C\big(\|u\|_{L^{2}}^{2}+\|\na Q\|_{L^{2}}^{2}+\| Q\|_{L^{2}}^{2}+\|Q\|_{L^{4}}^{4}\big) \qquad \,\,\,\mbox{for all $t\in (0, T)$}.
\end{split}
\en
In addition,  if $(c_{0}, \rho_{0}, m_{0}, Q_{0})(x)\in L^{\infty}\times L^{\gamma}\times L^{2}\times H^{1}$, then
\begin{align}
&\|c\|_{L^{\infty}(0, T; L^\infty(\mathcal{O}))}+\|c\|_{L^{\infty}(0, T; L^{2}(\mathcal{O}))\cap L^{2}(0, T; H^{1}(\mathcal{O}))}\leq C,\\
&\|\rho\|_{L^{\infty}(0, T; L^{\gamma}(\mathcal{O}))}\leq C,\\
&\|\sqrt{\rho}u\|_{L^{\infty}(0,
T; L^{2}(\mathcal{O}))}^2 +\frac{\mu }{2}\|\na u\|_{L^{2}(0, T; L^{2}(\mathcal{O}))}^2 +(\nu +\mu)\|\mathrm{div}\, u\|_{L^{2}(0, T; L^{2}(\mathcal{O}))}^2 \leq C,\\
&\|Q\|_{L^{\infty}(0, T; H^{1}(\mathcal{O})\cap L^{4}(\mathcal{O}))\cap L^{2}(0, T; H^{2}(\mathcal{O}))\cap L^{6}(0, T; L^{6}(\mathcal{O}))}\leq C.
\label{Q-apriori-estimate}
\end{align}
\end{Proposition}

\begin{proof}
The $L^\infty$--bound for the concentration $c(t,x)$ follows from the maximum principle and its initial condition $c_0\in L^\infty$.
Using the continuity equation (\ref{rho}) and the boundary equation (\ref{B-condition}),
we have
\begin{align*}
&((\rho u)_{t}, u)-(\rho u\otimes u, \na u)\\
&=\int_{\mathcal{O}}\rho_{t}|u|^{2}\,\md x+\int_{\mathcal{O}}\rho \del_{t}u^{i}u^{i}\,\md x
  -\int_{\mathcal{O}}\rho u^{i}u^{j}\del_{j}u^{i}\,\md x\\
&=\int_{\mathcal{O}}\rho_{t}|u|^2 \,\md x +\frac{1}{2}\int_{\mathcal{O}}\rho \del_{t}|u|^{2}\,\md x
  +\frac{1}{2}\int_{\mathcal{O}}|u|^{2}\mathrm{div}\, (\rho u)\,\md x\\
&=\int_{\mathcal{O}}\rho_{t}|u|^2 \,\md x +\frac{1}{2}\frac{\mathrm{d}}{\md t}\int_{\mathcal{O}}\rho |u|^{2}\,\md x
  -\frac{1}{2}\int_{\mathcal{O}}\rho_{t}|u|^{2}\,\md x +\frac{1}{2}\int_{\mathcal{O}}|u|^{2}\mathrm{div}\, (\rho u)\,\md x\\
&=\frac{1}{2}\frac{\mathrm{d}}{\md t}\int_{\mathcal{O}}\rho |u|^{2}\,\md x
  +\frac{1}{2}\int_{\mathcal{O}}|u|^{2}\big(\rho_{t}+\mathrm{div}\, (\rho u)\big)\,\md x\\
& =\frac{1}{2}\frac{\mathrm{d}}{\md t}\int_{\mathcal{O}}\rho |u|^{2}\,\md x,
\end{align*}
and
\begin{align*}
(\rho^{\gamma}, \mathrm{div}\, u)
& =-\int_{\mathcal{O}}\rho^{\gamma-1}(\rho_{t}+\na \rho \cdot u)\,\md x\\
&=-\frac{1}{\gamma}\frac{\md}{\md t}\int_{\mathcal{O}}\rho^{\gamma}\,\md x
 -\frac{1}{\gamma}\int_{\mathcal{O}}\del_{i}\rho^{\gamma}u^{i}\,\md x
 =-\frac{1}{\gamma}\frac{\md}{\md t}\int_{\mathcal{O}}\rho^{\gamma}\,\md x +\frac{1}{\gamma}\int_{\mathcal{O}}\rho^{\gamma}\mathrm{div}\, u\,\md x,
\end{align*}
which gives
\begin{align*}
(\rho^{\gamma}, \mathrm{div}\, u)=\frac{1}{1-\gamma}\frac{\mathrm{d}}{\mathrm{d}t}\int_{\mathcal{O}}\rho^{\gamma}\,\mathrm{d} x.
\end{align*}
We take the inner product of equation \eqref{c} with $c$,
equation \eqref{velocity} with $u$, and equation \eqref{Q} with $-\big(\D Q-Q-c_{*} Q\,\tr (Q^{2})\big)$ respectively,
sum them up, and then integrate by parts over $\mathcal{O}$ to obtain
\begin{equation*}
\begin{split}
&\frac{\md}{\mathrm{d}t}\int_{\mathcal{O}}\Big(\frac{1}{2}|(c, \sqrt{\rho} u)|^{2}
+\frac{\rho^{\gamma}}{\gamma-1} +\frac{1}{2}|(Q, \na Q)|^{2}
+\frac{c_{*}}{4}|Q|^{4} \Big)\md x
 +D_{0}\|\na c\|_{L^{2}}^{2} \\
&\quad +\mu \|\nabla u\|_{L^2}^2 +(\nu +\mu)\|\mathrm{div}\, u\|_{L^2}^2+\Gamma \|\D Q\|_{L^{2}}^{2}+\Gamma \|\na Q\|_{L^{2}}^{2}
 +c_{*}\Gamma \|Q\|_{L^{4}}^{4}+c_{*}^{2}\Gamma \|Q\|_{L^{6}}^{6}\\[2mm]
&=-(u\cdot \na c, c)
-(\na \cdot (\nabla Q \odot \nabla Q), u)-(\mathrm{F}(Q)\mathrm{I}_{3}, \na u) + (\na \cdot (Q\D Q-\D QQ), u)\\
&\quad -\sigma_{*}(c^{2}Q, \na u)+(u\cdot \nabla Q, \D Q)-(u\cdot \na Q, Q+c_{*}Q|Q|^{2})-(\Omega Q-Q\Omega, \D Q) \\
&\quad +\left(\O Q-Q\O, Q +c_{*}Q|Q|^{2}\right)+\Gamma(\frac{c-c_{*}}{2}Q, \D Q-Q-c_{*} Q\tr (Q^{2}))\\
&\quad -b\Gamma (Q^{2}, \D Q) + b\Gamma (Q^{2}, Q+c_{*} Q\tr (Q^{2}))+2c_{*}\Gamma(Q|Q|^{2}, \D Q)\\
&=\sum_{i=1}^{13}\mathcal{I}_{i}.
\end{split}
\end{equation*}
First, by Lemma \ref{estimate-matrix}, we have
$$
\mathcal{I}_{4} +\mathcal{I}_{8} =0.
$$
By simple calculation,
$$
\i_{2}+\mathcal{I}_{3} +\i_{6}+\mathcal{I}_{7} =0,  \qquad \i_{9}=0,
$$
as shown below:
\be\label{cancel}
\begin{split}
&\i_{2}+\i_{6} +\i_{7}\\
&=-(\na \cdot (\na Q \odot \na Q), u)+(u \cdot \na Q, \D Q)-(u\cdot \na Q, Q+c_{*}Q|Q|^{2})\\
&=-\int_{\mathcal{O}} \del_{j}\big(\del_{i}Q^{kl}\del_{j}Q^{kl}\big)u^{i}\,\mathrm{d}x
  +\int_{\mathcal{O}} u^{i}\del_{i}Q^{kl}\D Q^{kl}\,\mathrm{d}x
   -\int_{\mathcal{O}}u^{i}\del_{i}Q^{kl}\big(Q^{kl}+c_{*}Q^{kl}|Q|^2\big)\,\md x \\
&=-\int_{\mathcal{O}} \big(\del_{i}\del_{j}Q^{kl}\del_{j}Q^{kl}u^{i}+\del_{i}Q^{kl}\del_{j}\del_{j} Q^{kl}u^{i}\big)\,\mathrm{d}x
+\int_{\mathcal{O}} u^{i}\del_{i}Q^{kl}\D Q^{kl}\,\mathrm{d}x \\
&\,\quad-\int_{\mathcal{O}}u^{i}\del_{i}\big(\frac{1}{2}|Q|^2 +\frac{c_{*}}{4}\tr^{2}(Q^2)\big)\,\md x \\
&=-\int_{\mathcal{O}} \del_{i}\del_{j}Q^{kl}\del_{j}Q^{kl}u^{i}\,\mathrm{d}x
+\int_{\mathcal{O}}\big(\frac{1}{2}|Q|^2 
+\frac{c_{*}}{4}\tr^{2}(Q^2)\big)\mathrm{div}\, u \,\md x\\
&=\int_{\mathcal{O}} \big(\frac{1}{2}|\na Q|^{2}+\frac{1}{2}|Q|^2 
+\frac{c_{*}}{4}\tr^{2}(Q^2)\big)\mathrm{div}\, u \,\md x =-\i_{3},
\end{split}
\en
and
\begin{align*}
\i_{9}&=(\O Q-Q\O, Q+c_{*}Q|Q|^{2})=-(\O Q+Q\O, Q+c_{*}Q|Q|^{2})+2(\O Q, Q+c_{*}Q|Q|^{2})=0,
\end{align*}
where we have used the fact that $Q$ is symmetric and $\Omega$ is skew-symmetric.

Moreover, by the Young inequality, we have
\begin{align*}
|\i_{1}|&=|(u \cdot \na  c,
c)|\leq C\|u\|_{L^{2}}\|\na c\|_{L^{2}}\|c\|_{L^{\infty}}\leq \frac{D_{0}}{4}\|\na c\|_{L^{2}}^{2}+C\|u\|_{L^{2}}^{2},\\
|\i_{5}|&=|\sigma_{*}(c^{2}Q, \na u)| \leq  C\|c\|_{L^{\infty}}^{2}\|\na u\|_{L^{2}}\| Q\|_{L^{2}}\leq \frac{\mu}{4}\|\na u\|_{L^{2}}^{2}+C\| Q\|_{L^{2}}^{2},\\
|\i_{10}|&=\Gamma\big|(\frac{c-c_{*}}{2}Q, \D Q-Q-c_{*} Q\,\tr (Q^{2}))\big|
  \leq \frac{\Gamma}{8}\|\D Q\|_{L^{2}}^{2}+C\|Q\|_{L^{2}}^{2}+C\|Q\|_{L^{4}}^{4},\\
|\i_{11}|&=| -b\Gamma (Q^{2}, \D Q)|\leq \frac{\Gamma}{8}\|\D Q\|_{L^{2}}^{2}+C\|Q\|_{L^{4}}^{4},\\
|\i_{12}|&=|b\Gamma (Q^{2}, Q+c_{*} Q\,\tr (Q^{2}))|=\big|b\Gamma \int_{\mathcal{O}}\big(\tr(Q)\big)^3\,\mathrm{d}x +bc_{*}\Gamma \int_{\mathcal{O}}\big(\tr(Q)\big)^3 |Q|^2 \,\mathrm{d}x\big|\\
&\leq \frac{c^2 \Gamma}{2} \|Q\|_{L^6}^{6}+C\|Q\|_{L^{2}}^{2} +C\|Q\|_{L^4}^{4},\\
\i_{13}&=2c_{*}\Gamma(Q|Q|^{2}, \D Q)=2c_{*}\Gamma \int_{\mathcal{O}}\del_{kk} Q^{ij}Q^{ij}\tr(Q^2)\, \mathrm{d}x\\
&=-2c_{*}\Gamma \int_{\mathcal{O}}\del_{k} Q^{ij}\del_{k}Q^{ij}\tr(Q^2)\, \mathrm{d}x
  -2c_{*}\Gamma \int_{\mathcal{O}}\del_{k} Q^{ij}Q^{ij}\del_{k}\tr(Q^2)\, \mathrm{d}x\\
&=-2c_{*}\Gamma \int_{\mathcal{O}}|\na Q|^{2}|Q|^2 \,\mathrm{d}x
 -c_{*}\Gamma \int_{\mathcal{O}}|\na \tr(Q^2)|^2\, \mathrm{d}x\leq 0.
\end{align*}

Combining all the above estimates, we obtain the desired result.
\end{proof}

\begin{Corollary}\label{Q-estimate-addition}
For any smooth solution $(c, \rho, u, Q)$ of problem \eqref{c}--\eqref{compat-condition},
\begin{align*}
Q \in L^{10}(\mathcal{O}_{T})\cap L^{\infty}(0, T; H^{1}(\mathcal{O}))\cap L^{2}(0, T; H^{2}(\mathcal{O})), \quad \na Q\in L^{\frac{10}{3}}(\mathcal{O}_{T}).
\end{align*}
\end{Corollary}

\begin{proof} From estimate (\ref{Q-apriori-estimate}) and the Gagliardo-Nirenberg inequality in Lemma \ref{gn-inequality}, we have
\begin{eqnarray*}
&& \|Q\|_{L^{10} (\mathcal{O})}\leq C\|Q\|_{L^{6}(\mathcal{O})}^{\frac{4}{5}}\|\D Q\|_{L^{2}(\mathcal{O})}^{\frac{1}{5}}+C\|Q\|_{L^{6}(\mathcal{O})},\\
&& \|\na Q\|_{L^{\frac{10}{3}}(\mathcal{O})}\leq C\|\na Q\|_{L^{2}(\mathcal{O})}^{\frac{2}{5}}\|\D Q\|_{L^{2}(\mathcal{O})}^{\frac{3}{5}}+C\|\na Q\|_{L^{2}(\mathcal{O})}.
\end{eqnarray*}
Then the proof is complete.
\end{proof}

\section{The Faedo-Galerkin Approximation}\label{weak-solutions}

In this section, for fixed $\dl>0$ and $\ve >0$, we solve the following approximation problem:
\begin{align}
&\del_{t}c+(u\cdot \na) c=D_{0}\D c, \label{c-FG}\\
&\del_{t}\rho+\na \cdot (\rho u)=\ve \D \rho,\label{rho-FG}\\
&\del_{t}(\rho u)+\na \cdot (\rho u\otimes u)+\na \rho^{\gamma}+\delta \na \rho^{\b} +\ve (\na \rho \cdot \na)u\nonumber\\
&\quad \quad \quad \quad =\mu \D u+(\nu +\mu)\na \mathrm{div}\, u
    +\na \cdot (\mathrm{F}(Q)\mathrm{I}_{3}-\na Q\odot \na Q)\nonumber\\
& \quad \quad \quad \quad\quad +\nabla \cdot (Q\D Q-\D Q Q)+\sigma_{*} \na \cdot (c^{2}Q),\label{u-FG}\\
&\partial_{t}Q+(u\cdot \nabla )Q+Q\Omega-\Omega Q
=\Gamma H[Q, c],\label{Q-FG}
\end{align}
complemented with the modified initial conditions:
\begin{align}
&c|_{t=0}=c_{0}\in H^{1}(\mathcal{O}), \qquad 0<\underline{c}\leq c_{0}(x)\leq \bar{c},\label{I-C-c}\\
&\rho|_{t=0}=\rho_{0}\in C^{3}(\bar{\mathcal{O}}), \qquad 0<\underline{\varrho}\leq \rho_{0}(x)\leq \bar{\varrho},\label{I-C-rho-a}\\
&(\rho u)|_{t=0}=m_{0}(x)\in C^{2}(\bar{\mathcal{O}}), \label{I-C-rho-b}\\
&Q|_{t=0}=Q_{0}(x)\in H^{1}(\mathcal{O}), \qquad Q_{0}\in S^{3}_{0} \,\,\, \mbox{{\it a.e.} in } \mathcal{O},\label{I-C-Q}
\end{align}
and the boundary conditions on $\del \mathcal{O}$ with unit outward normal $\vec{n}$:
\begin{align}
&\na c \cdot \vec{n}|_{\del \mathcal{O}}=0, \qquad\,\, \na \rho \cdot \vec{n}|_{\del \mathcal{O}}=0,\label{B-C-rho}\\
& u|_{\del \mathcal{O}}=0, \qquad\qquad\,\,  \frac{\del Q}{\del \vec{n}}|_{\del \mathcal{O}}=0,\label{B-C-uQ}
\end{align}
where $\underline{c}$, $\bar{c}$, $\underline{\varrho}$, and $\bar{\varrho}$ are positive constants.

\begin{Remark}
We now remark 
on the extra terms in the approximation system.
The vanishing viscosity $\ve \D \rho$ converts equation \eqref{rho}
from the hyperbolic to parabolic type and provides the higher regularity of $\rho$.
The term, $\dl \na \rho^{\b}$, is added to obtain the higher integrability of $\rho$ for some constant $\b>0$.
The extra term $\ve \na \rho \cdot \na u$ is needed to cancel some {\it bad} terms  that do not vanish in the energy estimates.
\end{Remark}

\subsection{The Neumann problem for the density and $Q$-tensor}

We first state the following existence results, which can be found in \cite{F-N-P-2001}.

\begin{Lemma}\label{solution-rho}
Assume that $u \in C([0, T]; C^{2}(\bar{\mathcal{O}}, \R^{3}))$ with $u|_{\del \mathcal{O}}=0$.
Then there exists the following mapping $S=S[u]${\rm :}
\begin{equation*}
S: C([0, T]; C^{2}(\bar{\mathcal{O}}, \R^{3})) \rightarrow C([0, T]; C^{3}(\bar{\mathcal{O}}))
\end{equation*}
such that
\begin{enumerate}
\item[(i)] $\rho =S[u]$ is the unique classical solution of \eqref{rho-FG}, \eqref{I-C-rho-a}--\eqref{I-C-rho-b}, and \eqref{B-C-rho}{\rm ,}
\item[(ii)] For all $t\geq 0$,
 $$
 \underline{\varrho}\exp\Big(-\int_{0}^{t}\|\mathrm{div}\, u(s)\|_{L^{\infty}(\mathcal{O})}\md s\Big)
   \leq S[u](t, x)\leq \bar{\varrho}\exp\Big(\int_{0}^{t}\|\mathrm{div}\, u(s)\|_{L^{\infty}(\mathcal{O})}\md s\Big),
 $$
\item[(iii)] For any $u_{1}$ and $u_{2}$ in the set{\rm :}
\begin{equation}\label{Nset}
\mathcal{N}_{N}=\{v \,:\, v\in C([0, T]; C^{2}_{0}(\mathcal{O}, \R^{3})),\, \|v\|_{C([0, T]; C^{2}_{0}(\bar{\mathcal{O}}, \R^3 ))}\leq N\}
\end{equation}
with some suitable constant $N>0$,
\be\label{continous-S}
\|S[u_{1}]-S[u_{2}]\|_{C([0, T]; H^{1}(\mathcal{O}))}\leq TC(N, T)\|u_{1}-u_{2}\|_{C([0, T]; H^{1}_{0}(\mathcal{O}))}.
\en
\end{enumerate}
\end{Lemma}

\begin{Lemma}\label{solution-Q}
For each $u\in C([0, T]; C_{0}^{2}(\bar{\mathcal{O}}, \R^3 ))$ with $u|_{\del \mathcal{O}}=0$,
there exists a unique solution of the following initial-boundary value problem{\rm :}
\be\label{Q-system-FG}
\begin{cases}
\del_{t}c+(u\cdot \na) c=D_{0}\D c,\\[1mm]
\partial_{t}Q+(u\cdot \nabla )Q+Q\Omega-\Omega Q
=\Gamma H[Q, c],\\[1mm]
Q|_{t=0}=Q_{0}(x)\in H^{1}(\mathcal{O}), \qquad \na Q \cdot \vec{n}|_{\del \mathcal{O}}=0,\\[1mm]
c|_{t=0}=c_{0}(x)\in H^{1}(\mathcal{O}), \qquad\,\,\,\,\,  \na c\cdot  \vec{n}|_{\del \mathcal{O}}=0
\end{cases}
\en
with $0<\underline{c} \leq c_{0}\leq \bar{c}<\infty$
such that
$Q\in L^{\infty}(0, T; H^{1}(\mathcal{O}))\cap L^{2}(0, T; H^{2}(\mathcal{O}))$, 
$c\in L^{\infty}(0, T; H^{1}(\mathcal{O}))\cap L^{2}(0, T; H^{2}(\mathcal{O}))$,
and $0<\underline{c}\leq c \leq \bar{c}<\infty$. Moreover, the above mapping $u\mapsto (c[u], Q[u])$ is continuous
from $\mathcal{N}_{N}$ to $\big(L^{\infty}(0, T; H^{1}(\mathcal{O}))\cap L^{2}(0, T; H^{2}(\mathcal{O}))\big)
 \times \big(L^{\infty}(0, T; H^{1}(\mathcal{O}))\cap L^{2}(0, T; H^{2}(\mathcal{O}))\big)$,
 and $Q[u]\in S^{3}_{0}$ {\it a.e.} in $\mathcal{O}_{T}$.
\end{Lemma}

\begin{proof} We divide the proof into five steps.

\smallskip
1. The existence of a solution $(c, Q)$ can be achieved by the standard parabolic theory \cite{L-S-U-1968}.
The boundedness of $c$ is guaranteed by the fact that the maximum principle is valid for the first equation
in system \eqref{Q-system-FG} and the initial condition of $c$.
Next, we show $(c, Q)$ belongs to
$\big(L^{\infty}(0, T; H^{1}(\mathcal{O}))\cap L^{2}(0, T; H^{2}(\mathcal{O}))\big)
 \times \big(L^{\infty}(0, T; H^{1}(\mathcal{O}))\cap L^{2}(0, T; H^{2}(\mathcal{O}))\big)$.

\smallskip
2. Assume that $u\in \mathcal{N}_{N}$.
First, let us take the sum of the first equation in \eqref{Q-system-FG} multiplied by $c-\D c$
and the second equation in \eqref{Q-system-FG} multiplied by
$-\big(\D Q-Q-c_{*}Q\tr (Q^{2})\big)$, take the trace, and integrate by parts over $\mathcal{O}$ to obtain
\begin{align*}
&\frac{1}{2}\frac{\mathrm{d}}{\md t}\int_{\mathcal{O}}\big(|(c, Q, \na c, \na Q)|^{2}
  +\frac{c_{*}}{2}|Q|^{4}\big) \md x\\
&\quad\,  +D_{0}\|(\na c, \D c)\|_{L^{2}}^{2}
   +\Gamma \|(\na Q, \D Q)\|_{L^{2}}^{2} +c_{*}\Gamma \|Q\|_{L^{4}}^{4}+c_{*}^{2}\Gamma \|Q\|_{L^{6}}^{6}\\[1mm]
&=(u\cdot \na c, \D c)-(u\cdot \na c, c)+(u\cdot \nabla Q, \D Q)-(u\cdot \na Q, Q+c_{*}Q|Q|^{2})-(\Omega Q-Q\Omega, \D Q) \\
&\quad +\left(\O Q-Q\O, Q +c_{*}Q|Q|^{2}\right)
 +\frac{1}{2}\Gamma((c-c_{*}) Q, \D Q-Q-c_{*} Q\,\tr (Q^{2}))\\
 &\quad -b\Gamma (Q^{2}, \D Q-Q-c_{*} Q\,\tr (Q^{2})) +2c_{*}\Gamma(Q|Q|^{2}, \D Q)\\
&\leq \frac{D_{0}}{2}\|\D c\|_{L^{2}}^{2}+\frac{\Gamma}{2}\|\D Q\|_{L^2 }^{2} +\frac{c_{*}^{2}\Gamma}{2}\|Q\|_{L^{6} }^{6}
+C\big(\|(c, Q, \na c, \na Q)\|_{L^{2}}^{2}
+\|Q\|_{L^4 }^{4}\big).
\end{align*}

By the Gronwall inequality, we obtain
\begin{align*}
 \|(c,Q)\|_{L^{\infty}(0, T; H^{1}(\mathcal{O}))}+\|(c,Q)\|_{L^2 (0, T; H^2 (\mathcal{O}))}\leq C,
\end{align*}
where $C>0$ depends on $N, b, c_{*}, \Gamma, T, \|Q_{0}\|_{H^1 (\mathcal{O})}$,
and $\|c_{0}\|_{H^{1}(\mathcal{O})}$.

\smallskip
3. For the uniqueness, we denote $(\tilde{c}, \tilde{Q})=(c_{1}-c_{2}, Q_{1}-Q_{2})$
for any two solutions $(c_{1}, Q_{1})$ and $(c_{2}, Q_{2})$ of \eqref{Q-system-FG}.
Then $(\tilde{c}, \tilde{Q})$ satisfies
\be\label{tildeQ}
\begin{cases}
\del_{t}\tilde{c}+u\cdot \na \tilde{c} =D_{0}\D \tilde{c},\\[1mm]
\del_{t}\tilde{Q}+u\cdot \na \tilde{Q} -\O \tilde{Q}+\tilde{Q}\O 
-\Gamma \D \tilde{Q}
+\frac{1}{2}\Gamma\tilde{c} Q_{1}+\frac{1}{2}\Gamma \tilde{Q}(c_{2}-c_{*})\\[1mm]
\quad =\Gamma \Big(b\big(\tilde{Q}(Q_{1}+Q_{2})-\frac{1}{3}\tr (\tilde{Q}(Q_{1}+Q_{2}))\mathrm{I}_{3}\big)
  -c_{*}\tilde{Q}\,\tr \, Q_{1}^2 -c_{*}Q_{2}\,\tr (\tilde{Q}(Q_{1}+Q_{2}))\Big)
\end{cases}
\en
with $(\tilde{c}, \tilde{Q})|_{t=0}=(0, 0)$, $\na \tilde{Q}\cdot \vec{n}|_{\del \mathcal{O}}=0$, and $\na \tilde{c}\cdot \vec{n}|_{\del \mathcal{O}}=0$.
We sum up the first equation in (\ref{tildeQ}) multiplied by $\tilde{c}$ and the second equation in \eqref{tildeQ} multiplied by $\tilde{Q}$,
take the trace, and integrate by parts over $\mathcal{O}$ to obtain
\begin{align*}
&\frac{1}{2}\frac{\md}{\md t}\|(\tilde{c},\tilde{Q})\|_{L^{2}}^2
 +D_{0}\|\na \tilde{c}\|_{L^{2}}^{2}+\Gamma \|\na \tq\|_{L^{2}}^2 \\
&=-(u \cdot \na \tilde{c}, \tilde{c})-(u\cdot \na \tq, \tq)+(\O \tq-\tq \O, \tq)
-\frac{\Gamma}{2} (\tilde{c}Q_{1}, \tq)-\frac{\Gamma}{2}((c_{2}-c_{*})\tq, \tq)\\
&\quad +\Gamma(b(\tq (Q_{1}+Q_{2}) -\frac{1}{3}\tr(\tq (Q_{1}+Q_{2}))\mathrm{I}_{3})
  -c_{*}\tilde{Q}\,\tr \, Q_{1}^2 -c_{*}Q_{2}\,\tr (\tilde{Q}(Q_{1}+Q_{2})) ,\, \tq)\\
&\leq \frac{D_{0}}{2}\|\na \tilde{c}\|_{L^{2}}^{2}
  +C\big(\|(\tilde{c}, \tq)\|_{L^{2}}^{2}+ \| \tq\|_{L^{2}} \|\tq\|_{L^6 } (\|Q_{1}+Q_{2}\|_{L^3}+\|(Q_{1}, Q_{2})\|^{2}_{L^6 })+\|\na \tq\|_{L^{2}}\|\tq\|_{L^{2}})\\
&\leq\frac{D_{0}}{2}\|\na \tilde{c}\|_{L^{2}}^{2}+\frac{\Gamma}{2}\|\na \tq\|_{L^{2}}^2 +C\|(\tilde{c},\tq)\|_{L^{2}}^2,
\end{align*}
where $C$ depends on $b, c_{*}, \Gamma, N$, and $\|Q_{0}\|_{H^{1}}$, and we have also used Sobolev's embedding inequality,
Poincar\'{e} inequality, and Young's inequality in the last step.
Therefore, applying the Gronwall inequality, we obtain the desired uniqueness result.

\smallskip
4. Now we show that the map: $u\mapsto (c[u], Q[u])$ is continuous.
Let $\{u_{n}\}$ be a bounded sequence in $\mathcal{N}_{N}$ with
\begin{equation*}
\lim_{n\rightarrow \infty}\|u_{n}-u\|_{C(0, T; C^{2}_{0}(\bar{\mathcal{O}}))}=0
\end{equation*}
for some $u\in C(0, T; C^{2}_{0}(\bar{\mathcal{O}}))$.
Denote $(c_{n}, Q_{n})=(c[u_{n}], Q[u_{n}])$, $(c, Q)=(c[u], Q[u])$, and $(\tilde{c}_{n}, \tq_{n}) =(c_{n}-c, Q_{n}-Q)$.
Taking the difference of the equations satisfied by $(c_{n}, Q_{n})$ and $(c, Q)$,
multiplying the resulting equations by $(-\D \tilde{c}_{n}, -\D \tq_{n})$,
taking the trace, and integrating by parts over $\mathcal{O}$, we have
\begin{align*}
&\frac{1}{2}\frac{\md}{\md t} \| (\na \tilde{c}_{n},\na \tq_{n})\|_{L^{2}}^{2}
+D_{0}\|\D \tilde{c}_{n}\|_{L^{2}}^{2}+\Gamma \|\D \tq_{n}\|_{L^{2}}^2 \\
&=(u_{n}\cdot \na \tilde{c}_{n}+(u_{n}-u)\cdot \na c,  \D \tilde{c}_{n})+(u_{n}\cdot \na \tq_{n}+(u_{n}-u)\cdot \na Q, \D \tq_{n})\\
&\quad -(\O_{n}\tq_{n}+(\O_{n}-\O)Q, \D \tq_{n}) +(\tq_{n}\O_{n}+Q(\O_{n}-\O), \D \tq_{n})\\
&\quad +\frac{\Gamma}{2} (\tilde{c}_{n}Q_{n}+(c-c_{*})\tq_{n}, \D \tq_{n})
-\frac{1}{3}b\Gamma (3\tq_{n}(Q_{n}+Q)+\tr (\tq_{n}(Q_{n}+Q))\mathrm{I}_{3}, \D \tq_{n})\\
&\quad +c_{*}\Gamma (|Q_{n}|^{2}\tq_{n}+Q(|Q_{n}|^{2}-|Q|^{2}), \D \tq_{n})\\
&=\sum_{i=1}^{7}\i_{i}.
\end{align*}
In the following, we estimate all the terms on the right-hand side of the above equation:
\begin{align*}
|\i_{1}|&=|(u_{n}\cdot \na \tilde{c}_{n}+(u_{n}-u)\cdot \na c,  \D \tilde{c}_{n})|\\
&\leq C\big(\|u_{n}\|_{L^{\infty}}\|\na \tilde{c}_{n}\|_{L^{2}}+\|u_{n}-u\|_{L^{\infty}}\|\na c\|_{L^{2}}\big)\|\D  \tilde{c}_{n}\|_{L^{2}}\\
&\leq \frac{D_{0}}{2}\|\D \tilde{c}_{n}\|_{L^{2}}^{2}+C\|\na  \tilde{c}_{n}\|_{L^{2}}^{2}+C\| u_{n}-u\|_{L^{\infty}}^{2},
\end{align*}
\begin{align*}
|\i_{2}|&=|(u_{n}\cdot \na \tq_{n}+(u_{n}-u)\cdot \na Q, \D \tq_{n})|\\
&\leq C\big(\|u_{n}\|_{L^{\infty}}\|\na \tq_{n}\|_{L^{2}}+\|u_{n}-u\|_{L^{\infty}}\|\na Q\|_{L^{2}}\big)\|\D \tq_{n}\|_{L^{2}}\\
&\leq\frac{\Gamma}{12}\|\D \tq_{n}\|_{L^{2}}^{2}+C\|\na \tq_{n}\|_{L^{2}}^2 +C\|u_{n}-u\|_{L^{\infty}}^2,
\end{align*}
\begin{align*}
|\i_{3}|&=|(\O_{n}\tq_{n}+(\O_{n}-\O)Q, \D \tq_{n})|\\
&\leq C\|\O_{n}\|_{L^{\infty}}\|\tq_{n}\|_{L^{2}}\|\D \tq_{n}\|_{L^{2}}+\|\na u_{n}-\na u\|_{L^{\infty}}\|Q\|_{L^{2}}\|\D \tq_{n}\|_{L^{2}}\\
&\leq \frac{\Gamma}{12}\|\D \tq_{n}\|_{L^{2}}^2 +C\|\tq_{n}\|_{L^{2}}^2 +C\|\na u_{n}-\na u\|_{L^{\infty}}^2\\
&\leq \frac{\Gamma}{12}\|\D \tq_{n}\|_{L^{2}}^2 +C\|\na \tq_{n}\|_{L^{2}}^2 +C\|\na u_{n}-\na u\|_{L^{\infty}}^2,
\end{align*}
\begin{align*}
|\i_{4}|&=|(\tq_{n}\O_{n}+Q(\O_{n}-\O), \D \tq_{n})|\\
&\leq C\|\tq_{n}\|_{L^{2}}\|\O_{n}\|_{L^{\infty}}\|\D \tq_{n}\|_{L^{2}}+\|Q\|_{L^{2}}\|\O_{n}-\O\|_{L^{\infty}}\|\D \tq_{n}\|_{L^{2}}\\
&\leq \frac{\Gamma}{12}\|\D \tq_{n}\|_{L^{2}}^2 +C\|\tq_{n}\|_{L^{2}}^2 +C\|\na u_{n}-\na u\|_{L^{\infty}}^2\\
&\leq \frac{\Gamma}{12}\|\D \tq_{n}\|_{L^{2}}^2 +C\|\na \tq_{n}\|_{L^{2}}^2 +C\|\na u_{n}-\na u\|_{L^{\infty}}^2,
\end{align*}
\begin{align*}
|\i_{5}|&=\frac{\Gamma}{2}|(\tilde{c}_{n}Q_{n}+(c-c_{*})\tq_{n}, \D \tq_{n})|
\leq C\big(\|\tilde{c}_{n}\|_{L^{4}}\|Q_{n}\|_{L^{4}}+\|c-c_{*}\|_{L^{\infty}}\|\tq_{n}\|_{L^{2}}\big)\|\D \tq_{n}\|_{L^{2}}\\
&\leq \frac{\Gamma}{12}\|\D \tq_{n}\|_{L^{2}}^{2}+C\|\na \tilde{c}_{n}\|_{L^{2}}^{2}+C\|\tq_{n}\|_{L^{2}}^{2},
\end{align*}
\begin{align*}
|\i_{6}|&=\frac{1}{3}b\Gamma |(3\tq_{n}(Q_{n}+Q)+ \tr (\tq_{n}(Q_{n}+Q))\mathrm{I}_{3}, \D \tq_{n})|
 \leq C\|\tq_{n}\|_{L^{3}}\|Q_{n}+Q\|_{L^6}\|\D \tq_{n}\|_{L^{2}}\\
&\leq\frac{\Gamma}{12}\|\D \tq_{n}\|_{L^{2}}^{2}+C\|\na \tq_{n}\|_{L^{2}}^2,
\end{align*}
and
\begin{align*}
|\i_{7}|&=c_{*}\Gamma |(|Q_{n}|^{2}\tq_{n}+Q(|Q_{n}|^{2}-|Q|^{2}), \D \tq_{n})|\\
&\leq C\Gamma \big(\|Q_{n}\|_{L^6}^2 +\|Q_{n}\|_{L^6}\|Q\|_{L^{6}}+\|Q\|_{L^6}^{2}\big)\|\tq_{n}\|_{L^6}\|\D \tq_{n}\|_{L^{2}}\\
&\leq \frac{\Gamma}{12}\|\D \tq_{n}\|_{L^{2}}^2 +C\|\na \tq_{n}\|_{L^{2}}^2.
\end{align*}

Combining all the above estimates, we conclude
\begin{align*}
&\frac{1}{2}\frac{\md}{\md t}\|(\na \tilde{c}_{n}, \na \tq_{n})\|_{L^{2}}^{2}
+\frac{D_{0}}{2}\|\D \tilde{c}_{n}\|_{L^{2}}^{2} +\frac{\Gamma}{2}\|\D \tq_{n}\|_{L^{2}}^{2}\\[1mm]
&\leq C\|u_{n}-u\|_{C^{2}_{0}(\mathcal{O})}^{2}+C\| (\na \tilde{c}_{n}, \na \tq_{n})\|_{L^{2}}^{2}.
\end{align*}
Then, by the Gronwall inequality, we have
\be\label{continous-Q}
\begin{split}
&\|(\na \tilde{c}_{n}(t, \cdot), \na \tq_{n}(t, \cdot))\|_{L^{2}}^2
 +\Gamma \int_{0}^{t}\|(\D \tilde{c}_{n}(s, \cdot), \D \tq_{n}(s, \cdot))\|_{L^{2}}^{2}\,\md s\\
&\leq CTe^{CT}\|u_{n}-u\|_{L^{\infty}(0, T; C^{2}_{0}(\mathcal{O}))}^{2}.
\end{split}
\en
As $n\rightarrow \infty$, we conclude
\begin{align*}
\lim_{n\rightarrow \infty}\big(\|(\tilde{c}_{n},\tq_{n})\|_{L^{\infty}(0, T; H^{1})}
  +\|(\tilde{c}_{n}, \tq_{n})\|_{L^{2}(0, T; H^{2})}\big)=0.
\end{align*}

\smallskip
5. Finally, we show that $Q\in S^{3}_{0}$,
 {\it i.e.}, $Q^\top=Q$ and $\tr\, Q=0$ {\it a.e.} in $\mathcal{O}_{T}$.
It is clear that, if $Q$ is a solution of problem (\ref{Q-system-FG}), so is $Q^\top$.
Then, by the uniqueness of the solution 
we know $Q^\top=Q$.

Then the only thing left is to show that $\tr \, Q=0$.
We take the trace on both sides of the second equation in (\ref{Q-system-FG}) to obtain
\be\label{traceQ}
\del_{t}(\tr \, Q)+u\cdot \na \tr \, Q 
=\Gamma \big(\D \tr \, Q-\frac{1}{2}(c-c_{*})\tr \, Q-c_{*}\tr \, Q\,\tr \, Q^2\big),
\en
with $\tr \, Q |_{t=0}=\tr \, Q_{0}=0$ and $\na \tr \,Q \cdot \vec{n}|_{\del \mathcal{O}}=0$,
where we have used the fact that $Q^\top=Q$, $\O^\top=-\O$, and $\tr (Q\Omega)=\tr (\Omega Q)$.
Then we multiply equation (\ref{traceQ}) by $\tr \, Q$
and integrate by parts over $\mathcal{O}$ to obtain
\begin{align*}
&\frac{1}{2}\frac{\md}{\md t}\|\tr \, Q\|_{L^{2}}^2 +\Gamma \|\na \tr \, Q\|_{L^{2}}^{2}\\
&=-\frac{1}{2}\Gamma ((c-c_{*})\tr \, Q, \tr \, Q) -c_{*}\Gamma \int_{\mathcal{O}}|\tr \, Q|^2\,\tr \, Q^2 \mathrm{d}x
-\int_{\mathcal{O}} u\cdot \na (\tr \, Q) \,\tr \, Q\,\md x\\
&\leq C \|\tr \, Q\|_{L^{2}}^2 +C\|Q\|_{L^{6}}^{2}\|\tr \, Q\|_{L^{6}}\|\tr \, Q\|_{L^{2}}
+C\|\na \tr \, Q\|_{L^{2}}\|\tr \, Q\|_{L^{2}}\\
&\leq C \|\tr \, Q\|_{L^{2}}^2 +C\|\na \tr \, Q\|_{L^{2}}\|\tr \, Q\|_{L^{2}}
+C\|\na \tr \, Q\|_{L^{2}}\|\tr \, Q\|_{L^{2}}\\
&\leq \frac{\Gamma}{2}\|\na \tr \, Q\|_{L^{2}}^2 +C\|\tr \, Q\|_{L^{2}}^{2}.
\end{align*}
Applying the Gronwall inequality again, we complete the proof.
\end{proof}

\subsection{The Faedo-Galerkin approximation scheme}

In this section, we proceed to solve (\ref{u-FG}) by the Faedo-Galerkin approximation scheme. Let $\{\psi_{n} \}$ be a family of smooth eigenfunctions of the Laplacian operator:
\begin{align*}
&-\D \psi_{n}=\la_{n}\psi_{n} \qquad \mbox{ on } \mathcal{O}, \\
&\psi_{n}|_{\del \mathcal{O}}=0,
\end{align*}
where $0<\la_{1}\leq \la_{2}\leq \cdots$ are eigenvalues. We know that the eigenfunctions $\{\psi_{n}\}_{n=1}^{\infty}$ form an orthogonal basis of $H_{0}^{1}(\mathcal{O})$.

Now, consider a sequence of finite-dimensional spaces:
\begin{equation}\label{X-n}
X_{n}=\mbox{span} \{\psi_{1}, \psi_{2}, \cdots, \psi_{n}\}, \qquad n=1, 2, \cdots,
\end{equation}
and look for solutions $u_{n}\in C(0, T; X_{n})$ to the following variational approximation problem:
\be\label{variation}
\begin{split}
&\int_{\mathcal{O}}\rho(t, x) u_{n}(t, x)\cdot \psi(x) \,\md x-\int_{\mathcal{O}}m_{0}(x)\cdot \psi(x)\, \md x\\
&=\int_{0}^{t}\int_{\mathcal{O}}\big(\mu \D u_{n}+(\mu+\nu)\na \mathrm{div}\, u_{n}-\mathrm{div}\,(\rho u_{n}\otimes u_{n})
  -\na (\rho^{\gamma}+\delta \rho^{\b})-\ve \na \rho \cdot \na u_{n}\big)\cdot \psi\, \md x\md s \\
&\quad +\int_{0}^{t}\int_{\mathcal{O}}\na \cdot \big(\mathrm{F}(Q)\mathrm{I}_{3}-\na Q\odot \na Q
+ Q\D Q-\D QQ+\sigma_{*} c^{2}Q\big)\cdot \psi \,\md x\md s
\end{split}
\en
for any $t\in [0, T]$, and $\psi \in X_{n}$.

Next we introduce a family of operators, as in \cite{F-N-P-2001}:
\begin{align*}
\mathcal{M}[\rho]: X_{n}\mapsto X_{n}^{*},
\qquad (\mathcal{M}[\rho]v, w)=\int_{\mathcal{O}}\rho v\cdot w\, \md x \quad\mbox{for any $v, w\in X_{n}$},
\end{align*}

The map:
\begin{equation*}
\rho \mapsto \mathcal{M}^{-1}[\rho]
\end{equation*}
from $N_{\eta}=\{\rho\in L^{1}(\mathcal{O}) \; :\; \inf_{x\in \O}\rho \geq \eta >0\}$
into $\mathcal{L}(X_{n}^{*}, X_{n})$  has the following property:
\be\label{continous-M-inverse}
\|\mathcal{M}^{-1}[\rho_{1}]-\mathcal{M}^{-1}[\rho_{2}]\|_{\mathcal{L}(X_{n}^{*}, X_{n})}\leq C(n, \eta)\|\rho_{1}-\rho_{2}\|_{L^{1}(\mathcal{O})}.
\en
Using Theorems \ref{solution-rho}--\ref{solution-Q} with $\rho_{n}=S[u_{n}]$, $c_{n}=c[u_{n}]$, and $Q_{n}=Q[u_{n}]$,
we can rewrite the variational problem (\ref{variation}) as
\begin{equation*}
u_{n}(t)=\mathcal{M}^{-1}[\rho_{n}]\left(m^{*}+\int_{0}^{t}\mathcal{N}[c_{n}(s), \rho_{n}(s), u_{n}(s), Q_{n}(s)]\md s\right)
\end{equation*}
with
\begin{align*}
&(m^{*}, \psi)=\int_{\mathcal{O}}m_{0}\cdot \psi \,\md x,\\[2mm]
&(\mathcal{N}[c_{n}, \rho_{n}, u_{n}, Q_{n}], \psi)\\
&=\int_{\mathcal{O}}\big(\mu \D u_{n}+(\mu+\nu)\na \mathrm{div}\, u_{n}-\mathrm{div}\,(\rho_{n} u_{n}\otimes u_{n})
  -\na (\rho_{n}^{\gamma}+\delta \rho_{n}^{\b})\big)\cdot \psi\, \md x\\
&\qquad-\ve \int_{\mathcal{O}}(\na \rho_{n} \cdot \na) u_{n}\cdot \psi \,\md x +\int_{\mathcal{O}}\na \cdot \big(-\na Q_{n}\odot \na Q_{n}+\mathrm{F}(Q_{n})\mathrm{I}_{3}\big)\cdot \psi\, \md x\\
&\qquad +\int_{\mathcal{O}}\na \cdot \big(Q_{n}\D Q_{n}-\D Q_{n}Q_{n}
 +\sigma_{*} c_{n}^{2}Q_{n}\big)\cdot \psi\, \md x
\end{align*}
for any $t\in [0, T]$ and $\psi \in X_{n}$.
Therefore, combining  (\ref{continous-S}), (\ref{continous-Q}), and (\ref{continous-M-inverse}),
we achieve a local solution $(c_n, \rho_{n}, u_{n}, Q_{n})$ of problem (3.1)--(\ref{rho-FG}), (\ref{Q-FG}), and (\ref{variation}),
with initial-boundary data (\ref{I-C-rho-a})--(\ref{B-C-uQ}) on a short time interval $[0, T_{n}], T_{n}\leq T$,
by using the standard fixed point theorem on $C(0, T; X_{n})$.
In order to extend the existence time $T_{n}$ to $T$ for any $n=1, 2, \cdots$,
we need to prove that $u_{n}$ stays bounded in $X_{n}$ for the whole interval $[0, T_{n}]$.
Hence, in the following, we establish an energy inequality as in Proposition \ref{energy-inequality}.

Differentiate (\ref{variation}) with respect to $t$ and take $\psi =u_{n}$ as a test function to obtain
\be\label{u-n-energy}
\begin{split}
&\frac{\mathrm{d}}{\md t}\int_{\mathcal{O}}\big(\frac{1}{2}\rho_{n} |u_{n}|^{2}+\frac{\rho_{n}^{\gamma}}{\gamma-1}+\frac{\delta \rho_{n}^{\b}}{\b-1}\big)\,\md x
 +\mu \|\na u_{n}\|_{L^{2}}^2 +(\nu +\mu)\|\mathrm{div}\, u_{n}\|_{L^{2}}^2 \\
&\qquad +\ve \int_{\mathcal{O}}(\gamma \rho_{n}^{\gamma-2}+\delta \b \rho_{n}^{\b-2})|\na \rho_{n}|^{2}\,\md x \\
&=\int_{\mathcal{O}}\na \cdot \big(-\na Q_{n}\odot \na Q_{n}+\mathrm{F}(Q_{n})\mathrm{I}_{3}+Q_{n}\D Q_{n}-\D Q_{n}Q_{n}
+\sigma_{*}c_{n}^{2}Q_{n}\big)\cdot u_{n}\,\md x,
\end{split}
\en
where we have used the following equalities as in Proposition \ref{energy-inequality},
\begin{align*}
&\int_{\mathcal{O}}\del_{t}(\rho_{n}u_{n})\cdot u_{n}\,\md x +\int_{\mathcal{O}}\mathrm{div}\,(\rho_{n}u_{n}\otimes u_{n})\cdot u_{n}\,\md x \\
&\quad =\frac{1}{2}\frac{\mathrm{d}}{\md t}\int_{\mathcal{O}}\rho_{n} |u_{n}|^{2}\,\md x -\ve \int_{\mathcal{O}} (\na \rho_{n}\cdot \na) u_{n}\cdot u_{n} \, \md x,\\[2mm]
&\int_{\mathcal{O}}\na \rho_{n}^{\gamma}\cdot u_{n}\,\mathrm{d}x
=\frac{\mathrm{d}}{\mathrm{d}t}\int_{\mathcal{O}}\frac{\rho_{n}^{\gamma}}{\gamma-1}\,\mathrm{d}x
  +\ve \gamma \int_{\mathcal{O}} \rho_{n}^{\gamma -2}|\na \rho_{n}|^2 \,\mathrm{d}x,\\[1mm]
&\int_{\mathcal{O}}\delta \na \rho_{n}^{\b}\cdot u_{n}\,\mathrm{d}x
  =\frac{\mathrm{d}}{\mathrm{d}t}\int_{\mathcal{O}}\frac{\delta \rho_{n}^{\b}}{\b-1}\,\mathrm{d}x
  +\ve \delta \b \int_{\mathcal{O}} \rho_{n}^{\b -2}|\na \rho_{n}|^2\, \mathrm{d}x.
\end{align*}
Then we take the inner product of \eqref{c-FG} with $c_{n}$, \eqref{Q-FG} with $-\big(\D Q_{n}-Q_{n}-c_{*}Q_{n}\tr (Q_{n}^{2})\big)$,
add the resulting equations to \eqref{u-n-energy} and integrate by parts over $\mathcal{O}$ to obtain
 \begin{align*}
&\frac{\md}{\md t}E_{\delta}^{n}(t) +D_{0}\|\na c_{n}\|_{L^{2}}^{2}+\mu \|\na u_{n}\|_{L^{2}}^{2}+(\mu +\nu)\|\mathrm{div}\, u_{n}\|_{L^{2}}^{2}+\Gamma \|(\na Q_{n}, \D Q_{n})\|_{L^{2}}^{2}\\
&\quad +c_{*}\Gamma \|Q_{n}\|_{L^{4}}^{4}+c_{*}^{2}\Gamma \|Q_{n}\|_{L^{6}}^{6}+\ve \int_{\mathcal{O}}(\gamma \rho_{n}^{\gamma-2}+\delta \b \rho_{n}^{\b-2})|\na \rho_{n}|^{2}\,\md x  \\
&=\int_{\mathcal{O}}\na \cdot \big(-\na Q_{n}\odot \na Q_{n}+\mathrm{F}(Q_{n})\mathrm{I}_{3}+Q_{n}\D Q_{n}-\D Q_{n}Q_{n}
  +\sigma_{*}c_{n}^{2}Q_{n}\big)\cdot u_{n}\, \md x
\\
&\quad+(u_{n}\cdot \nabla Q_{n}, \D Q_{n})-(u_{n}\cdot \na Q_{n}, Q_{n}+c_{*}Q_{n}|Q_{n}|^{2})-(\Omega_{n} Q_{n}-Q_{n}\Omega_{n}, \D Q_{n})\\
&\quad +\left(\O_{n} Q_{n}-Q_{n}\O_{n}, Q_{n} +c_{*}Q_{n}|Q_{n}|^{2}\right)+\frac{1}{2}\Gamma((c_{n}-c_{*})Q_{n}, \D Q_{n}-Q_{n}-c_{*} Q_{n}\tr (Q_{n}^{2}))\\
&\quad -b\Gamma (Q_{n}^{2}, \D Q_{n}-Q_{n}-c_{*} Q_{n}\tr (Q_{n}^{2}))+2c_{*}\Gamma(Q_{n}|Q_{n}|^{2}, \D Q_{n}) -(u_{n}\cdot \na c_{n}, c_{n})\\
&\leq \frac{D_{0}}{2}\|\na c_{n}\|_{L^{2}}^{2} +\frac{\mu}{2}\|\na u_{n}\|_{L^{2}}^{2} +\frac{\Gamma}{2}\|\D Q_{n}\|_{L^{2}}^{2} +\frac{c_{*}\Gamma}{2}\| Q_{n}\|_{L^{6}}^{6}
   + C\big(\|(u_{n}, Q_{n})\|_{L^{2}}^{2}+\|Q_{n}\|_{L^{4}}^{4}\big)
\end{align*}
where $C$ is independent of $n$ and $\ve$, and 
\begin{align*}
E_{\delta}^{n}(t)=\int_{\mathcal{O}}
\Big(\frac{1}{2}|c_{n}|^{2} +\frac{1}{2}\rho_{n} |u_{n}|^{2}+\frac{\rho_{n}^{\gamma}}{\gamma-1}+\frac{\delta \rho_{n}^{\b}}{\b-1}+\mathrm{F}(Q_{n})\Big)\,\md x.
\end{align*} 
This implies
\be\label{energy-estimate-n-pre}
\begin{split}
&\frac{\md}{\md t}E^{n}_{\delta}(t)
+\frac{D_{0}}{2}\|\na c_{n}\|_{L^{2}}^{2}+\frac{\mu}{2}\|\na u_{n}\|_{L^{2}}^{2}+(\mu +\nu)\|\mathrm{div}\, u_{n}\|_{L^{2}}^{2}+\Gamma \|\na Q_{n}\|_{L^{2}}^{2}\\
&\quad +\frac{\Gamma}{2}\|\D Q_{n}\|^2_{L^{2}}+\frac{c_{*}^{2}\Gamma}{2}\|Q_{n}\|_{L^{6}}^{6}
 +\ve \int_{\mathcal{O}}\big(\gamma \rho_{n}^{\gamma-2}+\delta \b \rho_{n}^{\b-2}\big)|\na \rho_{n}|^{2}\md x  \\
&\leq C\big(\|(u_{n}, Q_{n})\|_{L^{2}}^{2}+\|Q_{n}\|_{L^{4}}^{4}\big),
\end{split}
\en
The above inequality yields
\be
\mu \int_{0}^{T_{n}}\|\na u_{n}\|_{L^{2}}^2\,\md t \leq C,
\en
\be\label{rho-u-n-bound}
\sup_{t\in [0, T_{n}]}\int_{\mathcal{O}}\rho_{n}|u_{n}|^2 \,\md x \leq C,
\en
where $C$ is a constant independent of $n$.
Since   $X_{n}$ is a finite-dimension space, we can deduce from Lemma \ref{solution-rho}
that there exists a constant $C=C(n, c_{0}, \rho_{0}, m_{0}, Q_{0}, b, \mathcal{O})$ such that
\be\label{rho-n-bound}
0< C\leq \rho_{n}(t, x)\leq \frac{1}{C} \qquad\,\, \mbox{for all $t\in (0, T_{n})$ and $x\in \mathcal{O}$},
\en
which, combined with (\ref{rho-u-n-bound}) and the fact that the $L^{\infty}$ and $L^2$ norms are equivalent on $X_{n}$, yields
\begin{equation*}
\sup_{t\in [0, T_{n}]}\|(u_{n}, \na u_{n})(t, \cdot)\|_{L^{\infty}(\mathcal{O})}
\leq C(n, E^{n}_{\delta}(0), N, \mathcal{O}).
\end{equation*}
Then we can extend the existence time-interval $[0, T_{n}]$ of $(c_{n}, \rho_{n}, u_{n}, Q_{n})$ to $[0, T]$.

We summarize the results in this subsection in the following lemma.

\begin{Lemma} \label{FG-n-estimates}
Let $\b \geq 4$. Then there exists a solution $(c_{n}, \rho_{n}, u_{n}, Q_{n})$ of  problem \eqref{c-FG}--\eqref{rho-FG},
\eqref{Q-FG}, and \eqref{variation} with the corresponding initial-boundary data \eqref{I-C-c}--\eqref{B-C-uQ}.
Moreover, the following estimates hold{\rm :}
\begin{align}
&0<\underline{c}\leq c_{n}(t, x)\leq \bar{c}, \quad \|c_{n}\|_{ L^{\infty}(0, T; L^{2}(\mathcal{O}))\cap L^{2}(0, T; H^{1}(\mathcal{O}))}\leq C,\label{c-n}\\[1mm]
&\sup_{t\in [0, T]}\|\rho_{n}(t, \cdot)\|_{L^{\gamma}(\mathcal{O})}^{\gamma}\leq C,\label{rho-n-gamma}\\
&\delta \sup_{t\in [0, T]}\|\rho_{n}(t, \cdot)\|_{L^{\b}(\mathcal{O})}^{\b}\leq C,\label{rho-n-b}\\[1mm]
&\ve \|\na \rho_{n}^{\frac{\gamma}{2}}\|^{2}_{L^{2}(0, T; L^{2}(\mathcal{O}))}
 +\ve \delta \|\na \rho_{n}^{\frac{\b}{2}}\|^{2}_{L^{2}(0, T; L^{2}(\mathcal{O}))}\leq C,\label{na-rho-n-g-b}\\
&\sup_{t\in [0, T]}\|\sqrt{\rho_{n}}(t, \cdot)u_{n}(t, \cdot)\|_{L^{2}(\mathcal{O})}^{2}\leq C,\label{rho-u-n-2}\\
&\|u_{n}\|_{L^{2}(0, T; H^{1}_{0}(\mathcal{O}))}\leq C,\label{u-n-h1}\\[1mm]
&\|\rho_{n}\|_{L^{\b +1}(\mathcal{O}_{T})}\leq C,\label{rho-n-b+1}\\[1mm]
&\ve \|\na \rho_{n}\|^{2}_{L^{2}(0, T; L^{2}(\mathcal{O}))}\leq C,\label{na-rho-n}\\[1mm]
&\|Q_{n}\|_{L^{\infty}(0, T; H^{1}(\mathcal{O})\cap L^{4}(\mathcal{O}))\cap L^{2}(0, T; H^{2}(\mathcal{O})\cap L^{6}(\mathcal{O}))}\leq C,\label{Q-n}\\[1mm]
&\|Q_{n}\|_{L^{10}(\mathcal{O}_{T})}\leq C,\label{Q-n-10}\\[1mm]
&\|\na Q_{n}\|_{L^{\frac{10}{3}}(\mathcal{O}_{T})}\leq C,\label{na-Q-n}
\end{align}
where $C$ is a constant independent of $n$ and $\ve$.
\end{Lemma}

\begin{proof}
Estimates (\ref{c-n})$-$(\ref{u-n-h1}) and (\ref{Q-n}) follow from the energy
estimate (\ref{energy-estimate-n-pre}).
Moreover, we can use
similar methods to obtain (\ref{Q-n-10})--(\ref{na-Q-n}) as in Corollary \ref{Q-estimate-addition}.
We only need to show (\ref{rho-n-b+1})--(\ref{na-rho-n}).

From \eqref{na-rho-n-g-b}, $\rho_{n}^{\b/2}\in L^{2}_{t}H^{1}_{x}$.
This, together with the embedding: $H^{1}(\mathcal{O})\subset L^{6}(\mathcal{O})$, yields
\be\label{rho-n-b-l3}
\|\rho_{n}^{\b}\|_{L^{1}_{t}L^{3}_{x}}\leq C
\en
with $C$ independent of $n$.
Combining (\ref{rho-n-b}), \eqref{rho-n-b-l3}, and the interpolation (pp. 623, \cite{E-1998}):
\begin{align*}
\|\rho_{n}^{\b}\|_{L^{2}_{x}}\leq C\|\rho_{n}^{\b}\|_{L^{1}_{x}}^{\frac{1}{4}}\|\rho_{n}^{\b}\|_{L^{3}_{x}}^{\frac{3}{4}},
\end{align*}
we have
\begin{align*}
\|\rho_{n}^{\b}\|^{\frac{4}{3}}_{L^{\frac{4}{3}}_{t} L^{2}_{x}}
\leq C\int_{0}^{T}\|\rho_{n}^{\b}\|_{L^{1}_{x}}^{\frac{1}{3}}\|\rho_{n}^{\b}\|_{L^{3}_{x}}\md t \leq C \|\rho_{n}^{\b}\|_{L^{\infty}_{t} L^{1}_{x}}^{\frac{1}{3}}\|\rho_{n}^{\b}\|_{L^{1}_{t} L^{3}_{x}}\leq C,
\end{align*}
which implies that $\rho_{n}\in L^{\frac{4}{3}\b}(0, T; L^{2\b}(\mathcal{O}))$.
Moreover, this, together with the following interpolation:
\begin{align*}
\|\rho_{n}\|_{L^{\frac{4}{3}\b}_{x}}\leq C\|\rho_{n}\|_{L^{\b}_{x}}^{\frac{1}{2}}\|\rho_{n}\|_{L^{2\b}_{x}}^{\frac{1}{2}},
\end{align*}
gives
\begin{align*}
\|\rho_{n}\|_{L^{\frac{4}{3}\b}_{t, x}}^{\frac{4\b}{3}}\leq C\int_{0}^{T}\|\rho_{n}\|_{L^{\b}_{x}}^{\frac{2}{3}\b}\|\rho_{n}\|_{L^{2\b}_{x}}^{\frac{2}{3}\b}\, \md t\leq C(\delta),
\end{align*}
{\it i.e.}, $\rho_{n}\in L^{\frac{4}{3}\b}(\mathcal{O}_{T})$.
Then, if $\b \geq 3$ ($\Leftrightarrow\frac{4}{3}\b\geq \b+1$), we have
\begin{align*}
\|\rho_{n}\|_{L^{\b+1}(\mathcal{O}_{T})}\leq C(\dl).
\end{align*}
Regarding (\ref{na-rho-n}), we multiply (\ref{rho-FG}) by $\rho_{n}$ and integrate by parts over $\mathcal{O}$ to obtain
\begin{align*}
&\frac{1}{2}\frac{\md }{\md t}\|\rho_{n}\|_{L^{2}_{x}}^{2}+\ve\|\na \rho_{n}\|_{L^{2}_{x}}^{2}
=-\frac{1}{2}\int_{\mathcal{O}}\mathrm{div}\, u_{n}\, \rho_{n}^{2}\,\md x \leq C\|\mathrm{div}\, u_{n}\|_{L^{2}_{x}}\|\rho_{n}\|_{L^{4}_{x}}^{2},
\end{align*}
which yields
\begin{align*}
\ve \|\na \rho_{n}\|_{L^{2}(\mathcal{O}_{T})}^{2}
\leq \frac{1}{2}\big(\|\rho_{0}\|_{L^{2}_{x}}^{2}+\sqrt{T}\|\rho_{n}\|^{2}_{L^{\infty}_{t}L^{4}_{x}}\|\mathrm{div} \,u_{n}\|_{L^{2}_{t, x}}\big).
\end{align*}
Therefore, (\ref{na-rho-n}) follows from (\ref{rho-n-b}) and (\ref{u-n-h1}), provided $\b\geq 4$.
\end{proof}

\subsection{The existence of the first level approximate solutions}\label{1st-approx}
In this subsection, we
obtain a solution $(c, \rho, u, Q)$ of  problem (3.1)$-$(\ref{B-C-uQ}),
by letting $n\to \infty$.
We do not distinguish between the sequence convergence and the subsequence convergence for the sake of convenience.
Assume that $\b>4$ and $\gamma>\frac{3}{2}$.
It follows from \cite{F-N-P-2001} that as $n\to \infty$,
\begin{eqnarray}
&&\rho_{n}\to \rho \qquad\,\, \mbox{ in } L^{4}(\mathcal{O}_{T}),\label{rho-n-to-rho-l4}\\
&&\rho_{n}^{\gamma}\to \rho^{\gamma}, \,\,\, \rho_{n}^{\b}\to \rho^{\b} \qquad\,\, \mbox{ in } L^{1}(\mathcal{O}_{T}) \mbox{ if } \b > \gamma, \\
&&u_{n}\rightharpoonup u \qquad\,\, \mbox{ in } L^{2}(0, T; H^{1}_{0}(\mathcal{O}))\label{u-n-to-u-h1}.
\end{eqnarray}
Moreover, we can infer that $\del_{t}c_{n}\in L^{1}(0, T; H^{-1}(\mathcal{O}))$ from \eqref{c-n}, \eqref{u-n-h1},
and that $\del_{t}c$ satisfies equation \eqref{c-FG}.
Applying the Aubin-Lions lemma, we have
\be\label{c-n-conv}
c_{n}\rightharpoonup c\quad  \mbox{in} \,\, L^{2}(0, T; H^{1}(\mathcal{O})),
\qquad \,\, c_{n}\to c \quad \mbox{in}\,\, L^{2}(0, T; L^{2}(\mathcal{O})).
\en
Thus, we know that the limit function $c$ is a weak solution to \eqref{c-FG}.
Similarly, estimates \eqref{c-n}, (\ref{u-n-h1}), and (\ref{Q-n}), along with the fact that $\del_{t}Q_{n}$ satisfies (\ref{Q-FG}),
yield that $\del_{t}Q_{n}\in L^{2}(0, T; L^{\frac{3}{2}}(\mathcal{O}))$.
In fact,
\begin{align*}
&\|(u_{n}\cdot \nabla )Q_{n}\|_{L^{2}_{t}L^{\frac{3}{2}}_{x}}
 \leq C\Big(\int_{0}^{T}\|u_{n}\|_{L^{6}_{x}}^{2}\|\na Q_{n}\|_{L^{2}_{x}}^{2}\md t\Big)^{\frac{1}{2}}
 \leq C\|u_{n}\|_{L^{2}_{t}(H^{1}_{0})_{x}}\|\na Q_{n}\|_{L^{\infty}_{t}L^{2}_{x}}\leq C,\\
&\|Q_{n}\O_{n}-\O_{n}Q_{n}\|_{L^{2}_{t}L^{\frac{3}{2}}_{x}}
 \leq C\Big(\int_{0}^{T}\|Q_{n}\|_{L^{6}_{x}}^{2}\|\na u_{n}\|_{L^{2}_{x}}^{2}\md t\Big)^{\frac{1}{2}}
  \leq C\|Q_{n}\|_{L^{\infty}_{t}H^{1}_{x}}\|\na u_{n}\|_{L^{2}(\mathcal{O}_{T})}\leq C,\\
&\|\Gamma H[Q_{n},c_{n}]\|_{L^{2}_{t}L^{\frac{3}{2}}_{x}}=\Gamma\|\D Q_{n}-\frac{c_{n}-c_{*}}{2}Q_{n}+b[Q_{n}^{2}-\frac{\tr (Q_{n}^{2})}{3}\mathrm{I}_{d}] -c_{*}Q_{n}|Q_{n}|^{2}\|_{L^{2}_{t}L^{\frac{3}{2}}_{x}}\\
&\leq C\Big(\|(Q_n, \D Q_{n})\|_{L^{2}(\mathcal{O}_{T})}
+(\|c_{n}\|_{L^{2}_{t}H^{1}_{x}}
  +\|Q_{n}\|_{L^{2}_{t}H^{2}_{x}})(\|Q_{n}\|_{L^{\infty}_{t}H^{1}_{x}}+\|Q_{n}\|^{2}_{L^{\infty}_{t}H^{1}_{x}})\Big)\leq C.
\end{align*}
Then, by the Aubin-Lions lemma, we have
\be\label{Q-n-conv}
Q_{n}\rightharpoonup Q \,\, \mbox{ in } L^{2}(0, T; H^{2}(\mathcal{O})),\qquad\,\,
Q_{n}\to Q \,\,\mbox{ in } L^{2}(0, T; H^{1}(\mathcal{O})).
\en
This ensures that we can pass to the limit in equation (\ref{Q-FG})
in $\mathcal{D}^{'}(\mathcal{O}_{T})$ as $n\to\infty$, {\it i.e.}, $Q$ is a weak solution of (\ref{Q-FG}).

Furthermore, we also see that $\rho_{n}u_{n}$ is bounded in $L^{\infty}(0, T; L^{\frac{2\gamma}{\gamma +1}}(\mathcal{O}))$
with $\frac{2\gamma}{\gamma +1}>\frac{6}{5}$ (since $\gamma >\frac{3}{2}$), by using (\ref{rho-n-gamma}) and (\ref{rho-u-n-2})--(\ref{u-n-h1}).
In fact,
\be\label{rho-u-n-2g/g+1}
\|\rho_{n}u_{n}\|_{L^{\frac{2\gamma}{\gamma +1}}_{x}}\leq
C\|\rho_{n}\|_{L^{\gamma}_{x}}^{\frac{1}{2}}\|\sqrt{\rho_{n}}u_{n}\|_{L^{2}_{x}}.
\en
This, together with (\ref{rho-n-to-rho-l4}) and (\ref{u-n-to-u-h1}), yields
$$
\rho_{n}u_{n} \overset{*}\rightharpoonup \rho u \quad \mbox{ in } L^{\infty}(0, T; L^{\frac{2\gamma}{\gamma +1}}(\mathcal{O})).
$$
Then we can conclude that $\rho$ is a weak solution of (\ref{rho-FG}) and we can pass to the limit  in (\ref{rho-FG}) as $n\to \infty$.

\smallskip
In the following, we show that the limit function $u$ satisfies equation (\ref{variation}),
by using Corollary \ref{suff-cond-C-Weak-star} in Appendix A.
Then we need to establish convergence results for the terms: $\rho_{n}u_{n}\otimes u_{n}$ and $\na \rho_{n}\cdot \na u_{n}$,
which require more estimates for the density.
From Lemma 2.4 in \cite{F-N-P-2001}, we know that (\ref{rho-FG}) holds in the following strong sense:

\begin{Lemma}\label{strong-conv-rho-n}
There exist $r>1$ and $q>2$ such that
\begin{align*}
&\del_{t}\rho_{n}, \D \rho_{n} \qquad \mbox{ are bounded in }L^{r}(\mathcal{O}_{T}),  \\
&\na \rho_{n} \qquad\qquad\,\, \mbox{ is bounded in } L^{q}(0, T; L^{2}(\mathcal{O})),
\end{align*}
independently with respect to $n$. Consequently, the limit function $\rho$ belongs to the same class
and satisfies equation \eqref{rho-FG} almost everywhere on $\mathcal{O}_{T}$
and the boundary conditions \eqref{B-C-rho} in the sense of traces.
\end{Lemma}

To continue the proof we first show that $\int_{\mathcal{O}}\rho_{n}u_{n}\cdot \psi\, \mathrm{d}x$ is equi-continuous in $t$,
for any fixed test function $\psi \in X_{n}$ in (\ref{variation}).
Using Lemmas \ref{FG-n-estimates}--\ref{strong-conv-rho-n}, we see that, for any $0<\xi <1$,
\begin{align*}
&\int_{t}^{t+\xi}\int_{\mathcal{O}}(\mu \D u_{n}+(\mu+\nu)\na \mathrm{div}\, u_{n})\cdot \psi\, \md x\md s
\leq C\int_{t}^{t+\xi}\|\na u_{n}\|_{L^{2}_{x}}\|\na \psi\|_{L^{2}_{x}}\,\md s
\leq C \sqrt{\xi},\\
&\int_{t}^{t+\xi}\int_{\mathcal{O}}\mathrm{div}\,(\rho_{n} u_{n}\otimes u_{n})\cdot \psi\, \md x\md s
  \leq  \int_{t}^{t+\xi}\|\sqrt{\rho_{n}}u_{n}\|_{L^{2}_{x}}^{2}\|\na \psi \|_{L^{\infty}_{x}}\,\md s\leq C\xi,\\[1mm]
&\int_{t}^{t+\xi}\int_{\mathcal{O}}\na (\rho^{\gamma}_{n}+\delta \rho_{n}^{\b})\cdot \psi\, \md x\md s
 \leq  \int_{t}^{t+\xi}\big(\|\rho_{n}\|_{L^{\gamma}_{x}}^{\gamma}+\delta \|\rho_{n}\|_{L^{\b}_{x}}^{\b}\big)\|\mathrm{div}\,\psi\|_{L^{\infty}_{x}}\,\md s\leq C\xi,
\end{align*}
\begin{align*}
\ve \int_{t}^{t+\xi}\int_{\mathcal{O}} \na \rho_{n} \cdot \na u_{n}\cdot \psi\, \md x\md s
\leq C\ve \|\na \rho_{n}\|_{L^{q}_{t}L^{2}_{x}}\|\na u_{n}\|_{L^{2}_{t}L^{2}_{x}}\xi^{\frac{1}{2}-\frac{1}{q}}
  \leq C\ve \xi^{\frac{1}{2}-\frac{1}{q}} \quad \mbox{for $q>2$},
\end{align*}
\begin{align*}
&\int_{t}^{t+\xi}\int_{\mathcal{O}}\na \cdot \mathrm{F}(Q_{n})\mathrm{I}_{3}
\cdot \psi \,\md x\md s
=-\frac{1}{2}\int_{0}^{t+\xi}\int_{\mathcal{O}}\big(|(Q_n, \na Q_{n})|^2
+\frac{c_{*}}{2}|Q_{n}|^{4}\big)\mathrm{div}\, \psi \,\md x \md s\\
&\quad \leq C\int_{t}^{t+\xi}\big(\|(Q_n, \na Q_{n})\|_{L^{2}_{x}}^{2}
+\|Q_{n}\|_{L^{4}_{x}}^{4}\big)\|\mathrm{div}\,\psi \|_{L^{\infty}_{x}}\,\md s
\leq
C\xi,\\
&\int_{t}^{t+\xi}\int_{\mathcal{O}}\na \cdot\big(-\na Q_{n}\odot \na Q_{n}+ Q_{n}\D Q_{n}-\D Q_{n}Q_{n}+\sigma_{*}c^{2}_{n} Q_{n}\big)\cdot \psi \,\md x\md s\\
&\quad \leq C\int_{t}^{t+\xi}\big((\|\na Q_{n}\|_{L^{2}_{x}}^{2}+\|Q_{n}\|_{L^{2}_{x}}\|\D Q_{n}\|_{L^{2}_{x}})\|\na \psi \|_{L^{\infty}_{x}}
  +\|c_{n}\|_{L^{\infty}}^{2}\|Q_{n}\|_{L^{2}_{x}}\|\na \psi \|_{L^{2}_{x}}\big) \,\md s\\
&\quad \leq C\big(\xi +\sqrt{\xi}\big).
\end{align*}
Together with the fact that $\rho_{n}u_{n}$ is uniformly bounded in $L^{\infty}(0, T; L^{\frac{2\gamma}{\gamma +1}}(\mathcal{O}))$
with respect to $n$ and $X_{n}$ is dense in $L^{\frac{\gamma-1}{2\gamma}}(\mathcal{O})$,
we conclude by Corollary \ref{suff-cond-C-Weak-star} that
\be\label{rho-u-n-conv-strong}
\rho_{n}u_{n}\to \rho u \qquad\,\, \mbox{in} \,\, C([0, T]; L_{\rm weak}^{\frac{2\gamma}{\gamma +1}}(\mathcal{O})) \,\,\, \mbox{ as }\; n\to \infty.
\en
Since $\frac{2\gamma}{\gamma+1}>\frac{6}{5}$ ($\gamma >\frac{3}{2}$), by Proposition \ref{conv-h-1}, \eqref{rho-u-n-conv-strong} yields
$$
\rho_{n}u_{n}\to \rho u \qquad \mbox{in}\,\,\, C([0, T]; H^{-1}(\mathcal{O})).
$$
This, combined with (\ref{u-n-to-u-h1}), yields
\be\label{rho-u-u-n-conv}
\rho_{n}u_{n}\otimes u_{n} \to \rho u\otimes u \qquad\,\, \mbox{in}\,\,\, \mathcal{D}^{'}(\mathcal{O}_{T}).
\en

Next, let us elaborate on the convergence result for the term: $\na u_{n}\cdot \na \rho_{n}$.
We multiply (\ref{rho-FG}) by $\rho_{n}$ and integrate by parts to obtain
\be\label{rho-n-L2}
\|\rho_{n}(t, \cdot)\|_{L^{2}_{x}}^{2}+2\ve \int_{0}^{t}\|\na \rho_{n}(t, \cdot)\|_{L^{2}_{x}}^{2}\,\md s
=-\int_{0}^{t}\int_{\mathcal{O}}\mathrm{div}\, u_{n}\,|\rho_{n}|^{2}\,\md x\md s +\|\rho_{0}\|_{L^{2}_{x}}^{2}.
\en
By Lemma \ref{strong-conv-rho-n}, we know that the limit function $\rho$ also satisfies (\ref{rho-FG}).
Applying the same argument to $\rho$ as above, we have
\be\label{rho-L2}
\|\rho(t, \cdot)\|_{L^{2}_{x}}^{2}+2\ve \int_{0}^{t}\|\na \rho(t, \cdot)\|_{L^{2}_{x}}^{2}\,\md s
=-\int_{0}^{t}\int_{\mathcal{O}}\mathrm{div}\, u|\rho|^{2}\,\md x\md s +\|\rho_{0}\|_{L^{2}_{x}}^{2}.
\en
Differentiating (\ref{rho-n-L2}) with respect to $t$,
we use (\ref{rho-n-b}), (\ref{u-n-h1}), and Lemma \ref{strong-conv-rho-n} to obtain
\begin{align*}
\frac{\md}{\md t}\|\rho_{n}(t, \cdot)\|_{L^{2}_{x}}^{2}
=-2\ve \|\na \rho_{n}(t, \cdot)\|_{L^{2}_{x}}^{2}-\int_{\mathcal{O}}\mathrm{div}\, u_{n}|\rho_{n}|^{2}\,\md x
  +\|\rho_{0}\|_{L^{2}_{x}}^{2} \in L^{q}(0, T),\;\;  1<q<2,
\end{align*}
which implies that $\|\rho_{n}(t, \cdot)\|_{L^{2}_{x}}^{2}$ is equi-continuous.
Then we conclude that $\|\rho_{n}(t, \cdot)\|_{L^{2}_{x}}^{2}$ converges in $C([0, T])$ by the Arzela-Ascoli theorem.
Moreover, from (\ref{rho-n-to-rho-l4}), (\ref{u-n-to-u-h1}), (\ref{rho-n-L2})--(\ref{rho-L2}),
and Lemma \ref{strong-conv-rho-n}, we have
\begin{align*}
&\|\rho_{n}(t, \cdot)\|_{L^{2}_{x}}\to \|\rho(t, \cdot)\|_{L^{2}_{x}} \qquad \mbox{for any}\; t, \\
&\|\na \rho_{n}\|_{L^{2}(\mathcal{O}_{T})}\to \|\na \rho\|_{L^{2}(\mathcal{O}_{T})}.
\end{align*}
Since $\na \rho_{n}\rightharpoonup \na \rho$, it yields
$$
\na \rho_{n}\to \na \rho \qquad \mbox{in}\; L^{2}(\mathcal{O}_{T}).
$$
Then we have
$$
\na \rho_{n}\cdot \na u_{n}\to \na \rho \cdot \na u \qquad \mbox{in} \;\mathcal{D}^{'}(\mathcal{O}_{T}).
$$
In addition, by \eqref{c-n-conv}--\eqref{Q-n-conv}, we have
\begin{align*}
&\int_{0}^{t}\int_{\mathcal{O}}\na \cdot \big(\mathrm{F}(Q_{n})\mathrm{I}_{3}-\na Q_{n}\odot \na Q_{n}
+ Q_{n}\D Q_{n}-\D Q_{n}Q_{n}+\sigma_{*} c_{n}^{2}Q_{n}\big)\cdot \psi \,\md x\md s\\
&\,\,\to \int_{0}^{t}\int_{\mathcal{O}}\na \cdot \big(\mathrm{F}(Q)\mathrm{I}_{3}-\na Q\odot \na Q
+ Q\D Q-\D QQ+\sigma_{*} c^{2}Q\big)\cdot \psi \,\md x\md s.
\end{align*}
Then we can pass to the limit  in equation (\ref{variation}) as $n\to \infty$.
We deduce that the limit function $(c, \rho, u, Q)$ is a weak solution of problem (\ref{c-FG})$-$(\ref{B-C-uQ}).
Finally, let us summarize the results in this section in the following:

\begin{Proposition}\label{existence-FG}
Suppose $\b >\max\{4, \gamma\}$. Then there exists a weak solution $(c, \rho, u, Q)$ of
problem \eqref{c-FG}$-$\eqref{B-C-uQ} with the same estimates as in Lemma {\rm \ref{FG-n-estimates}}
and $Q\in S^{3}_{0}$ a.e. in $\mathcal{O}_{T}$. Moreover, the energy inequality 
\eqref{energy-estimate-n-pre} and estimates \eqref{c-n}$-$\eqref{na-Q-n} hold for $(c, \rho, u, Q)$.
Finally, we can find $r>1$ such that $\rho_{t}, \D \rho \in L^{r}(\mathcal{O}_{T})$,
and equation \eqref{rho-FG} is satisfied a.e. in $\mathcal{O}_{T}$.
\end{Proposition}

\section{The Vanishing Artificial Viscosity Limit}

In this section, we let $\ve\to 0$ in (\ref{c-FG})$-$(\ref{Q-FG}).
We denote $(c_{\ve}, \rho_{\ve}, u_{\ve}, Q_{\ve})$ the solution of problem (\ref{c-FG})$-$(\ref{B-C-uQ}),
which we have obtained in Proposition \ref{existence-FG}.
However, unlike the previous step, we do not have the higher integrability of the density
as in Lemma \ref{strong-conv-rho-n}.
The boundedness of $\rho_{\ve}$ in $L^{\infty}(0, T; L^{\gamma}(\mathcal{O})\cap L^{\b}(\mathcal{O}))$
can guarantee only that $\rho_{\ve}^{\b}$ converges to a Radon measure as $\ve \to 0$,
which is not easy to deal with.
Thus, it is essential to obtain the strong compactness of $\rho_{\ve}$ in $L^{1}(\mathcal{O}_{T})$.
First, we introduce the useful operator $\mathcal{B}$ related to the equation: $\mathrm{div}\, v=f$.
See \cite{B-1980,B-S-1990,G-1994} for the construction and the proof of the following properties of the operator $\mathcal{B}$:
For the problem
\be\label{div-v=f}
\mathrm{div}\, v=f, \qquad v|_{\del \mathcal{O}}=0,
\en
there exists a linear operator $\mathcal{B}=[B_{1}, B_{2}, B_{3}]$ with the following properties:
\begin{enumerate}
\item[(i)] $\mathcal{B}: \{ f\in L^{p}(\mathcal{O}) \,:\, \int_{\mathcal{O}}f\, \md x=0\}\mapsto \big(W^{1, p}_{0}(\mathcal{O})\big)^{3}$
is a bounded linear operator such that, for any $1<p<\infty$,
\be\label{B-W-1-p}
\|\mathcal B [f]\|_{W^{1, p}_{0}(\mathcal{O})}\leq C(p)\|f\|_{L^{p}(\mathcal{O})};
\en
\item[(ii)] $v=\mathcal{B} [f]$ is a solution of problem (\ref{div-v=f});
\item[(iii)] If there is a vector function $\mathbf{g}\in \big(L^{r}(\mathcal{O})\big)^{3}$
with $\mathbf{g}\cdot \vec{n}|_{\del \mathcal{O}}=0$, then
\be\label{B-div-r}
\|\mathcal B [\mathrm{div}\, \mathbf{g}]\|_{L^{r}(\mathcal{O})}\leq C(p)\|\mathbf{g}\|_{L^{r}(\mathcal{O})},
\en
where $r\in (1, \infty)$ is arbitrary.
\end{enumerate}

\subsection{Estimates of the density independent of $\epsilon$}\label{density-ind-ve}

We take the quantities:
\be\label{test-function-ve}
\psi(t)\mathcal{B}[\rho_{\ve}-\bar{m}],
\en
with $\psi \in \mathcal D (0, T)$, $0\leq \psi \leq 1$, and $\bar{m}=\frac{1}{|\mathcal{O}|}\int_{\mathcal{O}}\rho_{0}(x)\md x$,
as test functions for (\ref{u-FG}). Note that $\bar{m}$ is a constant such that this test function is well defined. We have the following result.

\begin{Lemma}\label{lemma-density}
Assume that $(c_{\ve}, \rho_{\ve}, u_{\ve}, Q_{\ve})$ is the solution of problem \eqref{c-FG}$-$\eqref{B-C-uQ}
constructed  in Proposition {\rm \ref{existence-FG}}. Then
$$
\|\rho_{\ve}\|^{\gamma +1}_{L^{\gamma +1}(\mathcal{O}_{T})}+\delta \|\rho_{\ve}\|^{\b +1}_{L^{\b +1}(\mathcal{O}_{T})}\leq C,
$$
with $C$
independent of $\ve$.
\end{Lemma}

\begin{proof} The proof is similar to Lemma 3.1 in \cite{F-N-P-2001}.
Let us apply the test function (\ref{test-function-ve}) to (\ref{u-FG}).
Then, by a direct calculation, we have
\begin{align}
\int_{0}^{T}&\int_{\mathcal{O}}\psi \big(\rho_{\ve}^{\gamma +1}+\delta \rho_{\ve}^{\b +1}\big)\,\md x \md t\nonumber\\
=&\, \bar{m}\int_{0}^{T}\psi \int_{\mathcal{O}}(\rho_{\ve}^{\gamma}+\delta \rho_{\ve}^{\b})\,\md x \md t
  +(\mu +\nu)\int_{0}^{T}\psi \int_{\mathcal{O}}(\rho_{\ve}-\bar{m})\mathrm{div}\, u_{\ve}\,\md x \md t \nonumber\\
&\,-\int_{0}^{T}\psi_{t} \int_{\mathcal{O}}\rho_{\ve}u_{\ve}\cdot \mathcal B [\rho_{\ve}-\bar{m}]\,\md x \md t
  +\mu \int_{0}^{T}\psi \int_{\mathcal{O}}\del_{j}u_{\ve}^{i}\del_{j} B_{i} [\rho_{\ve}-\bar{m}]\,\md x \md t\nonumber \\
&\, -\int_{0}^{T}\psi \int_{\mathcal{O}} \rho_{\ve}u_{\ve}^{i}u_{\ve}^{j}\del_{j} B_{i} [\rho_{\ve}-\bar{m}]\,\md x \md t
  -\ve \int_{0}^{T}\psi \int_{\mathcal{O}}\rho_{\ve}u_{\ve}\cdot \mathcal{B}(\D \rho_{\ve})\,\md x \md t\nonumber\\
&\,+\int_{0}^{T}\psi \int_{\mathcal{O}}\rho_{\ve}u_{\ve}\cdot \mathcal B [\mathrm{div}\,(\rho_{\ve}u_{\ve})]\,\md x \md t
  +\ve \int_{0}^{T}\psi \int_{\mathcal{O}}\del_{j} u^{i}_{\ve} \del_{j}\rho_{\ve} B_{i} [\rho_{\ve}-\bar{m}]\, \md x \md t\nonumber\\
&\,+\int_{0}^{T}\psi \int_{\mathcal{O}}\na \cdot (\na Q_{\ve}\otimes \na Q_{\ve}-\mathrm{F}(Q_{\ve})\mathrm{I}_{3})\cdot  \mathcal B [\rho_{\ve}-\bar{m}]\,\md x \md t\nonumber\\
&\,
+\sigma_{*}\int_{0}^{T}\psi \int_{\mathcal{O}}c_{\ve}^{2}Q_{\ve}: \na \mathcal B [\rho_{\ve}-\bar{m}]\,\md x \md t\nonumber\\
&\,-\int_{0}^{T}\psi \int_{\mathcal{O}}\na \cdot (Q_{\ve}\D Q_{\ve}-\D Q_{\ve}Q_{\ve})\cdot \mathcal B [\rho_{\ve}-\bar{m}]\,\md x \md t\nonumber\\
=&\, \sum_{i=1}^{11}\i_{i}.
\label{rho-ve-higher-integrability}
\end{align}
Next, we estimate the terms on the right-hand side of the above equality by using the boundedness
of solution $(c_{\ve}, \rho_{\ve}, u_{\ve}, Q_{\ve})$ obtained in Proposition \ref{existence-FG},
in which the universal constant $C>0$ is independent of $\ve$:
\begin{align*}
|\i_{1}|&=\Big|\bar{m}\int_{0}^{T}\psi \int_{\mathcal{O}}(\rho_{\ve}^{\gamma}+\delta \rho_{\ve}^{\b})\,\md x \md t\Big|
\leq C\big(\|\rho_{\ve}\|_{L^{\infty}_{t}L^{\gamma}_{x}}^{\gamma}+\delta \|\rho_{\ve}\|_{L^{\infty}_{t}L^{\b}_{x}}^{\b}\big)\leq C.\\
|\i_{2}|&=(\mu +\nu)\big|\int_{0}^{T}\psi \int_{\mathcal{O}}(\rho_{\ve}-\bar{m})\mathrm{div}\, u_{\ve}\,\md x \md t\big|\\
&\leq  C\sqrt{T}\big(\|\rho_{\ve}\|_{L^{\infty}_{t}L^{2}_{x}}+\bar{m}|\mathcal{O}|^{\frac{1}{2}}\big)\|\mathrm{div} \, u_{\ve}\|_{L^{2}_{t, x}}\leq C.
\end{align*}
\begin{align*}
|\i_{3}|&=\Big|\int_{0}^{T}\psi_{t} \int_{\mathcal{O}}\rho_{\ve}u_{\ve}\cdot \mathcal B [\rho_{\ve}-\bar{m}]\,\md x \md t\Big|
\leq C\int_{0}^{T} \|\sqrt{\rho_{\ve}}\|_{L^{4}_{x}} \|\sqrt{\rho_{\ve}}u_{\ve}\|_{L^{2}_{x}} \|\mathcal B [\rho_{\ve}-\bar{m}]\|_{L^{4}_{x}}\,\md t\\
&\leq  C\int_{0}^{T} \|\rho_{\ve}\|_{L^{2}_{x}}^{\frac{1}{2}} \|\sqrt{\rho_{\ve}}u_{\ve}\|_{L^{2}_{x}} \|\rho_{\ve}-\bar{m}\|_{L^{4}_{x}}\md t\leq C(\delta, T).\\
|\i_{4}|&=\mu \Big|\int_{0}^{T}\psi \int_{\mathcal{O}}\del_{j}u_{\ve}^{i}\del_{j}B_{i} [\rho_{\ve}-\bar{m}]\,\md x \md t\Big|
  \leq  \mu \int_{0}^{T}\|\na u_{\ve}\|_{L^{2}_{x}}\|\na \mathcal B [\rho_{\ve}-\bar{m}]\|_{L^{2}_{x}}\,\md t\\
 &\leq \mu \int_{0}^{T}\|\na u_{\ve}\|_{L^{2}_{x}}\| \rho_{\ve}-\bar{m}\|_{L^{2}_{x}}\,\md t\leq C.\\
|\i_{5}|&=\Big|\int_{0}^{T}\psi \int_{\mathcal{O}} \rho_{\ve}u_{\ve}^{i}u_{\ve}^{j}\del_{j} B_{i} [\rho_{\ve}-\bar{m}]\,\md x \md t\Big|
  \leq C\int_{0}^{T}\|\rho_{\ve}\|_{L^{3}_{x}}\|u_{\ve}\|^{2}_{L^{6}_{x}}\| \na \mathcal B [\rho_{\ve}-\bar{m}]\|_{L^{3}_{x}}\,\md t\\
&\leq C\int_{0}^{T}\|\rho_{\ve}\|_{L^{3}_{x}}\|u_{\ve}\|^{2}_{L^{6}_{x}}\|\rho_{\ve}-\bar{m}\|_{L^{3}_{x}}\,\md t\leq C.\\
|\i_{6}|&=\ve\Big|\int_{0}^{T}\psi \int_{\mathcal{O}}\rho_{\ve}u_{\ve}\cdot \mathcal{B}[\D \rho_{\ve}]\,\md x \md t\Big|
\leq \ve \int_{0}^{T}\|\rho_{\ve}\|_{L^{3}_{x}}\|u_{\ve}\|_{L^{6}_{x}}\| \mathcal B [\D \rho_{\ve}]\|_{L^{2}_{x}}\,\md t\\
&\leq C\ve \int_{0}^{T}\|\rho_{\ve}\|_{L^{3}_{x}}\|u_{\ve}\|_{L^{6}_{x}}\|\na \rho_{\ve}\|_{L^{2}_{x}}\,\md t\leq C\quad\quad \mbox{for $\ve<1$}.\\
|\i_{7}|&=\Big|\int_{0}^{T}\psi \int_{\mathcal{O}}\rho_{\ve}u_{\ve}\cdot \mathcal B [\mathrm{div}\,(\rho_{\ve}u_{\ve})]\,\md x \md t\Big|
  \leq \int_{0}^{T}\|\rho_{\ve}\|_{L^{3}_{x}}\|u_{\ve}\|_{L^{6}_{x}}\| \mathcal B [\mathrm{div}\, (\rho_{\ve}u_{\ve})]\|_{L^{2}_{x}}\,\md t\\
&\leq  \int_{0}^{T}\|\rho_{\ve}\|_{L^{3}_{x}}\|u_{\ve}\|_{L^{6}_{x}}\|\rho_{\ve}u_{\ve}\|_{L^{2}_{x}}\,\md t
  \leq C \int_{0}^{T}\|\rho_{\ve}\|_{L^{3}_{x}}^{2}\|u_{\ve}\|_{L^{6}_{x}}^{2}\,\md t\leq C.
\end{align*}
Since $\b > 4$, by using the Sobolev embedding (Lemma \ref{gn-inequality} in Appendix A), we have
\be\label{B-l-infty}
\begin{split}
\|B[\rho_{\ve}-\bar{m}]\|_{L^{\infty}_{x}}&\leq C_{1}\|\na B[\rho_{\ve}-\bar{m}]\|_{L^{\b}_{x}}^{\frac{3}{\b}}\|B[\rho_{\ve}-\bar{m}]\|_{L^{\b}_{x}}^{1-\frac{3}{\b}}+C_{2}\|B[\rho_{\ve}-\bar{m}]\|_{L^{\b}_{x}}\\
&\leq C\|\rho_{\ve}-\bar{m}\|_{L^{\b}_{x}}.
\end{split}
\en
Then we have
\begin{align*}
|\i_{8}|&=\ve \Big|\int_{0}^{T}\psi \int_{\mathcal{O}}\na u_{\ve}\cdot \mathcal B [\rho_{\ve}-\bar{m}] \na \rho_{\ve}\,\md x \md t\Big|
  \leq\ve \int_{0}^{T}\|\na u_{\ve}\|_{L^{2}_{x}}\|\na \rho_{\ve}\|_{L^{2}_{x}}\| \mathcal B [\rho_{\ve}-\bar{m}]\|_{L^{\infty}_{x}}\,\md t\\
&\leq \ve \int_{0}^{T}\|\na u_{\ve}\|_{L^{2}_{x}}\|\na \rho_{\ve}\|_{L^{2}_{x}}\|\rho_{\ve}-\bar{m}\|_{L^{\b}_{x}} \,\md t\leq C,
\end{align*}
\begin{align*}
|\i_{9}|&=\Big|\int_{0}^{T}\psi \int_{\mathcal{O}}\na \cdot (\na Q_{\ve}\otimes \na Q_{\ve}
   -\mathrm{F}(Q_{\ve})\mathrm{I}_{3})\cdot \mathcal B [\rho_{\ve}-\bar{m}]\,\md x \,\md t\Big|\\
&\leq C \int_{0}^{T}\Big(\|\na Q_{\ve}\|_{L^{\frac{10}{3}}_{x}}^{2}\|\na \mathcal B [\rho_{\ve}-\bar{m}]\|_{L^{\frac{5}{2}}_{x}}
  +|\int_{\mathcal{O}}\mathrm{F}(Q_{\ve}) \mathrm{div}\, \mathcal B [\rho_{\ve}-\bar{m}]\,\md x |\Big)\,\md t\\
&\leq C \int_{0}^{T}\Big(\|\na Q_{\ve}\|_{L^{\frac{10}{3}}_{x}}^{2}\|\na \mathcal B [\rho_{\ve}-\bar{m}]\|_{L^{\frac{5}{2}}_{x}}
    +| \int_{\mathcal{O}}\mathrm{F}(Q_{\ve})(\rho_{\ve}-\bar{m})\md x |\Big)\,\md t\\
&\leq  C\int_{0}^{T}\big(\|\na Q_{\ve}\|_{L^{\frac{10}{3}}_{x}}^{2}\|\rho_{\ve}-\bar{m}\|_{L^{\frac{5}{2}}_{x}}
     + (\|Q_{\ve}\|_{L^{5}_{x}}^{2}+\|Q_{\ve}\|_{L^{10}_{x}}^{4})\|\rho_{\ve}-\bar{m}\|_{L^{\frac{5}{3}}_{x}}\big)\md t\leq C,
\end{align*}
\begin{align*}
|\i_{10}|&=\sigma_{*} \Big|\int_{0}^{T}\psi \int_{\mathcal{O}}c_{\ve}^{2}Q_{\ve}\cdot \na \mathcal B [\rho_{\ve}-\bar{m}]\,\md x \md t\Big|\\
&\leq C\|c_{\ve}\|_{L^{\infty}_{t, x}}^{2}\int_{0}^{T}\|Q_{\ve}\|_{L^{2}_{x}}\|\na \mathcal B [\rho_{\ve}-\bar{m}]\|_{L^{2}_{x}}\,\md t \leq C,
\end{align*}
\begin{align*}
|\i_{11}|&=\Big|\int_{0}^{T}\psi \int_{\mathcal{O}}\na \cdot (Q_{\ve}\D Q_{\ve}-\D Q_{\ve}Q_{\ve})\cdot \mathcal B [\rho_{\ve}-\bar{m}]\,\md x \md t\Big|\\
&\leq C\int_{0}^{T}\|Q_{\ve}\|_{L^{4}_{x}}\|\D Q_{\ve}\|_{L^{2}_{x}}\|\na \mathcal B [\rho_{\ve}-\bar{m}]\|_{L^{4}_{x}}\,\md t\\
&\leq C\int_{0}^{T}\|Q_{\ve}\|_{L^{4}_{x}}\|\D Q_{\ve}\|_{L^{2}_{x}}\|\rho_{\ve}-\bar{m}\|_{L^{4}_{x}}\,\md t\leq C.
\end{align*}
Combining all the above estimates together, we obtain our desired result.
\end{proof}

\begin{Remark}
Lemma \ref{lemma-density} implies that $\rho_{\ve}$ has higher integrability,
which provides the weak convergence result for the pressure, {\it i.e.},
$P_{\ve}=\rho_{\ve}^{\gamma}+\dl \rho_{\ve}^{\b}\rightharpoonup p$
in $L^{\frac{\b+1}{\b}}(\mathcal{O}_{T})$.
\end{Remark}

\subsection{Limit passage of $\ve \to 0$}

In this subsection, we fix parameter $\dl$, and pass to the limit $\ve \to 0$
in equations (\ref{c-FG})$-$(\ref{Q-FG}).
To begin with, similarly to \eqref{c-n-conv}--\eqref{Q-n-conv} in \S \ref{1st-approx},
we have
\begin{align}
&c_{\ve}\rightharpoonup c\quad  \mbox{in} \,\, L^{2}(0, T; H^{1}(\mathcal{O})),
\qquad c_{\ve}\to c \quad\,\,\,\, \mbox{in}\,\, L^{2}(0, T; L^{2}(\mathcal{O})), \label{c-ve-conv}\\
&Q_{\ve}\rightharpoonup Q \quad \mbox{in} \,\, L^{2}(0, T; H^{2}(\mathcal{O})),\quad\,\,
Q_{\ve}\to Q \quad \mbox{in} \,\, L^{2}(0, T; H^{1}(\mathcal{O})). \label{Q-ve-conv}
\end{align}

From the boundedness of $\rho_{\ve}$ in $L^{\b+1}(\mathcal{O}_{T})$,
$\sqrt{\ve} \na \rho_{\ve}$ in $L^{2}(\mathcal{O}_{T})$, and $u_{\ve}$ in $L^{2}(0, T; H_{0}^{1}(\mathcal{O}))$, we know that
\begin{align}
&\rho_{\ve}\rightharpoonup \rho \qquad \mbox{ in } L^{\b+1}(\mathcal{O}_{T}), \label{rho-ve-conv}\\
&u_{\ve}\rightharpoonup u \qquad \mbox{ in } L^{2}(0, T; H_{0}^{1}(\mathcal{O})),\label{u-ve-conv}\\
&\ve \na \rho_{\ve}\cdot \na u_{\ve}\to 0 \qquad \mbox{ in } L^{1}(\mathcal{O}_{T}), \\
&\ve \D \rho_{\ve}\to 0   \qquad \mbox{ in } L^{2}(0, T; H^{-1}(\mathcal{O})).
\end{align}
Moreover, we can also obtain the following convergence results as in \S \ref{1st-approx}:
\begin{align}
&\rho_{\ve}\to \rho \qquad \mbox{ in } C([0, T]; L^{\b}_{\rm weak}(\mathcal{O})),\\
&\rho_{\ve}\to \rho \qquad \mbox{ in } C([0, T]; L^{\gamma}_{\rm weak}(\mathcal{O})),\\
&\rho_{\ve}u_{\ve}\to \rho u \qquad \mbox{ in } C([0, T]; L_{\rm weak}^{\frac{2\gamma}{\gamma +1}}(\mathcal{O})), \label{rho-u-ve-conv}\\
&\rho_{\ve}u_{\ve}\otimes u_{\ve} \to \rho u\otimes u \qquad \mbox{ in } \mathcal D^{'} (\mathcal{O}_{T}).\label{rho-u-u-conv-D}
\end{align}

Finally, we conclude that the limit vector function $(c, \rho, u, Q)$ satisfies the following equations in $\mathcal D^{'}(\mathcal{O}_{T})$:
\begin{align}
&\del_{t}c+u\cdot \na c=D_{0}c, \label{c-dl}\\
&\del_{t}\rho+\na \cdot (\rho u)=0, \label{rho-dl}\\
&\del_{t}(\rho u)+\na \cdot (\rho u\otimes u)+\na p =\mu \D u+(\nu +\mu)\na \mathrm{div}\, u+\na \cdot (\mathrm{F}(Q)\mathrm{I}_{3}-\na Q\odot \na Q) \\
&\qquad\qquad\qquad\qquad\qquad\qquad\quad  +\nabla \cdot (Q\D Q-\D Q Q)+\sigma_{*} \na \cdot (c^{2}Q), \label{u-dl}\\
&\partial_{t}Q+(u\cdot \nabla )Q+Q\Omega-\Omega Q
=\Gamma H[Q, c], \label{Q-dl}
\end{align}
along with the initial--boundary conditions (\ref{I-C-rho-a})$-$(\ref{B-C-uQ}), with
\be\label{p-dl}
P_{\ve}=\rho_{\ve}^{\gamma}+\delta \rho_{\ve}^{\b} \rightharpoonup p \qquad   \mbox{   in } L^{\frac{\b+1}{\b}}(\mathcal{O}_{T}),
\en
for $\b>\mathrm{max}\{ 4, \gamma \}$.

In the next step, we show that $p=\rho^{\gamma}+\delta \rho^{\b}$,
which is equivalent to the strong convergence of $\rho_{\ve}$ in $L^{1}(\mathcal{O}_{T})$.

\subsection{The effective viscous flux}
The quantity, $\mathfrak{E}_{\ve}:=\rho_{\ve}^{\gamma}+\delta \rho_{\ve}^{\b}-(\nu+2\mu)\mathrm{div}\, u_{\ve}$,
is usually called the effective viscous flux, and its corresponding weak convergence limit
is $\mathfrak{E}:=p-(\nu+2\mu)\mathrm{div}\, u$ . The properties of $\mathfrak{E}_{\ve}$ ({\it cf.}  \cite{H-1995,L-1998,S-1991})
play an important role in our problem.
We introduce the following operator $\mathcal{A}=(\mathcal{A}_{1}, \mathcal{A}_{2}, \mathcal{A}_{3}): \R^{3}\to \R^{3}$:
\begin{align*}
\mathcal A_{j} [v]=\del_{j}\D^{-1}v
\end{align*}
with the Fourier transform:
\begin{align*}
\mathcal F (\mathcal A_{j})(\xi)=-\frac{i \xi_{j}}{|\xi|^{2}},
\end{align*}
and enjoying the following properties:
\begin{align}
&\mathrm{div}\, \mathcal A [v]=v, \qquad \D \mathcal{A}_{i}[v]=\del_{i}v,\label{4.22}\\
&\|\mathcal A_{i} [v]\|_{W^{1, p}(\mathcal{O})}\leq C(p, \mathcal{O})\|v\|_{L^{p}(\R^{3})} \qquad\,\,\, \mbox{for $1<p<\infty$},\label{a-w1p}\\
&\|\mathcal A_{i} [v]\|_{L^{q}(\mathcal{O})}\leq C(p, q, \mathcal{O})\|v\|_{L^{p}(\R^{3})} \qquad \mbox{ if }p<\infty \mbox{ and } \frac{1}{q}\geq \frac{1}{p}-\frac{1}{3},\\
& \|\mathcal A_{i} [v]\|_{L^{\infty}(\mathcal{O})}\leq C(p, \mathcal{O})\|v\|_{L^{p}(\R^{3})} \qquad\,\,\,\,\,\, \mbox{for all $p>3$}.
\end{align}

\begin{Lemma}\label{2}
Let $(c_{\ve}, \rho_{\ve}, u_{\ve}, Q_{\ve})$ be a sequence of solutions constructed in Proposition {\rm \ref{existence-FG}},
and let $(c, \rho, u, Q, p)$ be the limits satisfying \eqref{c-dl}$-$\eqref{p-dl}.
Then
$$
\lim_{\ve\to 0^{+}}\int_{0}^{T}\psi \int_{\mathcal{O}} \phi\,\mathfrak{E}_{\ve}\rho_{\ve}\, \md x \md t
=\int_{0}^{T}\psi \int_{\mathcal{O}} \phi\,\mathfrak{E} \rho\, \md x \md t
$$
for any $\psi \in \mathcal D (0, T)$ and $\phi \in \mathcal D (\mathcal{O})$.
\end{Lemma}

\begin{proof}
We prove this lemma based on the div-curl lemma of compensated compactness.
By Proposition \ref{existence-FG},
we know that $(\rho_{\ve}, u_{\ve})$ satisfies (\ref{rho-FG}) {\it a.e.} on $\mathcal{O}_{T}$
with boundary condition \eqref{B-C-rho}. If $(\rho_{\ve}, u_{\ve})$ are extended to
be zero outside $\mathcal{O}$, then it satisfies
\begin{align*}
\del_{t}\rho_{\ve}+\na \cdot  (\rho_{\ve}u_{\ve})=\ve \mathrm{div}\,(\mathrm{1}_{\mathcal{O}}\na \rho_{\ve})\qquad
\mbox{ in } \mathcal{D}^{'} ((0, T)\times \R^{3}),
\end{align*}
with $\mathrm{1}_{\mathcal{O}}$ the characteristic function on $\mathcal{O}$.
Consider the following test function to (\ref{u-FG}) as
\be\label{test-f-ve-a}
\va_{\ve}(t, x)=\psi(t)\phi(x)\mathcal A [\rho_{\ve}],
\en
with $\psi \in \mathcal D (0, T)$ and $\phi \in \mathcal D(\mathcal{O})$.
Similarly to (\ref{rho-ve-higher-integrability}),
we apply the test function \eqref{test-f-ve-a} to  (\ref{u-FG}).
By a direct calculation, we have
\begin{align}
&\int_{0}^{T}\psi \int_{\mathcal{O}}\phi \big(\rho_{\ve}^{\gamma +1}+\delta \rho_{\ve}^{\b +1}-(\nu +2\mu)\mathrm{div}\, u_{\ve}\big)\rho_{\ve}\,\md x \md t\nonumber\\
&=-\int_{0}^{T}\psi \int_{\mathcal{O}}(\rho_{\ve}^{\gamma}+\delta \rho_{\ve}^{\b})\del_{i}\phi \,\mathcal A_{i} [\rho_{\ve}]\,\md x \md t
 +(\mu +\nu)\int_{0}^{T}\psi \int_{\mathcal{O}}\mathrm{div}\, u_{\ve}\,\del_{i}\phi \mathcal A_{i} [\rho_{\ve}]\,\md x \md t\nonumber \\
& \quad+\mu \int_{0}^{T}\psi \int_{\mathcal{O}}\del_{j}u_{\ve}^{i}\del_{j}\phi \mathcal A_{i} [\rho_{\ve}]\,\md x \md t
  -\mu \int_{0}^{T}\psi \int_{\mathcal{O}}u_{\ve}^{i}\del_{j}\phi \del_{j}\mathcal A_{i} [\rho_{\ve}]\,\md x \md t\nonumber\\
&\quad +\mu \int_{0}^{T}\psi \int_{\mathcal{O}}u_{\ve}^{i}\del_{i}\phi \rho_{\ve}\, \md x \md t
  -\ve \int_{0}^{T}\psi \int_{\mathcal{O}}\phi \rho_{\ve}u_{\ve}^{i} \mathcal{A}_{i}\, \mathrm{div}\,(\mathrm{1}_{\O}\na \rho_{\ve})\,\md x \md t\nonumber \\
&\quad + \int_{0}^{T}\psi \int_{\mathcal{O}}\phi \rho_{\ve}u_{\ve}^{i}\big(\mathcal A_{i} [\mathrm{div}(\rho_{\ve}u_{\ve})]-u_{\ve}^{j}\del_{j}\mathcal A_{i} [\rho_{\ve}]\big)\, \md x \md t\nonumber \\
&\quad -\int_{0}^{T}\psi_{t} \int_{\mathcal{O}}\phi \rho_{\ve}u^{i}_{\ve}\mathcal A_{i} [\rho_{\ve}]\,\md x \md t
  -\int_{0}^{T}\psi \int_{\mathcal{O}} \rho_{\ve}u_{\ve}^{i}u_{\ve}^{j}\del_{j}\phi \mathcal A_{i} [\rho_{\ve}]\,\md x \md t\nonumber \\
&\quad +\ve \int_{0}^{T}\psi \int_{\mathcal{O}}\del_{j} u_{\ve}^{i}\del_{j}\rho_{\ve}\phi \mathcal A_{i} [\rho_{\ve}] \,\md x \md t
  +\sigma_{*} \int_{0}^{T}\psi \int_{\mathcal{O}}c_{\ve}^{2}Q^{ij}_{\ve} \big( \del_{j}\phi \mathcal A_{i} [\rho_{\ve}]+\phi \del_{j}\mathcal A_{i}[\rho_{\ve}]\big)\,\md x \md t\nonumber\\
&\quad -\int_{0}^{T}\psi \int_{\mathcal{O}}\big(\del_{i} Q^{kl}_{\ve}\del_{j} Q^{kl}_{\ve}-\mathrm{F}(Q_{\ve})\dl_{ij}\big) \del_{j}\phi \mathcal A_{i} [\rho_{\ve}]\,\md x \md t\nonumber\\
&\quad -\int_{0}^{T}\psi \int_{\mathcal{O}}\big(\del_{i} Q^{kl}_{\ve}\del_{j} Q^{kl}_{\ve}-\mathrm{F}(Q_{\ve})\dl_{ij}\big)\phi \del_{j}\mathcal A_{i}[\rho_{\ve}]\,\md x \md t\nonumber\\
&\quad +\int_{0}^{T}\psi \int_{\mathcal{O}}\big(Q^{ik}_{\ve}\D Q^{kj}_{\ve}
  -\D Q^{ik}_{\ve}Q^{kj}_{\ve}\big)\big( \del_{j}\phi \mathcal A_{i} [\rho_{\ve}]+\phi \del_{j}\mathcal A_{i}[\rho_{\ve}]\big)\,\md x \md t\nonumber\\
&=\sum_{i=1}^{14}\i_{i}.\label{rho-ve-higher-integrability-calculation}
\end{align}

\begin{Remark}
The function $\rho_{\ve}$
extended by zero outside $\mathcal{O}$ admits the time derivative as
\begin{equation*}
\del_{t}\rho_{\ve}
=\begin{cases}
\ve \D \rho_{\ve}-\mathrm{div}(\rho_{\ve}u_{\ve}) \quad &\mbox{in}\,\, \mathcal{O},\\
0 &\mbox{in} \,\, \R^{3}\setminus \mathcal{O}.
\end{cases}
\end{equation*}
Since $u_{\ve}|_{\del \mathcal{O}}=0$, we have
\begin{align*}
\mathrm{div}(\rho_{\ve}u_{\ve})=0 \qquad \mbox{on} \,\, \R^{3}\setminus \mathcal{O}.
\end{align*}
Moreover, since $\na \rho_{\ve}\cdot \vec{n}|_{\del \mathcal{O}}=0$ and $\D \rho_{\ve}\in L^{r}(\mathcal{O}_{T})$ for some $r>1$, we have
\begin{align*}
\mathrm{div}(1_{\mathcal{O}}\na \rho_{\ve})=
\begin{cases}
\D \rho_{\ve} \quad &\mbox{in} \,\, \mathcal{O},\\
0 &\mbox{in} \,\, \R^{3}\setminus \mathcal{O}.
\end{cases}
\end{align*}
\end{Remark}

Next, we do the zero extension to the limit function $\rho$ to $\R^{3}$,
and repeat the same procedure above. This step is guaranteed by the following results from \cite{F-N-P-2001}:

\begin{Lemma}\label{1}
Assume $(\rho, u)\in L^{2}(\mathcal{O}_{T})\times L^{2}(0, T; H^{1}_{0}(\mathcal{O}))$
is a solution of \eqref{rho-dl} in $\mathcal D^{'} (\mathcal{O}_{T})$.
Then, extending $(\rho, u)$ to be zero in $\R^{3}\setminus \mathcal{O}$,
equation \eqref{rho-dl} still holds in $\mathcal D^{'} ((0, T)\times \R^{3})$.
\end{Lemma}

Here, let us apply the test function $\va=\psi(t)\phi(x)\mathcal A [\rho]$ to (\ref{u-dl}).
By a similar calculation as before, we have
\begin{align}
&\int_{0}^{T}\psi \int_{\mathcal{O}}\phi \big(p-(\nu +2\mu)\mathrm{div}\, u\big)\rho\, \md x \md t\nonumber\\
&=-\int_{0}^{T}\psi \int_{\mathcal{O}}p\del_{i}\phi \mathcal A_{i} [\rho]\,\md x \md t
  +(\mu +\nu)\int_{0}^{T}\psi \int_{\mathcal{O}}\mathrm{div}\, u\del_{i}\phi \mathcal A_{i} [\rho]\,\md x \md t\nonumber \\
& \quad +\mu \int_{0}^{T}\psi \int_{\mathcal{O}}\del_{j}u^{i}\del_{j}\phi \mathcal A_{i} [\rho]\,\md x \md t
  -\mu \int_{0}^{T}\psi \int_{\mathcal{O}}u^{i}\del_{j}\phi \del_{j}\mathcal A_{i} [\rho]\,\md x \md t\nonumber\\
&\quad +\mu \int_{0}^{T}\psi \int_{\mathcal{O}}u^{i}\del_{i}\phi \rho \,\md x \md t
  + \int_{0}^{T}\psi \int_{\mathcal{O}}\phi \rho u^{i}\big(\mathcal A_{i} [\mathrm{div}(\rho u)]-u^{j}\del_{j}\mathcal A_{i} [\rho]\big)\,\md x \md t\nonumber \\
&\quad -\int_{0}^{T}\psi_{t} \int_{\mathcal{O}}\phi \rho u^{i} \mathcal A_{i} [\rho]\,\md x \md t
  -\int_{0}^{T}\psi \int_{\mathcal{O}} \rho u^{i}u^{j}\del_{j}\phi \mathcal A_{i} [\rho]\,\md x \md t \nonumber\\
& \quad +\sigma_{*} \int_{0}^{T}\psi \int_{\mathcal{O}}c^{2}Q^{ij} \big(\del_{j}\phi \mathcal A_{i} [\rho]+\phi \del_{j}\mathcal A_{i}[\rho]\big)\,\md x \md t\nonumber\\
&\quad -\int_{0}^{T}\psi \int_{\mathcal{O}}\big(\del_{i} Q^{kl}\del_{j} Q^{kl}-\mathrm{F}(Q)\dl_{ij}\big) \del_{j}\phi \mathcal A_{i} [\rho]\,\md x \md t\nonumber\\
&\quad -\int_{0}^{T}\psi \int_{\mathcal{O}}\big(\del_{i} Q^{kl}\del_{j} Q^{kl}-\mathrm{F}(Q)\dl_{ij}\big)\phi \del_{j}\mathcal A_{i}[\rho]\,\md x \md t\nonumber\\
&\quad +\int_{0}^{T}\psi \int_{\mathcal{O}}\big(Q^{ik} \D Q^{kj}-\D Q^{ik}Q^{kj}\big)\big( \del_{j}\phi \mathcal A_{i} [\rho]+\phi \del_{j}\mathcal A_{i}[\rho]\big)\,\md x \md t\nonumber\\
&=\sum_{j=1}^{12}\j_{i}. \label{rho-higher-integrability-calculation}
\end{align}
Next, in order to prove Lemma \ref{2}, we need to show that the right-hand side
of (\ref{rho-ve-higher-integrability-calculation}) converges to the right-hand
side of (\ref{rho-higher-integrability-calculation}).
First, by the uniform bound of $\rho_{\ve}$ in $L^{\infty}(0, T; L^{\b}(\mathcal{O}))$, similarly to \eqref{rho-u-n-conv-strong}, we have
$$
\rho_{\ve}\to \rho \qquad\;  \mbox{in} \; C([0, T]; L^{\b}_{\rm weak}(\mathcal{O}))\; \;\mbox{as} \; \ve\to 0.
$$
Then, for any $\phi \in (L^{\b}(\mathcal{O}))^{*}$, we extend
$\phi$ by zero outside of $\mathcal{O}$ to obtain
\begin{align*}
&(\del_{j}\mathcal{A}_{i}[\rho_{\ve}], \phi)_{L^{\b}(\mathcal{O})\times (L^{\b}(\mathcal{O}))^{*}}
=\int_{\R^{3}}\del_{j}\mathcal{A}_{i}[\rho_{\ve}] \phi \,\md x
=\int_{\R^{3}}\frac{\hat{\rho_{\ve}}\xi_{i}\xi_{j}}{|\xi |^{2}}\hat{\phi}\,\md \xi \\
&=\int_{\R^{3}}\rho_{\ve}\del_{j}\mathcal{A}_{i}[\phi]\, \md x
=(\rho_{\ve}, \del_{j}\mathcal{A}_{i}[\phi])_{L^{\b}(\mathcal{O})\times (L^{\b}(\mathcal{O}))^{*}}\\[2mm]
&\,\longrightarrow (\rho, \del_{j}\mathcal{A}_{i}[\phi])_{L^{\b}(\mathcal{O})\times (L^{\b}(\mathcal{O}))^{*}}
=(\del_{j}\mathcal{A}_{i}[\rho], \phi)_{L^{\b}(\mathcal{O})\times (L^{\b}(\mathcal{O}))^{*}} \qquad\mbox{as $\ve\to 0$},
\end{align*}
from which we infer
$$
\mathcal A [\rho_{\ve}]\to \mathcal A [\rho] \qquad\; \mbox{ in } C([0, T], W^{1, \b}_{\rm weak} (\mathcal{O})).
$$
This, combining with Proposition \ref{conv-h-1} ($\b>\frac{6}{5}$) and
the compact imbedding $W^{1, \b}(\mathcal{O})\Subset C(\mathcal{\bar{O}})$, gives
\begin{align}
&\na \mathcal A [\rho_{\ve}] \to \na \mathcal A [\rho]\qquad\,\,  \mbox{ in } C([0, T]; H^{-1}(\mathcal{O})),\label{na-a-rho-ve-conv-h-1}\\
&\mathcal A [\rho_{\ve}]\to \mathcal A [\rho] \qquad\qquad \mbox{ in } C(\overline{\mathcal{O}_{T}}). \label{a-rho-conv-Ctx}
\end{align}
From \eqref{u-ve-conv} and (\ref{a-rho-conv-Ctx}), we see that, as $\ve \to 0$,
\begin{align*}
\i_{2}\to \j_{2}, \qquad \i_{3}\to \j_{3}.
\end{align*}
By (\ref{rho-u-ve-conv}) and (\ref{a-rho-conv-Ctx}), we find that, as $\ve \to 0$,
$$
\i_{8}\to \j_{7}.
$$
From (\ref{u-ve-conv}) and (\ref{rho-u-ve-conv}), we have
$$
\|\rho_{\ve}u_{\ve}\otimes u_{\ve}\|_{L^{\frac{6\gamma}{3+4\gamma}}_{x}}\leq C\|\rho_{\ve}u_{\ve}\|_{L^{\frac{2\gamma}{\gamma +1}}_{x}}\|u_{\ve}\|_{L^{6}_{x}},
$$
which, combined with (\ref{rho-u-u-conv-D}), gives
$$
\rho_{\ve}u_{\ve}\otimes u_{\ve} \rightharpoonup \rho u\otimes u
\qquad \mbox{ in } L^{2}(0, T; L^{\frac{6\gamma}{3+4\gamma}}(\mathcal{O})).
$$
Thus, together with (\ref{a-rho-conv-Ctx}), we see that, if $\b > \frac{6\gamma}{2\gamma -3}$,
\begin{align*}
\i_{9}\to \j_{8} \qquad \mbox{as $\ve \to 0$}.
\end{align*}
Combining (\ref{u-ve-conv}) with (\ref{na-a-rho-ve-conv-h-1}), we obtain that, as $\ve \to 0$,
\begin{align*}
\i_{4}\to \j_{4}.
\end{align*}
In the same way as above, since $\b >4$, we know
\begin{align*}
\i_{5}\to \j_{5}.
\end{align*}
It follows from the boundedness of $\rho_{\ve}$ in $L^{\b}(\mathcal{O})$, $\rho_{\ve}u_{\ve}$ in $L^{\frac{2\gamma}{\gamma +1}}(\mathcal{O})$, and (\ref{a-w1p}) that
\begin{align*}
&\|\rho_{\ve}\del_{j}\mathcal A_{i} [\rho_{\ve}u_{\ve}^{j}]-\rho_{\ve}u_{\ve}^{j}\del_{j}\mathcal A_{i} [\rho_{\ve}]\|_{L^{\a}_{x}}\\
&\leq \|\rho_{\ve}\|_{L^{\b}_{x}} \|\del_{j}\mathcal A_{i} [\rho_{\ve}u_{\ve}^{j}]\|_{L^{\frac{2\gamma}{\gamma +1}}_{x}}
   +\|\rho_{\ve}u_{\ve}\|_{L^{\frac{2\gamma}{\gamma +1}}_{x}} \|\del_{j}\mathcal A_{i} [\rho_{\ve}]\|_{L^{\b}_{x}} \\
&\leq \|\rho_{\ve}\|_{L^{\b}_{x}} \|\rho_{\ve}u_{\ve}\|_{L^{\frac{2\gamma}{\gamma +1}}_{x}}\leq C,
\end{align*}
where $\frac{1}{\a}=\frac{\gamma +1}{2\gamma}+\frac{1}{\b}< \frac{5}{6}$, if $\b>\frac{6\gamma}{2\gamma -3}$.
Then we conclude
\begin{align*}
\rho_{\ve}\del_{j}\mathcal A_{i} [\rho_{\ve}u_{\ve}^{j}]-\rho_{\ve}u_{\ve}^{j}\del_{j}\mathcal A_{i} [\rho_{\ve}]\in L^{\infty}(0, T; L^{\a}(\mathcal{O})).
\end{align*}
From Lemma 3.4 in \cite{F-N-P-2001} and the compact embedding of $L^{\a}(\mathcal{O})$ in $H^{-1}(\mathcal{O})$, we infer that
\begin{align*}
\rho_{\ve}\del_{j}\mathcal A_{i} [\rho_{\ve}u_{\ve}^{j}]-\rho_{\ve}u_{\ve}^{j}\del_{j}\mathcal A_{i} [\rho_{\ve}]
\to \rho \del_{j}\mathcal A_{i} [\rho u^{j}]-\rho u^{j}\del_{j}\mathcal A_{i} [\rho]\qquad  \mbox{ in } H^{-1}(\mathcal{O}).
\end{align*}
Then, after applying Lebesgue convergence theorem, we obtain
\begin{align*}
\rho_{\ve}\del_{j}\mathcal A_{i} [\rho_{\ve}u_{\ve}^{j}]-\rho_{\ve}u_{\ve}^{j}\del_{j}\mathcal A_{i} [\rho_{\ve}]
\to \rho \del_{j}\mathcal A_{i} [\rho u^{j}]-\rho u^{j}\del_{j}\mathcal A_{i} [\rho]\qquad \mbox{ in } L^{2}(0, T; H^{-1}(\mathcal{O})),
\end{align*}
which, combined with (\ref{u-ve-conv}), yields
\begin{align*}
\i_{7}\to \j_{6} \qquad\, \mbox{as $\ve \to 0$}.
\end{align*}
Moreover, we have
\begin{align*}
|\i_{6}|&=\ve \Big|\int_{0}^{T}\psi \int_{\mathcal{O}}\phi \rho_{\ve}u_{\ve}^{i} \mathcal{A}_{i}[\mathrm{div}(\mathrm{1}_{\O}\na \rho_{\ve})]\,\md x \md t\Big|\\
&\leq C\ve \int_{0}^{T}\|\rho_{\ve}\|_{L^{3}_{x}} \|u_{\ve}\|_{L^{6}_{x}}\|\mathcal A_{i} [\mathrm{div} (\mathrm{1}_{\O}\na \rho_{\ve})] \|_{L^{2}_{x}}\md t\\
&\leq C\ve \int_{0}^{T}\|\rho_{\ve}\|_{L^{3}_{x}} \|u_{\ve}\|_{L^{6}_{x}}\|\na \rho_{\ve} \|_{L^{2}_{x}}\md t\\
&\leq C\ve \|\rho_{\ve}\|_{L^{\infty}_{t} L^{3}_{x}} \|u_{\ve}\|_{L^{2}_{t}L^{6}_{x}}\|\na \rho_{\ve} \|_{L^{2}_{t} L^{2}_{x}}
\leq C\sqrt{\ve} \to 0\;\qquad  \mbox{as}\; \ve \to 0.
\end{align*}
Since $\b >3$, we have
\begin{align*}
|\i_{10}|&=\ve\Big|\int_{0}^{T}\psi \int_{\mathcal{O}}\del_{j} u_{\ve}^{i}\del_{j}\rho_{\ve}\phi \mathcal A_{i} [\rho_{\ve}] \,\md x \md t\Big|\\
&\leq C\ve \int_{0}^{T}\|\na u_{\ve}\|_{L^{2}_{x}} \|\na \rho_{\ve}\|_{L^{2}_{x}}\|\mathcal A_{i} [\rho_{\ve}] \|_{L^{\infty}_{x}}\,\md t\\
&\leq C\ve \int_{0}^{T}\|\na u_{\ve}\|_{L^{2}_{x}} \|\na \rho_{\ve}\|_{L^{2}_{x}}\|\rho_{\ve} \|_{L^{\b}_{x}}\,\md t\\
&\leq C\ve \|\na u_{\ve}\|_{L^{2}_{t} L^{2}_{x}} \|\na \rho_{\ve}\|_{L^{2}_{t} L^{2}_{x}}\|\rho_{\ve} \|_{L^{\infty}_{t} L^{\b}_{x}}
\leq C\sqrt{\ve}\to 0\qquad \mbox{as}\; \ve \to 0.
\end{align*}
From \eqref{c-ve-conv}--(\ref{Q-ve-conv}) and  (\ref{na-a-rho-ve-conv-h-1})--(\ref{a-rho-conv-Ctx}),
we have
\begin{align*}
\i_{11}=\sigma_{*} \int_{0}^{T}\psi \int_{\mathcal{O}}c_{\ve}^{2}Q^{ij}_{\ve}
  \big(\del_{j}\phi \mathcal A_{i} [\rho_{\ve}]+\phi \del_{j}\mathcal A_{i}[\rho_{\ve}]\big)\,\md x \md t \to \j_{9}.
\end{align*}
From (\ref{Q-ve-conv}), (\ref{a-rho-conv-Ctx}), and the boundedness of $Q_{\ve}$ in $L^{2}(0, T; H^{2}(\mathcal{O}))$,
we have
\begin{align*}
\i_{12}&=\int_{0}^{T}\psi \int_{\mathcal{O}}\big(\del_{i} Q^{kl}_{\ve} \del_{j} Q^{kl}_{\ve}-\mathrm{F}(Q_{\ve})\dl_{ij}\big)\del_{j}\phi \mathcal A_{i} [\rho_{\ve}]\,\md x \md t
\to \j_{10},
\end{align*}
\begin{align*}
\i_{13}&=\int_{0}^{T}\psi \int_{\mathcal{O}}\big(\del_{i} Q^{kl}_{\ve} \del_{j} Q^{kl}_{\ve}-\mathrm{F}(Q_{\ve})\dl_{ij}\big)\phi \del_{j}\mathcal A_{i}[\rho_{\ve}]\,\md x \md t
\to \j_{11},
\end{align*}
and
\begin{align*}
\i_{14}&=-\int_{0}^{T}\psi \int_{\mathcal{O}}\big(Q^{ik}_{\ve}\D Q^{kj}_{\ve}-\D Q^{ik}_{\ve}Q^{kj}_{\ve}\big)
  \big( \del_{j}\phi \mathcal A_{i} [\rho_{\ve}]+\phi \del_{j}\mathcal A_{i}[\rho_{\ve}]\big)\,\md x \md t
\to \j_{12}.
\end{align*}

Next, following the same argument as in Subsection 3.5 in \cite{F-2001}, we obtain the strong convergence of the density sequence $\rho_{\ve}$
in $L^{1}(\mathcal{O}_{T})$, {\it i.e.}
\be\label{p}
p=\rho^{\gamma}+\delta \rho^{\b}.
\en
\end{proof}

Then we have the following proposition:
\begin{Proposition}\label{existence-dl}
Let $\b >\max \{\frac{6\gamma}{2\gamma -3}, \gamma, 4\}$.
Then, for any given $T>0$ and $\delta >0$, there exists a finite-energy weak solution $(c, \rho, u, Q)$ of the problem{\rm :}
\begin{align}
&\del_{t}c+u\cdot \na c=D_{0}\D c, \label{c-dl-1}\\
&\del_{t}\rho+\mathrm{div}(\rho u)=0, \label{rho-dl-1}\\
&\del_{t}(\rho u)+\na \cdot (\rho u\otimes u)+\na (\rho^{\gamma}+\delta \rho^{\b})
  =\mu \D u+(\nu +\mu)\na \mathrm{div}\, u +\nabla \cdot (Q\D Q-\D Q Q)\nonumber\\
& \qquad\qquad\qquad\qquad\qquad\qquad\qquad\quad\quad\,\,\, +\na \cdot (\mathrm{F}(Q)\mathrm{I}_{3}-\na Q\odot \na Q)+\sigma_{*} \na \cdot (c^{2}Q), \label{u-dl-1}\\
&\partial_{t}Q+(u\cdot \nabla )Q+Q\Omega-\Omega Q
=\Gamma H[Q, c], \label{Q-dl-1}
\end{align}
with initial-boundary conditions \eqref{I-C-c}$-$\eqref{B-C-uQ}.
Moreover, equation \eqref{rho-dl-1} holds in the sense of renormalized solutions
on $\mathcal D^{'} ((0, T)\times \R^{3})$, provided that  $(\rho, u)$ are extended
to be zero on $\R^{3}\setminus \mathcal{O}$.
In addition, the following estimates are valid{\rm :}
\begin{align}
&0<\underline{c}\leq c(x, t)\leq \bar{c}, \quad \|c\|_{L^{\infty}(0, T; L^{2}(\mathcal{O}))\cap L^{2}(0, T; H^{1}(\mathcal{O}))}\leq C, \label{c-dl-gamma}\\
&\sup_{t\in [0, T]}\|\rho(t, \cdot)\|_{L^{\gamma}(\mathcal{O})}^{\gamma}\leq C,\label{rho-dl-gamma}\\
&\delta \sup_{t\in [0, T]}\|\rho(t, \cdot)\|_{L^{\b}(\mathcal{O})}^{\b}\leq C,\label{rho-dl-b}\\
&\sup_{t\in [0, T]}\|\sqrt{\rho}(t, \cdot)u(t, \cdot)\|_{L^{2}(\mathcal{O})}^{2}\leq C,\label{rho-u-2}\\[1mm]
&\|u\|_{L^{2}(0, T; H^{1}_{0}(\mathcal{O}))}\leq C,\label{u-dl-h1}\\[1mm]
&\|Q\|_{L^{\infty}(0, T; H^{1}(\mathcal{O})\cap L^{4}(\mathcal{O}))\cap L^{2}(0, T; H^{2}(\mathcal{O})\cap L^{6}(\mathcal{O}))}\leq C,\\[1mm]
&\|Q\|_{L^{10}(\mathcal{O}_{T})}\leq C,\label{Q-dl-10}\\[1mm]
&\|\na Q\|_{L^{\frac{10}{3}}(\mathcal{O}_{T})}\leq C,\label{na-Q-dl}
\end{align}
where $C$ is a constant independent of $\ve$.
\end{Proposition}

\begin{Remark}
The initial conditions (\ref{I-C-rho-a})$-$(\ref{I-C-Q}) are satisfied in the weak sense,
since
$$
\rho_{\ve}\to \rho \,\,\,\, \mbox{in}\,\, C([0, T]; L^{\b}_{\rm weak}(\mathcal{O})), \qquad
\rho_{\ve}u_{\ve}\to \rho u \,\,\,\, \mbox{in}\,\, C([0, T]; L^{\frac{2\gamma}{\gamma +1}}_{\rm weak}(\mathcal{O}))
$$
from (\ref{rho-ve-conv})--(\ref{u-ve-conv}).
\end{Remark}

\section{Passing to the Limit in the Artificial Pressure}

We denote by $(c_{\dl}, \rho_{\dl}, u_{\dl}, Q_{\dl})$ the approximate solutions constructed in Proposition \ref{existence-dl}. In this section, we
let $\delta \to 0$ in (\ref{c-dl-1})$-$\eqref{Q-dl-1} to obtain the solution of the original problem \eqref{c}$-$\eqref{Q}.

In order for solution $(\rho_{\dl}, u_{\dl})$ to satisfy the initial condition \eqref{I-C-rho-a}--\eqref{I-C-rho-b} in Proposition \ref{existence-dl}, we
first modify the general initial data
$(\rho_{0}, m_{0})$ to satisfy the compatibility condition (\ref{compat-condition}).
As in \cite{F-N-P-2001}, it is easy to find a sequence $\tilde{\rho}_{\dl} \in C_{0}^{3}(\mathcal{O})$ with the property:
\begin{align*}
0\leq \tilde{\rho}_{\dl}(x)\leq \frac{1}{2}\dl^{-\frac{1}{\b}},\qquad \|\tilde{\rho}_{\dl}-\rho_{0}\|_{L^{2}(\mathcal{O})}<\dl.
\end{align*}
Take $\rho_{0, \dl}=\tilde{\rho}_{\dl}+\dl$. From \eqref{I-C-rho-a}--\eqref{I-C-rho-b}, we have
\be\label{rho-0-dl}
0<\dl \leq \rho_{0, \dl}(x)\leq \dl^{-\frac{1}{\b}},\qquad \na \rho_{0, \dl} \cdot \vec{n}|_{\partial \mathcal{O}}=0,
\en
with
\be\label{rho-0-dl-rho-0}
\rho_{0, \dl}\to \rho_{0} \qquad \mbox{ in } L^{\gamma}(\mathcal{O}) \mbox{ as } \dl \to 0.
\en
Set
\begin{align*}
\tilde{q}_{\dl}(x)=
\begin{cases}
m_{0}(x)\sqrt{\frac{\rho_{0, \dl}}{\rho_{0}}}  &\mbox{ if } \rho_{0}(x)>0,\\
0 &\mbox{ if } \rho_{0}(x)=0.
\end{cases}
\end{align*}
Then it follows from (\ref{compat-condition}) that $\frac{|\tilde{q}_{\dl}|^{2}}{\rho_{0, \dl}}$
is uniformly bounded in $L^{1}(\mathcal{O})$.
It is also direct to find $h_{\dl}\in C^{2}(\bar{\mathcal{O}})$ such that
\begin{align*}
\big\|\frac{\tilde{q}_{\dl}}{\sqrt{\rho_{0, \dl}}}-h_{\dl}\big\|_{L^{2}(\mathcal{O})}<\dl.
\end{align*}
Taking $m_{0, \dl}=h_{\dl}\sqrt{\rho_{0, \dl}}$,
we can check that
\be\label{uniform-bd}
\frac{|m_{0, \dl}|^{2}}{\rho_{0, \dl}} \qquad \mbox{ is uniformly bounded in } L^{1}(\mathcal{O}) \mbox{
 with respect to } \dl>0,
\en
and
\be\label{q-dl-q}
m_{0, \dl}\to m_{0}\qquad \mbox{ in } L^{1}(\mathcal{O}) \mbox{ as } \dl \to 0.
\en

From now on, we deal with the sequence of approximate solutions
$(c_{\dl}, \rho_{\dl}, u_{\dl}, Q_{\dl})$ of  problem (\ref{c-dl-1})$-$(\ref{Q-dl-1})
with the initial data $(c_{0}, \rho_{0, \dl}, m_{0, \dl}, Q_{0})$.
The existence of such a solution is provided by Proposition \ref{existence-dl}.
We notice that all the corresponding estimates in Proposition \ref{existence-dl}
are independent of $\dl$, by virtue of (\ref{rho-0-dl}) and (\ref{uniform-bd}).

\subsection{On the integrability of the density}

Now we provide some pressure estimates independent of $\dl>0$.
From \eqref{rho-dl-gamma}, $\rho_{\dl}\in L^{\infty}(0, T; L^{\gamma}(\O))$ uniformly in $\dl$  yields that
$\rho^{\gamma}_{\dl}\rightharpoonup \mu$ by measure, as $\dl \to 0$.
Hence, we need the higher integrability of $\rho_{\dl}$.
The technique is similar to that in \S \ref{density-ind-ve}.

Since the continuity equation (\ref{rho-dl-1}) is satisfied in the sense of renormalized solutions
in $\mathcal D^{'} ((0, T)\times \R^{3})$,
we can apply the standard mollifying operator on both sides of (\ref{rho-dl-1}) to obtain
\begin{align*}
\del_{t}S_{m}[g(\rho_{\dl})]+\mathrm{div}(S_{m}[g(\rho_{\dl})u_{\dl}])+S_{m}[(g'(\rho_{\dl})\rho_{\dl}
-g(\rho_{\dl}))\mathrm{div}\, u_{\dl}]=r_{m},
\end{align*}
with
\begin{align*}
r_{m}\to 0 \qquad \mbox{ in } L^{2}(0, T; L^{2}(\R^{3})) \mbox{ as } m\to \infty.
\end{align*}
As in \S \ref{density-ind-ve}, we use operator $\mathcal B$ to construct the test function  as
\be\label{test-function-dl}
\phi(t, x)=\psi(t)\mathcal B \big[S_{m}[g(\rho_{\dl})]-\avint_{\mathcal{O}}S_{m}[g(\rho_{\dl})]\,\md y\big],
\en
with
$$
\psi \in \mathcal D (0, T),\qquad
\avint_{\mathcal{O}}S_{m}[g(\rho_{\dl})]\,\md y:=\frac{1}{|\mathcal{O}|}\int_{\mathcal{O}}S_{m}[g(\rho_{\dl})]\,\md y.
$$
Since $\phi|_{\del \mathcal{O}}=0$, $\phi\in L^{\infty}(0, T; H^{1}_{0}(\mathcal{O}))$,
and $\del_{t}\phi \in L^{2}(0, T; H^{1}_{0}(\mathcal{O}))$, we know
that $\phi$ can be used as a test function for (\ref{u-dl-1}).
We approximate the function: $g(z)=z^{\theta}, 0<\theta \leq 1$,
by a sequence of functions $\{z^{\theta}\chi_{n}(z)\}_{n=1}^{\infty}$, where $\{\chi_{n}(z)\}_{n=1}^{\infty}$ are cutoff functions with
\begin{align*}
\chi_{n}(z)=
\begin{cases}
1 \qquad \mbox{ if } z\in [0, n],\\
0 \qquad \mbox{ if } z>2n.
\end{cases}
\end{align*}
Then, from (\ref{c-dl-gamma})$-$(\ref{na-Q-dl}),
we employ the same techniques as in Lemma \ref{lemma-density} to obtain

\begin{Lemma}\label{uniform-rho-dl}
If $\gamma >\frac{3}{2}$, there exists a constant $\theta$, depending only on $\gamma$, such that
\be\label{5.6}
\int_{0}^{T}\int_{\mathcal{O}}\big(\rho_{\dl}^{\gamma +\theta}+\dl \rho_{\dl}^{\b +\theta}\big)\,\md x \md t \leq C
\en
where $C$ is independent of $\dl$, and $0<\theta <\min \{\frac{1}{4}, \frac{2\gamma}{3}-1\}$.
\end{Lemma}

\begin{proof}
Applying the test function (\ref{test-function-dl}) to (\ref{u-dl-1}), we have
\begin{align}
&\int_{0}^{T}\psi(t) \int_{\mathcal{O}} \big(\rho_{\dl}^{\gamma }+\delta \rho_{\dl}^{\b}\big)S_{m}[g(\rho_{\dl})]\,\md x \md t\nonumber\\
=&\int_{0}^{T}\psi(t) \int_{\mathcal{O}} \big(\rho_{\dl}^{\gamma }+\delta \rho_{\dl}^{\b}\big)\md x \avint_{\mathcal{O}}S_{m}[g(\rho_{\dl})]\,\md y \md t
+(\mu +\nu)\int_{0}^{T}\psi \int_{\mathcal{O}}\mathrm{div}\, u_{\dl}S_{m}[g(\rho_{\dl})]\,\md x \md t \nonumber\\
&-\int_{0}^{T}\psi_{t} \int_{\mathcal{O}}\rho_{\dl}u_{\dl}\cdot \mathcal B\big[S_{m}[g(\rho_{\dl})]-\avint_{\mathcal{O}}S_{m}[g(\rho_{\dl})]\md y\big]\,\md x \md t\nonumber \\
&+\mu \int_{0}^{T}\psi \int_{\mathcal{O}}\del_{j}u_{\dl}^{i}\del_{j} B_{i}\big[S_{m}[g(\rho_{\dl})]-\avint_{\mathcal{O}}S_{m}[g(\rho_{\dl})]\md y\big]\,\md x \md t\nonumber \\
& -\int_{0}^{T}\psi \int_{\mathcal{O}} \rho_{\dl}u_{\dl}^{i}u_{\dl}^{j}\del_{j} B_{i}\big[S_{m}[g(\rho_{\dl})]-\avint_{\mathcal{O}}S_{m}[g(\rho_{\dl})]\md y\big]\,\md x \md t\nonumber\\
&-\int_{0}^{T}\psi \int_{\mathcal{O}}\rho_{\dl}u_{\dl}\cdot \mathcal B \big[r_{m}-\avint_{\mathcal{O}}r_{m}\md y\big]\,\md x \md t \nonumber\\
&-\int_{0}^{T}\psi \int_{\mathcal{O}}\rho_{\dl}u_{\dl}\cdot \mathcal B\big[S_{m}[(g'(\rho_{\dl})\rho_{\dl} -g(\rho_{\dl}))\mathrm{div}\, u_{\dl}]\nonumber\\
&-\avint_{\mathcal{O}}S_{m}[(g'(\rho_{\dl})\rho_{\dl} -g(\rho_{\dl}))\mathrm{div}\, u_{\dl}]\md y\big]\,\md x \md t \nonumber\\
&+\int_{0}^{T}\psi \int_{\mathcal{O}}\rho_{\dl}u_{\dl}\cdot \mathcal B [\mathrm{div}\, (S_{m}[g(\rho_{\dl})]u_{\dl})]\,\md x \md t \nonumber\\
&-\int_{0}^{T}\psi \int_{\mathcal{O}} \big(\na Q_{\dl}\otimes \na Q_{\dl}
  -\mathrm{F}(Q_{\dl})\mathrm{I}_{3}\big): \na \mathcal B\big[S_{m}[g(\rho_{\dl})]-\avint_{\mathcal{O}}S_{m}[g(\rho_{\dl})]\md y\big]\,\md x \md t\nonumber\\
&
+\sigma_{*}\int_{0}^{T}\psi \int_{\mathcal{O}}c_{\dl}^{2}Q_{\dl}: \na \mathcal B \big[S_{m}[g(\rho_{\dl})]-\avint_{\mathcal{O}}S_{m}[g(\rho_{\dl})]\md y\big]\,\md x \md t\nonumber\\
&-\int_{0}^{T}\psi \int_{\mathcal{O}}\na \cdot \big(Q_{\dl}\D Q_{\dl}-\D Q_{\dl}Q_{\dl}\big)\cdot \mathcal B\big[S_{m}[g(\rho_{\dl})]-\avint_{\mathcal{O}}S_{m}[g(\rho_{\dl})]\md y\big]\,\md x \md t\nonumber\\
&=\sum_{i=1}^{11}\i_{i}.
\label{rho-dl-higher-integrability}
\end{align}
By the estimates
in Proposition \ref{existence-dl},
we can pass to the limit in \eqref{rho-dl-higher-integrability} as $m\to \infty$ to obtain
\begin{align}
&\int_{0}^{T}\psi(t) \int_{\mathcal{O}} \big(\rho_{\dl}^{\gamma }+\delta \rho_{\dl}^{\b}\big)g(\rho_{\dl})\,\md x \md t\nonumber\\
=&\int_{0}^{T}\psi(t) \int_{\mathcal{O}} \big(\rho_{\dl}^{\gamma }+\delta \rho_{\dl}^{\b}\big)\,\md x \avint_{\mathcal{O}}g(\rho_{\dl})\,\md y \md t
  +(\mu +\nu)\int_{0}^{T}\psi \int_{\mathcal{O}}\mathrm{div}\, u_{\dl}\,g(\rho_{\dl})\,\md x \md t \nonumber\\
&-\int_{0}^{T}\psi_{t} \int_{\mathcal{O}}\rho_{\dl}u_{\dl}\cdot \mathcal B\big[g(\rho_{\dl})-\avint_{\mathcal{O}}g(\rho_{\dl})\md y\big]\,\md x \md t \nonumber\\
&+\mu \int_{0}^{T}\psi \int_{\mathcal{O}}\del_{j}u_{\dl}^{i}\del_{j} B_{i}\big[g(\rho_{\dl})-\avint_{\mathcal{O}}g(\rho_{\dl})\md y\big]\,\md x \md t \nonumber\\
& -\int_{0}^{T}\psi \int_{\mathcal{O}} \rho_{\dl}u_{\dl}^{i}u_{\dl}^{j}\del_{j} B_{i}\big[g(\rho_{\dl})-\avint_{\mathcal{O}}g(\rho_{\dl})\md y\big]\,\md x \md t \nonumber\\
&-\int_{0}^{T}\psi \int_{\mathcal{O}}\rho_{\dl}u_{\dl}\cdot \mathcal B\big[(g'(\rho_{\dl})\rho_{\dl}
  -g(\rho_{\dl}))\mathrm{div}\, u_{\dl}-\avint_{\mathcal{O}}(g'(\rho_{\dl})\rho_{\dl}
  -g(\rho_{\dl}))\mathrm{div}\, u_{\dl}\md y\big]\,\md x \md t \nonumber\\
&+\int_{0}^{T}\psi \int_{\mathcal{O}}\rho_{\dl}u_{\dl}\cdot \mathcal B\big[\mathrm{div}\, (g(\rho_{\dl})u_{\dl})\big]\,\md x \md t \nonumber\\
&-\int_{0}^{T}\psi \int_{\mathcal{O}} \big(\na Q_{\dl}\otimes \na Q_{\dl}
  -\mathrm{F}(Q_{\dl})\mathrm{I}_{3}\big):  \na \mathcal B\big[g(\rho_{\dl})-\avint_{\mathcal{O}}g(\rho_{\dl})\md y\big]\,\md x \md t\nonumber\\
&
+\sigma_{*}\int_{0}^{T}\psi \int_{\mathcal{O}}c_{\dl}^{2}Q_{\dl}: \na \mathcal B\big[g(\rho_{\dl})-\avint_{\mathcal{O}}g(\rho_{\dl})\md y\big]\,\md x \md t\nonumber\\
&-\int_{0}^{T}\psi \int_{\mathcal{O}}\na \cdot \big(Q_{\dl}\D Q_{\dl}-\D Q_{\dl}Q_{\dl}\big)\cdot \mathcal B\big[g(\rho_{\dl})-\avint_{\mathcal{O}}g(\rho_{\dl})\md y\big]\,\md x \md t\nonumber\\
=&\sum_{j=1}^{10}\j_{j}. \label{rho-dl-higher-integrability-1}
\end{align}
For example, we have the following estimates:
Since $\rho_{\dl}\in L^{\b +1}(\mathcal{O}_{T})$ and $g'(z)=0$ for sufficiently large $z$, we have
\begin{align*}
&\Big|\int_{0}^{T}\psi \int_{\mathcal{O}}\rho_{\dl}^{\b}\big(S_{m}[g(\rho_{\dl})]-g(\rho_{\dl})\big)\,\md x\md t\Big|\\
&\leq C\|\rho_{\dl}\|_{L^{\b+1}_{t, x}}^{\b}\|S_{m}[g(\rho_{\dl})]-g(\rho_{\dl})\|_{L^{\b +1}_{t, x}}\\
&\leq C(\dl)\|S_{m}[g(\rho_{\dl})]-g(\rho_{\dl})\|_{L^{\b +1}_{t, x}} \to 0 \qquad  \mbox{ as } m\to \infty.
\end{align*}
By the property of $\mathcal{B}$ in \eqref{B-W-1-p}, we have
\begin{align*}
\big\|\del_{j}B_{i}\big[S_{m}[g(\rho_{\dl})]-\avint_{\mathcal{O}}S_{m}[g(\rho_{\dl})]\md y\big]\big\|_{L^{\infty}_{t}L^{3}_{x}}
&\leq C\|S_{m}[g(\rho_{\dl})]-\avint_{\mathcal{O}}S_{m}[g(\rho_{\dl})]\md x\|_{L^{\infty}_{t}L^{3}_{x}}\\[1mm]
&\leq C\|g(\rho_{\dl})\|_{L^{\infty}_{t}L^{3}_{x}}\leq C(\dl),
\end{align*}
where $C(\dl)$ is independent of $m$.
This shows that
\begin{align*}
\del_{j}B_{i}\big[S_{m}[g(\rho_{\dl})]-\avint_{\mathcal{O}}S_{m}[g(\rho_{\dl})]\md y\big]
\overset{*}\rightharpoonup G \qquad \mbox{weak-star in $L^{\infty}_{t}L^{3}_{x}$}.
\end{align*}
On the other hand, since
\begin{align*}
&\big\|\del_{j}B_{i}\big[S_{m}[g(\rho_{\dl})]-g(\rho_{\dl})-\avint_{\mathcal{O}}(S_{m}[g(\rho_{\dl})]-g(\rho_{\dl}))\,\md x\big]\big\|_{L^{2}_{t}L^{2}_{x}}\\
&\leq C\|S_{m}[g(\rho_{\dl})]-g(\rho_{\dl})-\avint_{\mathcal{O}}\big(S_{m}[g(\rho_{\dl})]-g(\rho_{\dl})\big)\,\md x\|_{L^{2}_{t}L^{2}_{x}}\\
&\leq C\|S_{m}[g(\rho_{\dl})]-g(\rho_{\dl})\|_{L^{2}_{t}L^{2}_{x}}\to 0,
\end{align*}
we have
\begin{align*}
\del_{j}B_{i}\big[S_{m}[g(\rho_{\dl})]-\avint_{\mathcal{O}}S_{m}[g(\rho_{\dl})]\md y\big]
\overset{*}\rightharpoonup \del_{j}B_{i}\big[g(\rho_{\dl})-\avint_{\mathcal{O}}g(\rho_{\dl})\md y\big] \,\,\,\, \mbox{weak-star in $L^{\infty}_{t}L^{3}_{x}$}.
\end{align*}
Then
\begin{align*}
\i_{5}&=\int_{0}^{T}\psi \int_{\mathcal{O}} \rho_{\dl}u_{\dl}^{i}u_{\dl}^{j}\del_{j}
B_{i}\big[S_{m}[g(\rho_{\dl})]-\avint_{\mathcal{O}}S_{m}[g(\rho_{\dl})]\md y\big]\,\md x \md t\\
&\to \int_{0}^{T}\psi \int_{\mathcal{O}} \rho_{\dl}u_{\dl}^{i}u_{\dl}^{j}\del_{j} B_{i}\big[g(\rho_{\dl})-\avint_{\mathcal{O}}g(\rho_{\dl})\md y\big]\,\md x \md t
=\j_{5}.
\end{align*}
By \eqref{B-div-r}, we have
\begin{align*}
&\mathcal B\big[\mathrm{div}\, (S_{m}[g(\rho_{\dl})]u_{\dl})\big]
  \to \mathcal B\big[\mathrm{div}\, (g(\rho_{\dl})u_{\dl})\big] \qquad \mbox{ strongly in }L^{\frac{3}{2}}_{t}L^{2}_{x},\\
&\mathcal B\big[\mathrm{div}\, (S_{m}[g(\rho_{\dl})]u_{\dl})\big]
  \rightharpoonup G \qquad\qquad\qquad\qquad\, \mbox{ weakly in }L^{2}(Q_{T}).
\end{align*}
Thus, we obtain
\begin{align*}
&\mathcal B\big[\mathrm{div}\, (S_{m}[g(\rho_{\dl})]u_{\dl})\big]
  \rightharpoonup \mathcal B\big[\mathrm{div}\, (g(\rho_{\dl})u_{\dl})\big] \qquad \mbox{ weakly in }L^{2}(Q_{T}).
\end{align*}
Moreover, since $\psi \rho_{\dl}u_{\dl}\in L^{2}(Q_{T})$, we have
\begin{align*}
&\i_{8}=\int_{0}^{T}\int_{\mathcal{O}}\psi \rho_{\dl}u_{\dl}\cdot \mathcal B\big[\mathrm{div}\, (S_{m}[g(\rho_{\dl})]u_{\dl})\big]\md x \md t
 \to \int_{0}^{T}\int_{\mathcal{O}}\psi \rho_{\dl}u_{\dl}\cdot \mathcal B \big[\mathrm{div}\, (g(\rho_{\dl})u_{\dl})\big] \md x \md t =\j_{7}
\end{align*}
as $m \to \infty$. Then, as we mentioned before, we can use a sequence of functions $\{z^{\theta}\chi_{n}(z)\}$
to approximate $g(z)=z^{\theta}$ to obtain
\begin{equation}\label{5.9a}
\int_{0}^{T}\int_{\mathcal{O}}\psi \big(\rho_{\dl}^{\gamma +\theta}+\dl \rho_{\dl}^{\b +\theta}\big)\,\md x \md t
=\sum_{i=1}^{10}\j_{i},
\end{equation}
where $g(z)$ is substituted by $z^{\theta}$ in every $\j_{i}$, $i=1, 2, \cdots, 10$.

Next, we estimate the terms on the right-hand side of \eqref{5.9a}
to obtain the condition for $\theta$,
for which the universal constant $C$ is independent of $\dl$:
\begin{align*}
&|\j_{1}|=\Big|\int_{0}^{T}\psi \int_{\mathcal{O}}(\rho_{\dl}^{\gamma}+\delta \rho_{\dl}^{\b})\,\md x \avint_{\mathcal{O}}\rho_{\dl}^{\theta}\,\md x \md t\Big| \leq C\big(\|\rho_{\dl}\|_{L^{\infty}_{t}L^{\gamma}_{x}}^{\gamma}+\delta \|\rho_{\dl}\|_{L^{\infty}_{t}L^{\b}_{x}}^{\b}\big)\|\rho_{\dl}\|_{L^{\infty}_{t}L^{\theta}_{x}}^{\theta}\leq C.\\
&|\j_{2}|=(\mu +\nu)\Big|\int_{0}^{T}\psi \int_{\mathcal{O}}\rho_{\dl}^{\theta}\mathrm{div}\, u_{\dl}\,\md x \md t\Big|
 \leq C\|\na u_{\dl}\|_{L^{2}(Q_{T})}\|\rho_{\dl}^{\theta}\|_{L^{\infty}_{t}L^{2}_{x}}
\leq C \qquad\mbox{if $\theta \leq \frac{\gamma}{2}$}.\\
&|\j_{3}|=\Big|\int_{0}^{T}\psi_{t} \int_{\mathcal{O}}\rho_{\dl}u_{\dl}\cdot \mathcal B\big[\rho_{\dl}^{\theta}-\avint_{\mathcal{O}}\rho_{\dl}^{\theta}\md y\big]\,\md x \md t\Big|\\
&\quad \leq C\int_{0}^{T} \|\sqrt{\rho_{\dl}}\|_{L^{2\gamma}_{x}} \|\sqrt{\rho_{\dl}}u_{\dl}\|_{L^{2}_{x}}
   \big\|\mathcal B\big[\rho_{\dl}^{\theta}-\avint_{\mathcal{O}}\rho_{\dl}^{\theta}\md y\big]\big\|_{L^{\frac{2\gamma}{\gamma -1}}_{x}}\md t\\
&\quad\leq  C\int_{0}^{T} \|\rho_{\dl}\|_{L^{\gamma}_{x}}^{\frac{1}{2}} \|\sqrt{\rho_{\dl}}u_{\dl}\|_{L^{2}_{x}}
   \big\|\rho_{\dl}^{\theta}-\avint_{\mathcal{O}}\rho_{\dl}^{\theta}\md y\big\|_{L^{\frac{2\gamma}{\gamma -1}}_{x}}\md t\\
&\quad\leq CT\|\rho_{\dl}\|_{L^{\infty}_{t}L^{\gamma}_{x}}^{\frac{1}{2}} \|\sqrt{\rho_{\dl}}u_{\dl}\|_{L^{\infty}_{t}L^{2}_{x}}
  \|\rho_{\dl}\|_{L^{\infty}_{t}L^{\frac{2\gamma}{\gamma -1}\theta}_{x}}^{\theta}\leq C(T)\;\quad  \mbox{ if $\theta \leq \frac{\gamma-1}{2}$}. \\
|\j_{4}|&=\mu \Big|\int_{0}^{T}\psi \int_{\mathcal{O}}\del_{j}u_{\dl}^{i}\del_{j}B_{i}\big[\rho_{\dl}^{\theta}-\avint_{\mathcal{O}}\rho_{\dl}^{\theta}\md y\big]\md x \md t\Big| \\
&\leq  \mu \int_{0}^{T}\|\na u_{\dl}\|_{L^{2}_{x}}\big\|\na \mathcal B\big[\rho_{\dl}^{\theta}-\avint_{\mathcal{O}}\rho_{\dl}^{\theta}\md y\big]\big\|_{L^{2}_{x}}\md t\\
&\leq \mu \int_{0}^{T}\|\na u_{\dl}\|_{L^{2}_{x}}\| \rho_{\dl}^{\theta}-\avint_{\mathcal{O}}\rho_{\dl}^{\theta}\md y\|_{L^{2}_{x}}\md t\leq C \qquad \mbox{if $\theta \leq \frac{\gamma}{2}$}.
\end{align*}
\begin{align*}
&|\j_{5}|=\Big|\int_{0}^{T}\psi \int_{\mathcal{O}} \rho_{\dl}u_{\dl}^{i}u_{\dl}^{j}\del_{j} B_{i}\big[\rho_{\dl}^{\theta}-\avint_{\mathcal{O}}\rho_{\dl}^{\theta}\md y\big]\md x \md t\Big|\\
&\quad\leq C\int_{0}^{T}\|\rho_{\dl}\|_{L^{\gamma}_{x}}\|u_{\dl}\|^{2}_{L^{6}_{x}}\big\| \na \mathcal B\big[\rho_{\dl}^{\theta}-\avint_{\mathcal{O}}\rho_{\dl}^{\theta}\md y\big]\big\|_{L^{\frac{3\gamma}{2\gamma -3}}_{x}}\md t\\
&\quad\leq C\int_{0}^{T}\|\rho_{\dl}\|_{L^{\gamma}_{x}}\|u_{\dl}\|^{2}_{L^{6}_{x}}\|\rho_{\dl}^{\theta}\|_{L^{\frac{3\gamma}{2\gamma -3}}_{x}}\md t\leq C \qquad\mbox{if $\theta \leq \frac{2\gamma}{3}-1$}.
\end{align*}

\begin{align*}
|\j_{6}|&= (1-\theta)\Big|\int_{0}^{T}\psi \int_{\mathcal{O}}\rho_{\dl}u_{\dl}\cdot \mathcal{B}\big[\rho_{\dl}^{\theta}\mathrm{div}\, u_{\dl}
  -\avint_{\mathcal{O}}\rho_{\dl}^{\theta}\mathrm{div}\, u_{\dl}\md y\big]\,\md x \md t\Big|\\
&\leq C \int_{0}^{T}\|\rho_{\dl}\|_{L^{\gamma}_{x}}\|u_{\dl}\|_{L^{6}_{x}}
      \big\| \mathcal B\big[\rho_{\dl}^{\theta}\mathrm{div}\, u_{\dl}-\avint_{\mathcal{O}}\rho_{\dl}^{\theta}\mathrm{div}\, u_{\dl}\md y\big]\big\|_{L^{\frac{6\gamma}{5\gamma-6}}_{x}}\md t\\
&\leq C \int_{0}^{T}\|u_{\dl}\|_{L^{6}_{x}}\big\| \mathcal B\big[\rho_{\dl}^{\theta}\mathrm{div}\, u_{\dl}-\avint_{\mathcal{O}}\rho_{\dl}^{\theta}\mathrm{div}\, u_{\dl}\md y\big]\big\|_{W^{1,p}_{x}}\md t\\
& \leq C \int_{0}^{T}\|u_{\dl}\|_{L^{6}_{x}}\| \rho_{\dl}^{\theta}\mathrm{div}\, u_{\dl}\|_{L^{p}_{x}}\md t\\
&\leq C \int_{0}^{T}\|u_{\dl}\|_{L^{6}_{x}}\|\na u_{\dl}\|_{L^{2}_{x}}\| \rho_{\dl}^{\theta}\|_{L^{q}_{x}}\md t\\
&\leq C \int_{0}^{T}\|u_{\dl}\|_{H^{1}_{x}}^{2}\| \rho_{\dl}\|_{L^{\gamma}_{x}}\md t\leq C,
\end{align*}
where
$p=\frac{6\gamma}{7\gamma -6}$  if $\gamma<6$, $\,\, p=\frac{3}{2}$ if $\gamma \geq 6$;
$\,q=\frac{3\gamma}{2\gamma -3}$ if $\gamma <6$, $p=6$  if $\gamma \geq 6$;
$\theta \leq \frac{2}{3}\gamma -1$ and $\theta \leq 1$.

\begin{align*}
|\j_{7}|&=\Big|\int_{0}^{T}\psi \int_{\mathcal{O}}\rho_{\dl}u_{\dl}\cdot \mathcal B\big[\mathrm{div}(\rho_{\dl}^{\theta}u_{\dl})\big]\,\md x \md t\Big|
 \leq \int_{0}^{T}\|\rho_{\dl}\|_{L^{\gamma}_{x}}\|u_{\dl}\|_{L^{6}_{x}}
   \big\| \mathcal B [\mathrm{div} (\rho_{\dl}^{\theta}u_{\dl})]\big\|_{L^{\frac{6\gamma}{5\gamma -6}}_{x}}\md t\\
&\leq \int_{0}^{T}\|u_{\dl}\|_{L^{6}_{x}}\|\rho_{\dl}^{\theta}u_{\dl}\|_{L^{\frac{6\gamma}{5\gamma -6}}_{x}}\md t \leq  \int_{0}^{T}\|u_{\dl}\|_{L^{6}_{x}}^{2}\|\rho_{\dl}^{\theta}\|_{L^{\frac{3\gamma}{2\gamma -3}}_{x}}\md t \leq C \|u_{\dl}\|_{L^{2}_{t}H^{1}_{x}}^{2}\|\rho_{\dl}\|_{L^{\infty}_{t}L^{\gamma}_{x}}^{\theta}\\
&\leq C \qquad\mbox{if $\theta \leq \frac{2\gamma-3}{3}$}.
\end{align*}

Similarly to \eqref{B-l-infty}, we have
\begin{align*}
&\big\|B[\rho_{\dl}^{\theta}-\avint_{\mathcal{O}}\rho_{\dl}^{\theta}\,\md y]\big\|_{L^{\infty}_{x}}\\
&\leq C_{1}\big\|\na B\big[\rho_{\dl}^{\theta}-\avint_{\mathcal{O}}\rho_{\dl}^{\theta}\,\md y\big]\big\|_{L^{\gamma}_{x}}^{\frac{3}{\gamma}}
  \big\|B\big[\rho_{\dl}^{\theta}-\avint_{\mathcal{O}}\rho_{\dl}^{\theta}\,\md y\big]\big\|_{L^{\gamma}_{x}}^{1-\frac{3}{\gamma}}
  +C_{2}\big\|B\big[\rho_{\dl}^{\theta}-\avint_{\mathcal{O}}\rho_{\dl}^{\theta}\,\md y\big]\big\|_{L^{\gamma}_{x}}\\
&\leq C\|\rho_{\dl}^{\theta}\|_{L^{\gamma}_{x}}.
\end{align*}
Then we have
\begin{align*}
&|\j_{8}|=\Big|\int_{0}^{T}\psi \int_{\mathcal{O}}\big(\na Q_{\dl}\otimes \na Q_{\dl}-\mathrm{F}(Q_{\dl})\mathrm{I}_{3}\big):\na \mathcal B\big[\rho_{\dl}^{\theta}
   -\avint_{\mathcal{O}}\rho_{\dl}^{\theta}\md y\big]\,\md x \md t\Big|\\
&\leq C \int_{0}^{T}\Big(\|\na Q_{\dl}\|_{L^{\frac{10}{3}}_{x}}^{2}
   \big\|\na \mathcal B\big[\rho_{\dl}^{\theta}-\avint_{\mathcal{O}}\rho_{\dl}^{\theta}\md y\big]\big\|_{L^{\frac{5}{2}}_{x}}
     +\big| \int_{\mathcal{O}}\mathrm{F}(Q_{\dl}) \mathrm{div} \mathcal B\big[\rho_{\dl}^{\theta}-\avint_{\mathcal{O}}\rho_{\dl}^{\theta}\md y\big]\md x\big|\Big)\md t\\
&\leq C \int_{0}^{T}\Big(\|\na Q_{\dl}\|_{L^{\frac{10}{3}}_{x}}^{2}\big\|\na \mathcal B\big[\rho_{\dl}^{\theta}-\avint_{\mathcal{O}}\rho_{\dl}^{\theta}\md y\big]\big\|_{L^{\frac{5}{2}}_{x}}
 +\big|\int_{\mathcal{O}}\mathrm{F}(Q_{\dl})(\rho_{\dl}^{\theta}-\avint_{\mathcal{O}}\rho_{\dl}^{\theta}\md y)\md x\big|\Big)\md t\\
&\leq  C\int_{0}^{T}\Big(\|\na Q_{\dl}\|_{L^{\frac{10}{3}}_{x}}^{2}\big\|\rho_{\dl}^{\theta}-\avint_{\mathcal{O}}\rho_{\dl}^{\theta}\md y\big\|_{L^{\frac{5}{2}}_{x}}+ \big(\|Q_{\dl}\|_{L^{5}_{x}}^{2}+\|Q_{\dl}\|_{L^{10}_{x}}^{4}\big)\big\|\rho_{\dl}^{\theta}-\avint_{\mathcal{O}}\rho_{\dl}^{\theta}\md y\big\|_{L^{\frac{5}{3}}_{x}}\Big)\md t\\
&\leq C \qquad\mbox{if $\theta \leq \frac{2}{5}$}.
\end{align*}

\begin{align*}
|\i_{9}|&=\sigma_{*}\Big|\int_{0}^{T}\psi \int_{\mathcal{O}}c_{\dl}^{2}Q_{\dl}\cdot \na \mathcal B\big[\rho_{\dl}^{\theta}-\avint_{\mathcal{O}}\rho_{\dl}^{\theta}\md y\big]\md x \md t\Big|\\
&\leq C\|c_{\dl}\|_{L^{\infty}_{t, x}}^{2}\int_{0}^{T}\|Q_{\dl}\|_{L^{2}_{x}}\big\|\na \mathcal B\big[\rho_{\dl}^{\theta}-\avint_{\mathcal{O}}\rho_{\dl}^{\theta}\md y\big]\big\|_{L^{2}_{x}}\md t \\
&\leq C\int_{0}^{T}\|Q_{\dl}\|_{L^{2}_{x}}\|\rho_{\dl}^{\theta}\|_{L^{2}_{x}}\md t \leq C \qquad\mbox{if $\theta \leq \frac{\gamma}{2}$}.
\end{align*}
\begin{align*}
|\i_{10}|&=\Big|\int_{0}^{T}\psi \int_{\mathcal{O}}\na \cdot \big(Q_{\dl}\D Q_{\dl}-\D Q_{\dl}Q_{\dl}\big)\cdot \mathcal B \big[\rho_{\dl}^{\theta}
  -\avint_{\mathcal{O}}\rho_{\dl}^{\theta}\,\md y\big]\,\md x \md t\Big|\\
&\leq\int_{0}^{T}\|Q_{\dl}\|_{L^{4}_{x}}\|\D Q_{\dl}\|_{L^{2}_{x}}\big\|\na \mathcal B\big[\rho_{\dl}^{\theta}-\avint_{\mathcal{O}}\rho_{\dl}^{\theta}\,\md y\big]\big\|_{L^{4}_{x}}\md t\\
&\leq\int_{0}^{T}\|Q_{\dl}\|_{L^{4}_{x}}\|\D Q_{\dl}\|_{L^{2}_{x}}\|\rho_{\dl}^{\theta}\|_{L^{4}_{x}}\md t\leq C \qquad \mbox{if $\theta \leq \frac{1}{4}$}.
\end{align*}

Combining the above estimates, we obtain the desired result.
\end{proof}

\subsection{The limit passage and the effective viscous flux}
We infer from the uniform estimates (\ref{c-dl-gamma})$-$(\ref{na-Q-dl}) in Proposition \ref{existence-dl}
and Lemma \ref{uniform-rho-dl} that as $\dl \to 0$,
\begin{align*}
&c_{\dl}\rightharpoonup c \qquad\quad\,\,\,\,\, \mbox{ in }L^{2}(0, T; H^{1}(\mathcal{O})),\\
&\rho_{\dl}\to \rho \qquad\quad\,\,\,\,  \mbox{ in } C([0, T]; L^{\gamma}_{\rm weak}(\mathcal{O})),\\
&u_{\dl}\rightharpoonup u \qquad\quad\,\,\,  \mbox{ in } L^{2}(0, T; H_{0}^{1}(\mathcal{O})),\\
&\rho_{\dl}u_{\dl}\to \rho u \qquad \mbox{ in } C([0, T]; L_{\rm weak}^{\frac{2\gamma}{\gamma +1}}(\mathcal{O})),\\
&Q_{\dl}\rightharpoonup Q \qquad\quad \mbox{ in } L^{2}(0, T; H^{2}(\mathcal{O})),\\
&Q_{\dl}\to Q \qquad\quad \mbox{ in } L^{2}(0, T; H^{1}(\mathcal{O})),\\
&\rho_{\dl}^{\gamma}\to \overline{\rho^{\gamma}} \qquad\quad \mbox{ in } L^{\frac{\gamma +\theta}{\gamma}}(\mathcal{O}_{T}).
\end{align*}
Moreover, we  have
\begin{align*}
&\rho_{\dl}u_{\dl}\otimes u_{\dl}\to \rho u\otimes u \qquad \mbox{ in } \mathcal D^{'} (\mathcal{O}_{T}),\\
&\mathrm{F}(Q_{\dl})\mathrm{I}_{3}-\na Q_{\dl}\odot \na Q_{\dl}+(Q_{\dl}\D Q_{\dl}-\D Q_{\dl}Q_{\dl})
+\sigma_{*} c_{\dl}^{2}Q_{\dl} \\
&\qquad \to \mathrm{F}(Q)\mathrm{I}_{3}-\na Q\odot \na Q+(Q\D Q-\D QQ)
+\sigma_{*} c^{2}Q,\\
& \dl \rho_{\dl}^{\b}\to 0 \qquad \mbox{ in } L^{1}(\mathcal{O}_{T}).
\end{align*}
Then
the limit functions $(c, \rho, u, Q)$ satisfy
\begin{align}
& c_{t}+u\cdot \na c=D_{0}\D c, \label{c-dl-2}\\
& \del_{t}\rho+\na \cdot (\rho u)=0, \label{rho-dl-2}\\
& (\rho u)_{t}+\na \cdot (\rho u\otimes u)+\na \overline{\rho^{\gamma}}
   =\mu \D u+(\nu +\mu)\na \mathrm{div}\, u+\nabla \cdot (Q\D Q-\D Q Q)\\
&\qquad\qquad\qquad\qquad\qquad\qquad\quad +\na \cdot \big(\mathrm{F}(Q)\mathrm{I}_{3}-\na Q\odot \na Q\big)+\sigma_{*} \na \cdot (c^{2}Q),  \label{u-dl-2}\\
&\partial_{t}Q+(u\cdot \nabla )Q+Q\Omega-\Omega Q
=\Gamma H[Q, c] \label{Q-dl-2}
\end{align}
in $\mathcal D^{'} (\mathcal{O}_{T})$, with the initial-boundary conditions (\ref{I-C})$-$(\ref{compat-condition}), due to (\ref{rho-0-dl-rho-0}) and (\ref{q-dl-q}).

\medskip
Next, in order to prove that $(c, \rho, u, Q)$ is a weak solution of \eqref{c}$-$\eqref{compat-condition},
we only need to
show that $\overline{\rho^{\gamma}}=\rho^{\gamma}$ {\it a.e.} in $\mathcal{O}_{T}$,
or equivalently, the strong convergence of $\rho_{\dl}$ in $L^{1}(\mathcal{O})$.
As in \S 4, we need to show that $(\rho, u)$ is a renormalized solution of \eqref{rho-dl-2}.
From Lemma \ref{uniform-rho-dl}, the best estimate is $\rho \in L^{\gamma +\frac{2}{3}\gamma -1}(\mathcal{O}_{T})$.
Then $\gamma >\frac{3}{2}$ is not enough to guarantee that $\rho$ is square integrable,  and that ($\rho, u$) is a renormalized solution by Lemma \ref{suff-cond-renorm-solu} in Appendix A.
In order to deal with this difficulty, 
we introduce the cut-off function $T_{k}(z)=kT(\frac{z}{k})$ for $z\in \R$, $k=1, 2, \cdots$,
where $T$ is a smooth and concave function satisfying
\begin{align*}
T(z)=
\begin{cases}
z \qquad \mbox{ if } z\leq 1, \\
2 \qquad \mbox{ if } z\geq 3.
\end{cases}
\end{align*}
Since $(\rho_{\dl}, u_{\dl})$ is a renormalized solution of (\ref{rho-dl-1}),
taking $g(z)=T_{k}(z)$, we obtain
\be\label{renorm-rho-dl-trunc}
\del_{t}\big(T_{k}(\rho_{\dl})\big)+\mathrm{div} \big(T_{k}(\rho_{\dl})u_{\dl}\big)+\big(T'_{k}(\rho_{\dl})-T_{k}(\rho_{\dl})\big)\mathrm{div}\, u_{\dl}=0
\qquad \mbox{in $\mathcal D^{'} ((0, T)\times \R^{3})$},
\en
where  $T_{k}(\rho_{\dl})\in L^{\infty}(\mathcal{O}_{T})$ for any fixed $k$.
Then
\begin{align*}
T_{k}(\rho_{\dl})\overset{*}\rightharpoonup\, \overline{T_{k}(\rho)}\qquad\,\, \mbox{in} \,\,\, L^{\infty}(\mathcal{O}_{T}).
\end{align*}
Moreover, since $\del_{t}T_{k}(\rho_{\dl})$ satisfies (\ref{renorm-rho-dl-trunc}),
as before, we have
\begin{align*}
T_{k}(\rho_{\dl})\to \overline{T_{k}(\rho)} \qquad \mbox{ in } C([0, T]; L^{p}_{\rm weak}(\mathcal{O})), \quad \forall\, 1\leq p <\infty.
\end{align*}
Letting $\dl \to 0$, it yields
\begin{align*}
\del_{t}\overline{T_{k}(\rho)}+\mathrm{div} \big(\overline{T_{k}(\rho)}u\big)
 +\overline{(T'_{k}(\rho)-T_{k}(\rho))\mathrm{div} u}=0
\qquad\, \mbox{in}\,\,\, \mathcal{D}^{'}(\mathcal{O}_{T}),
\end{align*}
where
\begin{align*}
\big(T'_{k}(\rho_{\dl})-T_{k}(\rho_{\dl})\big)\mathrm{div} u_{\dl}\rightharpoonup \overline{\big(T'_{k}(\rho)-T_{k}(\rho)\big)\mathrm{div} u}
\qquad  \mbox{ in } L^{2}(\mathcal{O}_{T}).
\end{align*}

\subsection{The effective viscous flux}
Similarly to Lemma \ref{2}, we define the effective viscous flux
as $\tilde{\mathfrak{E}}_{\dl}:=\rho_{\dl}^{\gamma}-(\nu+2\mu)\mathrm{div}\, u_{\dl}$,
and its correspondingly weak convergence limit $\tilde{\mathfrak{E}}:=\overline{\rho^{\gamma}}-(\nu+2\mu)\mathrm{div}\, u$.
Then we have

\begin{Lemma}\label{e-v-f-dl}
Assume $(\rho_{\dl}, u_{\dl})$ is a family of the approximate solutions constructed in Proposition {\rm \ref{existence-dl}}.
Then
$$
\lim_{\dl \to 0^{+}}\int_{0}^{T}\psi \Big(\int_{\mathcal{O}} \phi \tilde{\mathfrak{E}}_{\dl}T_{k}(\rho_{\dl})\,\md x\Big) \md t
=\int_{0}^{T}\psi \Big(\int_{\mathcal{O}} \phi \tilde{\mathfrak{E}}\overline{T_{k}(\rho)} \,\md x\Big) \md t
$$
for any $\psi \in \mathcal D (0, T)$ and $\phi \in \mathcal D (\mathcal{O})$.
\end{Lemma}

\subsection{Renormalized solutions}
The following lemma implies that $T_{k}(\rho)-\overline{T_{k}(\rho)} \in L^{2}(\mathcal{O}_{T})$,
which
helps us establish that the limit function $(\rho, u)$ is a renormalized solution.

\begin{Lemma}[The amplitude of oscillations]\label{oscillation}
There exists a constant $C$, independent of $k$, such that
\be
\limsup_{\dl\to 0}\|T_{k}(\rho_{\dl})-T_{k}(\rho)\|_{L^{\gamma +1}(\mathcal{O}_{T})}\leq C.
\en
\end{Lemma}

The proof of this lemma is the same as that for Lemma 4.3 of  Subsection 4.4 in \cite{F-N-P-2001}.

\begin{Remark} By the concavity of the norm, we know from Lemma \ref{oscillation} that
\be
\|\overline{T_{k}(\rho)}-T_{k}(\rho)\|_{L^{\gamma +1}(\mathcal{O}_{T})}\leq C.
\en
From the proof of Lemma \ref{oscillation}, we have
\be
0\leq \limsup_{\dl\to 0}\|T_{k}(\rho_{\dl})-T_{k}(\rho)\|^{\gamma +1}_{L^{\gamma +1}(\mathcal{O}_{T})}\leq \lim_{\dl \to 0}\int_{0}^{T}\int_{\mathcal{O}_{T}}\big(\rho_{\dl}^{\gamma}T_{k}(\rho_{\dl})-\overline{\rho^{\gamma}}\,\overline{T_{k}(\rho)}\big)\md x \md t.
\en
\end{Remark}

Based on the uniform estimate of the amplitude of oscillations in Lemma \ref{oscillation}, we see that
the limit function $(\rho, u)$ satisfies (\ref{rho-dl-2}) in the renormalized sense.

\begin{Lemma}\label{renorm-solu-rho-u}
The limit function $(\rho, u)$ is a renormalized solution to \eqref{rho-dl-2}{\rm ;} that is,
\be\label{renorm-rho-final}
\del_{t}g(\rho)+\mathrm{div}(g(\rho)u)+\big(g'(\rho)\rho -g(\rho)\big)\mathrm{div}\, u=0 \qquad \mbox{in} \,\,\, \mathcal{D}^{'}((0, T)\times \R^{3})
\en
for any $g\in C^{1}(\R)$ with the property $g'(z)\equiv 0$ when $z\geq M$ for sufficiently large constant $M$, provided $(\rho, u)$ are prolonged zero outside $\mathcal{O}$.
\end{Lemma}

The detailed proof can be found in Subsection 4.5 of \cite{F-N-P-2001} for Lemma 4.4 there.

\subsection{Strong convergence of density $\rho_{\dl}$}
We now give an outline of the proof for the strong convergence of density $\rho_{\dl}$, {\it i.e.}, $\overline{\rho^{\gamma}}=\rho^{\gamma}$,
where $\rho_{\dl}^{\gamma}\to \overline{\rho^{\gamma}}$ in  $L^{\frac{\gamma +\theta}{\gamma}}(\mathcal{O}_{T})$
for the completeness of the proof of Theorem \ref{main-thm}.

Introduce a family of function in $C^{1}(\R^{+})\cap C[0, \infty)$:
\begin{align*}
L_{k}(z)=\begin{cases}
z \log z \quad &\mbox{for $0\leq z<k$},\\
z\log k +z\int_{k}^{z}\frac{T_{k}(s)}{s^{2}}\md s &\mbox{for $z\geq k$}.
\end{cases}
\end{align*}
By the construction of $L_{k}$, we know that $L_{k}$ is a linear function for large $z$. In particular, we see that, for $z\geq 3k$,
\begin{equation*}
L_{k}(z)=\b_{k}z-2k
\end{equation*}
with
\begin{equation*}
\b_{k}=\log k +\int_{k}^{3k}\frac{T_{k}(s)}{s^{2}}\md s +\frac{2}{3}.
\end{equation*}
Then, if $g_{k}(z):=L_{k}(z)-\b_{k}z$, we obtain that
$g_{k}(z)\in C^{1}(\R^{+})\cap C[0, \infty)$, $g'_{k}(z)=0$ for $z$ is sufficiently large, and
\begin{equation*}
g_{k}'(z)z-g_{k}(z)=T_{k}(z).
\end{equation*}

By Proposition \ref{existence-dl} and Lemma \ref{renorm-solu-rho-u}, we know that $(\rho_{\dl}, u_{\dl})$ and $(\rho, u)$
are renormalized solutions to (\ref{rho-dl-2}). Then we substitute function $g$ by $g_{k}$ in the definition of renormalized solutions
and
take the difference of these two equations to obtain
\be\label{difference-renorm-solu}
\begin{split}
\del_{t}\big(L_{k}(\rho_{\dl})-L_{k}(\rho)\big)
+\mathrm{div}\big(L_{k}(\rho_{\dl})u_{\dl}-L_{k}(\rho)u\big)
+T_{k}(\rho_{\dl})\mathrm{div} u_{\dl}-T_{k}(\rho)\mathrm{div}\, u=0
\end{split}
\en
in $\mathcal{D}^{'}((0, T)\times \R^{3})$.

Since $L_{k}$ is linear when $z$ is large,
$L_{k}(\rho_{\dl})$ is uniformly bounded with respect to $\dl$ in $L^{\infty}(0, T; L^{\gamma}(\mathcal{O}))$
so that
\begin{equation*}
\begin{split}
L_{k}(\rho_{\dl})\overset{*}\rightharpoonup\, \overline{L_{k}(\rho)} \qquad \mbox{in } L^{\infty}_{t}L^{\gamma}_{x} \,\,\, \mbox{as } \dl \to 0.
\end{split}
\end{equation*}
Moreover, since $L_{k}(\rho_{\dl})$ is a renormalized solution, similarly as before, we have
\be\label{strong-conv-Lk-tho-dl}
\begin{split}
L_{k}(\rho_{\dl})\to \overline{L_{k}(\rho)} \qquad \mbox{in } C([0, T], L^{\gamma}_{\rm weak}(\mathcal{O}))\cap C([0, T]; H^{-1}(\mathcal{O})) \,\,\,\, \mbox{as } \dl \to 0.
\end{split}
\en
Now, using function $\phi(x)\in \mathcal{D}(\mathcal{O})$ to test (\ref{difference-renorm-solu}) and then integrate over $(0, t)$, we have
\begin{align*}
&\int_{\mathcal{O}}\big(L_{k}(\rho_{\dl})-L_{k}(\rho)\big)(t, x)\phi (x)\, \md x\\
&=\,\int_{\mathcal{O}}\big(L_{k}(\rho_{0, \dl})-L_{k}(\rho_{0})\big)\phi\, \md x
  +\int_{0}^{t}\int_{\mathcal{O}}\big(L_{k}(\rho_{\dl})u_{\dl}-L_{k}(\rho)u\big)\cdot \na \phi\, \md x \md s\\
 &\quad -\int_{0}^{t}\int_{\mathcal{O}}\big(T_{k}(\rho_{\dl})\mathrm{div} u_{\dl}-T_{k}(\rho)\mathrm{div}\, u\big)\phi\, \md x \md s.
\end{align*}
Sending $\dl \to 0$, we have
\begin{equation*}
\begin{split}
&\int_{\mathcal{O}}\big(\overline{L_{k}(\rho)}-L_{k}(\rho)\big)(t, x)\phi (x)\, \md x\\
&=
\int_{0}^{t}\int_{\mathcal{O}}\big(\overline{L_{k}(\rho)}-L_{k}(\rho)\big)u \cdot \na \phi\, \md x \md s
-\lim_{\dl \to 0}\int_{0}^{t}\int_{\mathcal{O}}\big(T_{k}(\rho_{\dl})\mathrm{div} u_{\dl}-T_{k}(\rho)\mathrm{div}\, u\big)\phi\, \md x \md s.
\end{split}
\end{equation*}
Taking $\phi =\phi_{m}$ with $\phi_{m}(z)\to 1_{\mathcal{O}}(z)$ in the above equation and sending  $m\to \infty$, we have
\be\label{overline-Lk-tho-dl-conv}
\int_{\mathcal{O}}\big(\overline{L_{k}(\rho)}-L_{k}(\rho)\big)(t)\, \md x
=\int_{0}^{t}\int_{\mathcal{O}}T_{k}(\rho)\,\mathrm{div}\, u\,\md x \md s
-\lim_{\dl \to 0}\int_{0}^{t}\int_{\mathcal{O}}T_{k}(\rho_{\dl})\mathrm{div} u_{\dl}\, \md x \md s.
\en
Then, using Lemmas \ref{e-v-f-dl}--\ref{oscillation}, we find that the right-hand side of the above equation
is non-positive, which yields
\be\label{overline-Lk-rho-dl-limit}
\lim_{k\to \infty}\int_{\mathcal{O}}\big(\overline{L_{k}(\rho)}-L_{k}(\rho)\big)(t)\, \md x
\leq 0 \qquad \mbox{for $t \in [0, T]$}.
\en
Moreover, by the definition of $L_{k}$ and the absolution continuity of $\rho\,\log \rho \in L^{1}(\mathcal{O}_{T})$, we have
\be
\begin{split}
\|L_{k}(\rho)-\rho \log \rho\|_{L^{1}(\mathcal{O}_{T})}&\leq \int \int_{\{\rho \geq k\}}|L_{k}(\rho)-\rho \log \rho|\,\md x \md t\\
&\leq C\int \int_{\{\rho \geq k\}}|\rho \log \rho|\,\md x \md t \to 0 \qquad \mbox{as } k\to \infty.
\end{split}
\en
Similarly, we have
\begin{align*}
\|L_{k}(\rho_{\dl})-\rho_{\dl}\log \rho_{\dl}\|_{L^{1}(\mathcal{O}_{T})}&\leq \int \int_{\{\rho_{\dl} \geq k\}}|L_{k}(\rho_{\dl})-\rho_{\dl} \log \rho_{\dl}|\,\md x \md t\\
&\leq \int \int_{\{\rho_{\dl} \geq k\}}\frac{\log k +\int_{k}^{\rho_{\dl}}\frac{T_{k}(s)}{s^{2}}\md s+\log \rho_{\dl}}{\rho_{\dl}^{\gamma-1}}\rho_{\dl}^{\gamma}\,\md x \md t\\
&\leq C(\ve)k^{1+\ve-\gamma}\int \int_{\{\rho_{\dl} \geq k\}}\rho_{\dl}^{\gamma}\,\md x \md t \to 0 \qquad \mbox{as } k\to \infty.
\end{align*}
This, together with the lower-semicontinuity of the norm, we have
\be\label{overline-Lk-rho-l1-limit}
\|\overline{L_{k}(\rho)}-\overline{\rho \log \rho }\|_{L^{1}(\mathcal{O}_{T})}\leq \liminf_{\dl \to 0}\|L_{k}(\rho_{\dl})-\rho_{\dl}\log \rho_{\dl}\|_{L^{1}(\mathcal{O}_{T})} \to 0, \quad \mbox{as }k\to \infty.
\en
Finally, combining (\ref{overline-Lk-rho-dl-limit})$-$(\ref{overline-Lk-rho-l1-limit}), we have
\begin{equation*}
\begin{split}
\int_{\mathcal{O}}\big(\overline{\rho \log \rho}-\rho \log \rho\big)(t)\md x \leq 0.
\end{split}
\end{equation*}
Moreover, since $\overline{\rho \log \rho}\geq \rho \log \rho$, we see that $\overline{\rho \log \rho}=\rho \log \rho$ for {\it a.e.} $(t,x)\in\mathcal{O}_{T}$.
In addition, by the restrict concavity of function $z \log z$, we have
\begin{equation*}
\rho_{\dl}\to \rho \qquad\,\, \mbox{in }L^{p}(\mathcal{O}_{T}) \,\,\, \mbox{for any $p\in [1, \gamma +\theta)$}.
\end{equation*}
Then we conclude
\begin{equation*}
\overline{\rho^{\gamma}}=\rho^{\gamma}, \quad \mbox{\it a.e.}
\end{equation*}
Therefore, we complete the proof of Theorem \ref{main-thm}.

\bigskip

\appendix

\section{Preliminaries}\label{appendix-a}

In this appendix, we collect some important theories and lemmas
that we use extensively in this paper.

In order to deal with the highest derivatives of $u$ in (\ref{velocity})
(and the corresponding ones in the approximation systems) and the highest derivatives of $Q$ in (\ref{Q})
(and the corresponding ones in the approximation systems), we need the following lemma whose proof can be found
in the proof of Lemma A.1 in \cite{C-M-W-Z}.

\begin{Lemma}\label{estimate-matrix}
Let $Q$ and $Q'$ be two $3\times 3$ symmetric matrices,
and let $\Omega =\frac{1}{2}\big(\na u-(\na u)^{\top}\big)$ be the vorticity with $(\na u)_{\a \b}=\del_{\b}u_{\a}$.
Then
\begin{equation*}
(\O Q'-Q' \O, \D Q)-(\na \cdot (Q' \D Q-\D Q Q'), u)=0.
\end{equation*}
\end{Lemma}

\begin{Lemma}[Aubin-Lions lemma \cite{A-1963}]\label{al-lemma}
Let $X_{0}, X$, and $X_{1}$ be three Banach spaces with $X_{0}\subseteq X \subseteq X_{1}$, $X_{0}$ compactly embedded in $X$, and $X$ continuously embedded in $X_{1}$.
For $1\leq p, q \leq \infty$, let
\begin{equation*}
W=\{ u\in L^p (0, T; X_{0}) \;:\; \dot{u}\in L^q (0, T; X_{1})\}.
\end{equation*}
Then
\begin{enumerate}
\item[(i)] If $p<\infty$, then the embedding of $W$ into $L^p (0, T; X)$ is compact{\rm ;}
\item[(ii)]  If $p=\infty$ and $q>1$, then the embedding of $W$ into $C(0, T; X)$ is compact.
\end{enumerate}
\end{Lemma}

\begin{Lemma}[Gagliardo-Nirenberg interpolation inequality \cite{N-1955}]\label{gn-inequality}
Let $1\leq q, r \leq \infty$ and  $0\leq j<m$.
Then the following inequalities hold{\rm :}
$$
\|D^{j}u\|_{L^p}\leq C_{1} \|D^{m}u\|_{L^r}^{a}\|u\|_{L^q}^{1-a}+C_{2}\|u\|_{L^{s}}
$$
for any function $u: \mathcal{O}\to \R$ defined on a bounded Lipschitz
domain $\mathcal{O}\subseteq \R^{3}$,
where
$$
\frac{1}{p}=\frac{j}{3}+a(\frac{1}{r}-\frac{m}{3})+(1-a)\frac{1}{q},
\qquad \,\,\frac{j}{m}\leq a \leq 1,
$$
$s>0$ is arbitrary, and $C_{1}$ and $C_{2}$ depend only on $\mathcal{O}$ and $m$.
\end{Lemma}

Next, we introduce a sufficient condition for a solution $(\rho, u)$ to be a renormalized solution.

\begin{Lemma}\label{suff-cond-renorm-solu}
Let $\mathcal{O}\subseteq \R^{3}$ be a bounded domain, and let $\rho \in L^{2}(\mathcal{O}_{T})$ and $u\in L^{2}(0, T; H^{1}(\mathcal{O}))$ such that
$$
\del_{t}\rho +\na \cdot (\rho u)=0 \qquad \mbox{in}\,\, \mathcal{D}^{'}(\mathcal{O}_{T}).
$$
Then
\be\label{renormalize}
\del_{t}g(\rho)+\mathrm{div}(g(\rho)u)+\big(g'(\rho)\rho -g(\rho)\big)\mathrm{div}\, u=0 \qquad \mbox{in} \,\, \mathcal{D}^{'}(\mathcal{O}_{T})
\en
for any $g\in C^{1}(\R)$ with the property: $g'(z)\equiv 0$, when $z\geq M$ for sufficiently large constant $M$, i.e., $(\rho, u)$ is a renormalized solution.
\end{Lemma}

\begin{Definition}
The metric space $C([0, T]; X^{*}_{\rm weak})$ contains all the functions $v: [0, T]\mapsto X^{*}$ which are continuous with respect to the weak topology.
We say
$$
v_{n} \to v \qquad \mbox{in }C([0, T]; X^{*}_{\rm weak}),
$$
if $(v_{n}(t), \phi)\to (v(t), \phi)$ uniformly with respect to $t\in [0, T]$ for any $\phi \in X$.
\end{Definition}

In the following corollary, we introduce a sufficient condition for a family of functions
to converge in $C([0, T]; X^{*}_{\rm weak})$ (see Corollary 2.1 in \cite{F-2004}).

\begin{Corollary}\label{suff-cond-C-Weak-star}
Let $X$ be a separable Banach space. Assume that  $v_{n}: [0, T] \to X^{*}$, $n=1, 2, \cdot \cdot \cdot$,
is a sequence of measurable functions such that
$$
\mathrm{ess}\sup_{t\in [0, T]}\|v_{n}(t)\|_{X^{*}}\leq M,\qquad \mbox{uniformly in }n=1, 2, \cdot \cdot \cdot.
$$
Moreover, let the family of functions{\rm :}
$$
(v_{n}, \phi): t\mapsto (v_{n}(t), \phi), \qquad t\in [0, T], n=1, 2, \cdot \cdot \cdot,
$$
be equi-continuous for any fixed $\phi$ belonging to a dense subset in $X$.

Then $v_{n}\in C([0, T]; X^{*}_{\rm weak})$ for any $n=1, 2, \cdot \cdot \cdot$,
and there exist $v\in C([0, T]; X^{*}_{\rm weak})$ and a subsequence (still denoted)
$v_{n}$ such that
$$
v_{n}\to v \qquad \mbox{in }C([0, T]; X^{*}_{\rm weak}) \quad \mbox{as }n\to \infty.
$$
\end{Corollary}

\begin{Proposition}\label{conv-h-1}
Let $\mathcal{O}$ be a bounded domain in $\R^{3}$.
If $v_{n}\in L^{\infty}(0, T; L^{p}(\mathcal{O}))$ for $p>\frac{6}{5}$, and $v_{n}\to v$ in $C([0, T], L^{p}_{\rm weak}(\mathcal{O}))$,
then $v_{n}\to v$ in $C([0, T], H^{-1}(\mathcal{O}))$.
\end{Proposition}
\bigskip

\section*{Acknowledgments}
The research of G.-Q. Chen was supported in part by the UK
Engineering and Physical Sciences Research Council Award EP/L015811/1
and the Royal Society--Wolfson Research Merit Award (UK).
The research of A.  Majumdar was supported by an EPSRC Career Acceleration Fellowship EP/J001686/1 and EP/J001686/2,
an OCIAM Visiting Fellowship and the Advanced Studies Centre at Keble College.
The research of D. Wang was supported in part by the National Science Foundation under grants DMS-1312800 and DMS-1613213.
The research of R. Zhang was supported in part by the National Science Foundation under grant DMS-1613213 and DMS-1613375.

\end{document}